\documentclass[times,final,11pt]{elsarticle}

\usepackage{jcomp}

\usepackage{framed,multirow}

\usepackage{amssymb,amsmath,amsthm,bm}
\usepackage{latexsym}
\usepackage{mathtools}
\usepackage{hyperref}
\hypersetup{hypertex=true,
	colorlinks=true,
	linkcolor=blue,
	anchorcolor=blue,
	citecolor=blue}

\usepackage{cleveref}
\usepackage{subcaption}
\usepackage{graphicx}

\usepackage{booktabs}
\usepackage{makecell}
\usepackage{enumitem} 

\usepackage{url}
\usepackage{xcolor}

\usepackage{caption}
\usepackage{geometry}   
\definecolor{newcolor}{rgb}{.8,.349,.1}

\makeatletter

\newcommand{\Rmnum}[1]{\uppercase\expandafter{\romannumeral #1}}
\makeatother

\newtheorem{theorem}{Theorem}[section]

\newtheorem{example}{Example}[section]

\theoremstyle{definition}

\theoremstyle{remark}
\newtheorem{remark}{Remark}[]

\newcommand\bx{\bm{x}}

\newcommand\bU{\bm{U}}
\newcommand\bF{\bm{F}}

\newcommand\bbR{\mathbb{R}}
\newcommand\pd[2]{\dfrac{\partial {#1}}{\partial {#2}}}
\newcommand\mean[1]{\{\!\{ #1 \}\!\}}
\newcommand\abs[1]{\lvert #1 \rvert}

\def\jump#1{\llbracket #1 \rrbracket }

\allowdisplaybreaks

\renewcommand\baselinestretch{1.26}

\journal{Journal of Computational Physics}
\begin{document}
	
	\verso{Zhihao Zhang, Huazhong Tang, Kailiang Wu}

	\begin{frontmatter}
		
		\title{{\bf High-order accurate structure-preserving finite volume schemes\\ on adaptive moving meshes for shallow water equations: well-balancedness and positivity}\tnoteref{tnote1}}
		\author[1]{Zhihao {Zhang}}
		\ead{zhihaozhang@pku.edu.cn}
		\author[2,1]{Huazhong {Tang}}
		\ead{hztang@math.pku.edu.cn}
		\author[3,4]{Kailiang {Wu}\corref{cor1}}
		\cortext[cor1]{Corresponding author.}
		\ead{wukl@sustech.edu.cn}
		
		\address[1]{Center for Applied Physics and Technology, HEDPS and LMAM, School of Mathematical Sciences, Peking
			University, Beijing, 100871, P.R. China}
		\address[2]{Nanchang Hangkong University, Nanchang, 330000, Jiangxi Province, P.R. China}
		\address[3]{Department of Mathematics \& Shenzhen International Center for Mathematics, Southern University of Science and
			Technology, Shenzhen, Guangdong, 518055, P.R. China}
		\address[4]{Guangdong Provincial Key Laboratory of Computational Science and Material Design, Shenzhen 518055, China}
		

		\begin{abstract}
			This paper develops high-order accurate, well-balanced (WB), and positivity-preserving (PP) finite volume schemes for shallow water equations on adaptive moving structured meshes. The mesh movement poses new challenges in maintaining the WB property, which not only depends on the balance between flux gradients and source terms but is also affected by the mesh movement. To address these complexities, the WB property in curvilinear coordinates is decomposed into flux source balance and mesh movement balance. The flux source balance is achieved by suitable decomposition of the source terms, the numerical fluxes based on hydrostatic reconstruction, and appropriate discretization of the geometric conservation laws (GCLs). Concurrently, the mesh movement balance is maintained by integrating additional schemes to update the bottom topography during mesh adjustments. The proposed schemes are rigorously proven to maintain the WB property by using the discrete GCLs and these two balances. We provide rigorous analyses of the PP property under a sufficient condition enforced by a PP limiter. Due to the involvement of mesh metrics and movement, the analyses are nontrivial, while some standard techniques, such as splitting high-order schemes into convex combinations of formally first-order PP schemes, are not directly applicable. Various numerical examples validate the high-order accuracy, high efficiency, WB, and PP properties of the proposed schemes.

		\end{abstract}
		
		
		\begin{keyword}	
			\vspace{-8mm}
		\end{keyword}
		
	\end{frontmatter}
	
	
	\renewcommand\baselinestretch{1.39}

	
	\section{Introduction}\label{section:Intro}
	
	Shallow water equations (SWEs) describe the behavior of a shallow layer of flowing water influenced by gravity and bottom topology. They are widely applied in research related to atmospheric, river, and coastal dynamics, as well as in studies of phenomena such as tsunamis. The two-dimensional (2D) SWEs are defined within a time-space physical domain $(t,\bx)\in\bbR^+\times\Omega_p$, where $\Omega_p\subset \bbR^2$ and $\bx=(x_1,x_2)$. With non-flat, time-independent bottom topography $b(\bx)$, the 2D SWEs can be expressed as the following hyperbolic balance laws
	\begin{equation}\label{eq:SWE}
		\left\{
		\begin{aligned}
			&~\pd{h}{t}+\pd{\left(hv_1\right)}{x_1}+\pd{\left(hv_2\right)}{x_2}=0, \\
			&\pd{\left(hv_1\right)}{t}+\pd{\left(h v_1^2+\frac{1}{2} g h^2\right)}{x_1}+\pd{\left(h v_1 v_2\right)}{x_2}=-g h \pd{b}{x_1}, \\
			&\pd{\left(hv_2\right)}{t}+\pd{\left(h v_1 v_2\right)}{x_1}+\pd{\left(h v_2^2+\frac{1}{2} g h^2\right)}{x_2}=-g h \pd{b}{x_2},
		\end{aligned}
		\right.
	\end{equation}
	where $h(t,\bx) \geqslant 0$ represents the water depth, $v_\ell(t,\bx)$ denotes the $x_\ell$-component of the fluid velocity for $\ell=1,2$, and $g$ is the gravitational constant. Numerical simulation plays a crucial role in the applications of SWEs and has been extensively studied in the literature; see, for example, \cite{AuDusse2004fast,Bermudez1994Upwind,Noelle2007High,Tang2004Solution,Xing2017Numerical,Xing2005High,Zhao2022Well}
	and the references therein.
	
	Due to the presence of the source terms, the SWEs \eqref{eq:SWE} admit nontrivial steady states, where these source terms balance the flux gradients,
	e.g., the lake at rest solution
	\begin{equation}\label{eq:SWEs_WB}
		h+b = C,~v_1=v_2 = 0.
	\end{equation}
	Many physical phenomena, such as waves on a lake or tsunamis in the deep ocean, can be viewed as small perturbations of these steady states. Consequently, capturing these small perturbations in numerical simulations is very important. This need has led to the development of well-balanced (WB) schemes, which are designed to preserve steady states discretely, thus facilitating the capture of perturbations even on coarse meshes. Standard numerical methods often lack WB properties, necessitating the use of fine meshes to accurately capture perturbations, which can result in prohibitively high computational costs.  The WB or the ``C-property'' was initially demonstrated in \cite{Bermudez1994Upwind}. Since then, various types of WB numerical methods for the SWEs have been explored, including finite difference methods \cite{Li2020High,Vukovic2002ENO,Xing2005High}, finite volume methods \cite{AuDusse2004fast,Li2012Hybrid,Noelle2006Well,Noelle2007High}, and discontinuous Galerkin (DG) methods \cite{Xing2014Exactly,Xing2013Positivity,Zhang2021High}. The readers are referred to the review article \cite{Kurganov2018Finite} and the references therein.
	
	Another challenge in the numerical solution of the SWEs is to preserve the non-negativity of the water depth. In fact, negative water depths can destroy the essential hyperbolicity of the equations, leading to serious numerical issues due to  numerical ill-posedness, such as nonphysical solutions, instability, or simulation crashes. Over the past few decades, there has been considerable interest in developing high-order positivity-preserving (PP) or bound-preserving schemes \cite{WuShu2023}. A general approach was proposed in \cite{Zhang2010On, Zhang2010OnM} for finite volume and DG schemes for scalar conservation laws and compressible Euler equations. The key step in this approach is to establish sufficient conditions to maintain the cell averages of the numerical solutions as PP and then apply a simple scaling limiter to enforce these conditions. This approach has been further extended to the SWEs, incorporating additional treatments to preserve both the PP and WB properties \cite{Xing2010Positivity,Xing2011High}. More recently, a universal geometric quasilinearization (GQL) approach was proposed in \cite{WuShu2023} for bound-preserving problems with nonlinear constraints. 
	
	The resolution of localized structures in solutions to hyperbolic balance laws, such as shock waves and sharp transitions, often requires fine meshes. In these situations, adaptive moving mesh methods are highly effective in enhancing both the efficiency and quality of numerical solutions. These methods have played a crucial role in solving partial differential equations; see, for example, \cite{Brackbill1993An, Brackbill1982Adaptive, CAO1999221, CENICEROS2001609, Davis1982, Miller1981, Ren2000An, Stockie2001, Tang2003Adaptive, Wang2004A, Winslow1967Numerical}. 
	For SWEs specifically, an adaptive moving mesh kinetic flux-vector splitting method was introduced in \cite{Tang2004Solution}. In recent years, there has been growing interest in developing structure-preserving schemes on moving meshes. For instance, high-order adaptive moving mesh DG schemes were proposed in \cite{Zhang2021High, Zhang2022AWell}, ensuring the preservation of non-negativity in water height and the lake at rest condition. In \cite{Zhang2023Structure}, a WB and PP finite volume weighted essentially non-oscillatory (WENO) arbitrary Lagrangian--Eulerian scheme was developed. More recently, high-order WB energy-stable adaptive moving mesh schemes for single- and multi-layer SWEs were proposed in \cite{Zhang2023High, Zhang2023HighML}.
	
	This paper aims to contribute to the development of high-order structure-preserving finite volume schemes for SWEs on adaptive moving structured meshes. These schemes maintain the WB property and ensure the non-negativity of water depth and preserve the positivity of the determinant of the Jacobi matrix of the coordinate transformation. The key elements of the proposed schemes include appropriate discretization of GCLs, numerical fluxes derived from hydrostatic reconstruction, special discretization of source terms to balance the flux gradient, and WENO reconstructions carefully designed in the computational domain. The primary efforts of this study include:
	
	\begin{itemize}
		\item  We establish the WB property of our schemes in time-dependent curvilinear coordinates for the SWEs. The mesh movement poses new challenges, as it impacts the WB property. To address this, we achieve a balance through flux-source integration that incorporates surface conservation laws (SCLs) and mesh movement, aligned with volume conservation law (VCL) and bottom topography adjustments. The source terms are decomposed following the approach proposed in \cite{Xing2006High}, numerical fluxes are designed through hydrostatic reconstruction, and the source term fluxes are meticulously crafted to balance these numerical fluxes. GCLs are discretized using linear interpolation coupled with PP limiter  to also preserve the free-stream property. Additional schemes for updating the bottom topography during mesh movement are integrated with the original SWEs to maintain mesh movement balance. The schemes are proven to be WB using the discrete GCLs and these two balances.
		\item  We provide a rigorous analysis of the PP property for our proposed schemes. Due to the involvement of mesh metrics and movement, this analysis is complex, and standard techniques, such as splitting high-order schemes into convex combinations of formally first-order PP schemes, are not directly applicable. We prove that the proposed WB schemes are PP under a sufficient condition, enforced by a PP limiter. The PP property here refers to the non-negativity of the water depth and the positivity of the determinant of the Jacobi matrix of the coordinate transformation.
		\item We implement the proposed schemes on adaptive moving structured meshes. Mesh adaptation is achieved through iterative solutions of the Euler--Lagrange equations from a mesh adaptation functional, using a monitor function to focus mesh points around critical features. Various numerical experiments are conducted to validate the accuracy, robustness, and both WB and PP properties of the proposed schemes. 
	\end{itemize}
	
	This paper is organized as follows. Section \ref{Sec:Reformulate_SWEs} introduces the SWEs with a technical decomposition of the source terms and the formulations in both Cartesian and curvilinear coordinates, along with a discussion of the WB property in curvilinear coordinates. Sections  \ref{Sec:Numerical_Schemes} and \ref{Sec:2D_case} 
	present the high-order WB PP finite volume schemes on moving meshes for the one-dimensional (1D) and 2D SWEs, respectively. Section \ref{Sec:MM} details the adaptive moving mesh strategy.
	Section \ref{sec:Numerical_Test} conducts several numerical experiments to verify the high-order accuracy, structure-preserving properties, shock-capturing ability, and high efficiency of our schemes. Section \ref{Sec:Conclusion} presents the conclusions of this paper. 
	
	\section{Formulations of SWEs with technical decomposition of the source terms }\label{Sec:Reformulate_SWEs}
	This section introduces the SWEs with the decomposition of the source terms in both Cartesian and curvilinear coordinates.
	\subsection{Formulation of SWEs in Cartesian coordinates}
	The 2D SWEs \eqref{eq:SWE} can be cast into the following form
	\begin{equation}\label{eq:2D_fix_BalanceLaw}
		\bm{U}_{t}+\nabla\cdot\bm{F}(\bm{U}) = \bm{S}(\bm{U},\bm{x}),   
	\end{equation}
	where $\bm{F}\left(\bm{U}\right) = \left(\bm{F}_1\left(\bm{U}\right), \bm{F}_2\left(\bm{U}\right)\right)$, and 
	\begin{align*}
		&\bm{U} = \left(h,hv_1,hv_2\right)^{\top}, \quad 
		\bm{F}_1 = \left(hv_1,hv_1^2+\frac{1}{2}gh^2,hv_1v_2\right)^{{\top
		}},~
		\\
		&\bm{F}_2 = \left(hv_2,hv_1v_2,hv_2^2+\frac{1}{2}gh^2\right)^{{\top
		}}, \quad 
		\bm{S} = \left(0,-ghb_{x_1},-ghb_{x_2}\right)^{\top},
	\end{align*}
	represent the conservative variables, physical fluxes, and source terms, respectively.
	The steady-state solutions described in \eqref{eq:SWEs_WB} and the source terms for the  SWEs in \eqref{eq:2D_fix_BalanceLaw} can be expressed using the formulations proposed in \cite{Xing2006High}. Specifically, the steady-state solutions \eqref{eq:SWEs_WB} satisfy
	\begin{equation}\label{eq:Steady_state} 
		\bm{U}+\bm{w}(\bm{x}
		)
		= C\bm{q}(\bm{x}),        
	\end{equation}
	where $\bm{w}=(b,0,0)^{\top}$ and $\bm{q}=(1,0,0)^{\top}$ are time-independent.
	The source terms $\bm{S}(\bm{U},\bm{x}) $ in \eqref{eq:2D_fix_BalanceLaw} can be decomposed as  
	\begin{equation}\label{eq:SourceTerm_Decom}
		\bm{S}(\bm{U},\bm{x}) = \sum_{m = 1}^{2} {s}_{m}(\bm{x},t)  \left(
		(\bm{r}_{m1})_{x_1}(\bm{x})+(\bm{r}_{m2})_{x_2}(\bm{x})
		\right) = \sum_{m = 1}^{2} {s}_{m}(\bm{x},t) \nabla_{\bm{x}} \cdot(\bm{r}_{m}(\bm{x})),
	\end{equation}
	where $s_{m}(\bm{x},t)$ and $\bm{r}_m(\bm{x}) = (\bm{r}_{m1}(\bm{x}),\bm{r}_{m2}(\bm{x}))$ are defined as 
	\begin{align*}
		&s_1 = -g(h+b),\quad \bm{r}_{11} = (0,b,0)^{\top}, \quad \bm{r}_{12} = (0,0,b)^{\top},\\
		&s_2 = \frac{1}{2}g,\quad \bm{r}_{21} = (0,b^2,0)^{\top}, \quad \bm{r}_{22} = (0,0,b^2)^{\top}.
	\end{align*}
	In physics, the water depth should be nonnegative for the SWEs \eqref{eq:2D_fix_BalanceLaw}, meaning that the conservative vector $\bm{U}$ should belong to the following admissible state set  
	\begin{equation*}
		\mathcal{G}= \left\{\bm{U}=(h,hv_1,hv_2)^\top\in\mathbf{R}^{3}:~ h \geqslant0\right\}.
	\end{equation*}
	The PP property of a numerical scheme for SWEs refers to maintaining the numerical solutions within the set $\mathcal{G}$.
	
	\subsection{Formulation of SWEs in curvilinear coordinates}\label{sec:22}
	This subsection will give the formulation of the SWEs in curvilinear coordinates. Denote $\Omega_c$ as the computational domain, and
	assume that there is a time-dependent, differentiable coordinate transformation $t = \tau, \bx = \bx(\tau, \bm{\xi})$ that maps from the computational domain $\Omega_c$ to the physical domain $\Omega_p$. 
	An adaptive mesh can be generated from the reference mesh in $\Omega_c$ based on such a transformation, as illustrated in \cite{DUAN2021109949,Duan2022High}.
	The determinant of the Jacobian matrix is
	$J = \det\left(\frac{\partial (t,\bm{x})}{\partial (\tau,\bm{\xi})}\right).$
	The mesh metrics induced by this transformation satisfy the geometric conservation laws (GCLs), which include the volume conservation law (VCL) and the surface conservation laws (SCLs):
	\begin{equation}\label{eq:GCL_2D}
		\begin{aligned}
			&\text { VCL: } \quad \frac{\partial J}{\partial \tau}+ \sum_{\ell=1}^{2}\frac{\partial}{\partial \xi_\ell}\left(J \frac{\partial \xi_\ell}{\partial t}\right)=0,\\
			&\text { SCLs: } \quad \sum_{\ell = 1}^2\frac{\partial}{\partial \xi_\ell}\left(J \frac{\partial \xi_\ell}{\partial x_k}\right)=0,~k = 1,2.
		\end{aligned}
	\end{equation}
	Based on the GCLs and the decomposition \eqref{eq:SourceTerm_Decom},  the 1D version of \eqref{eq:2D_fix_BalanceLaw} in curvilinear coordinates can be written as 
	\begin{equation}\label{eq:1D_Cur_eq}
		\pd{J\bm{U}}{t}+\pd{\bm{\mathcal{F}}}{\xi} = \sum_{m}s_m\pd{\bm{r}_m}{\xi},
	\end{equation}
	with 
	\begin{equation*}
		\bm{\mathcal{F}} = J\pd{\xi}{t}\bm{U}+\bm{F}.
	\end{equation*}
	For the 2D case, \eqref{eq:2D_fix_BalanceLaw} can be transformed into
	\begin{equation}\label{eq:2D_cur_eq}
		\pd{J\bm{U}}{t}
		+\sum_{\ell = 1}^{2}
		\pd{\bm{\mathcal{F}}_{\ell}}{\xi_{\ell}} 
		= \sum_{m = 1}^{2}
		s_m
		\sum_{\ell=1}^{2}\pd{\bm{\mathcal{R}}_{m\ell}}{\xi_{\ell}},
	\end{equation}
	and 
	\begin{equation*}
		\bm{\mathcal{F}}_{\ell} = J\pd{\xi_{\ell}}{t}\bm{U}+\sum_{k=1}^{2}J\pd{\xi_{\ell}}{x_{k}} \bm{F}_{k} = J\pd{\xi_{\ell}}{t}\bm{U} + L_\ell \langle \bm{n}_\ell, \bm{F} \rangle = J\pd{\xi_{\ell}}{t}\bm{U} + L_\ell \bm{F}_{n_\ell},
	\end{equation*}
	\begin{equation*}
		\bm{\mathcal{R}}_{m\ell} = \sum_{k=1}^{2}J\pd{\xi_{\ell}}{x_{k}}\bm{r}_{mk} 
		= 
		L_\ell 
		\langle \bm{n}_\ell, \bm{r}_{m} \rangle 
		= L_\ell (\bm{r}_m)_{n_\ell} ,
	\end{equation*}
	where
	\begin{equation}\label{eq:L_N_Def}
		L_\ell = 
		\sqrt{\sum_{k=1}^d \left(J \frac{\partial \xi_\ell}{\partial x_k}\right)^2},\quad
		\bm{n}_\ell = \frac{1}{L_\ell}\left(J \frac{\partial \xi_\ell}{\partial x_1},~J \frac{\partial \xi_\ell}{\partial x_2}\right)^{\top}, \quad  \langle \bm{n}_\ell,  \bm{F} \rangle  = \left(
		\bm{n}_{\ell}
		\right)_1 \bm{F}_1
		+
		\left(
		\bm{n}_{\ell}
		\right)_2 \bm{F}_2.
	\end{equation}
	Note that $J$ should be positive to ensure valid interval length (1D) or cell area (2D) in the physical domain.  
	 	The PP property of our adaptive moving mesh schemes aims to maintain $J>0$ and ensure the numerical solutions within $\mathcal{G}$.
	Developing WB schemes in curvilinear coordinates for \eqref{eq:2D_cur_eq} is difficult, as the mesh movement also affects the WB property. Motivated by the previous work \cite{Zhang2023High,Zhang2023HighML}, we separate the WB property in the time-dependent curvilinear coordinates into two parts.
	\begin{itemize}
		\item \textbf{Flux Source Balance:}  This refers to the balance between the flux gradients and the source terms. Specifically, in curvilinear coordinates, the following equilibrium is required at the continuous case upon reaching a steady state: \begin{equation*}
			\pd{}{\xi_{\ell}}
			\left(
			\sum_{k=1}^{2}J\pd{\xi_{\ell}}{x_{k}} \bm{F}_k
			\right) 
			= \sum_{m = 1}^{2}s_m\pd{}{\xi_{\ell}}
			\left(
			\sum_{k=1}^{2}J\pd{\xi_{\ell}}{x_{k}}\bm{r}_{mk}
			\right). 
		\end{equation*}
		
		\item \textbf{Mesh Movement Balance:} This balance arises from mesh movement. Specifically, when the system reaches a steady state, the following equations associated with mesh movement must be satisfied:
		\begin{equation*}
			\pd{J\bm{U}}{t}
			+\sum_{\ell = 1}^{2}
			\pd{}{\xi_{\ell}}
			\left(
			J\pd{\xi_{\ell}}{t} \bm{U}
			\right) = 0.
		\end{equation*}
		Furthermore, for the still water case of the SWEs,  the conditions $h+b = C$ and $v_1 = v_2 = 0$ must be maintained, so that $\bm{U}(\bm{x},t) + \bm{w}(\bm{x}) = C\bm{q}(\bm{x})$ remains preserved. 
	\end{itemize}
	There are two approaches to computing $\bm{w}$ and $\bm{q}$:
	\begin{itemize}
		\item $\bm{w}$ and $\bm{q}$ are obtained through their analytical expressions.
		\item $\bm{w}$ and $\bm{q}$ are treated as time-dependent variables, evolved concurrently with the original equations  \eqref{eq:2D_fix_BalanceLaw}. This method includes the integration of an additional equation, $\partial b/\partial t = 0$, into the original equations.
	\end{itemize}
	While the first approach tends to be more accurate, especially in scenarios involving discontinuous bottom topography, it can compromise the WB property during mesh movement, as demonstrated in Theorems \ref{Prop:1D_WB} and \ref{2D:Prove_WB}. Consequently, we opt for the second approach, which allows the numerical solutions to serve as high-order approximations of the analytical function $b$.

	\section{Structure-preserving schemes for 1D SWEs}\label{Sec:Numerical_Schemes}
	For clarity, this section describes the high-order structure-preserving finite volume schemes for the 1D SWEs in curvilinear coordinates.
	In the 1D case, since $J \frac{\partial \xi}{\partial x} \equiv 1$, the SCL holds automatically. 
	For the 1D SWEs, the conservative vector and flux are defined as
	\begin{equation*}
		\bm{U} = \left(h,hv_1\right)^{\top}, \quad 
		\bm{F} = \left(hv_1,hv_1^2+\frac{1}{2}gh^2\right)^{{\top
		}}, \quad 
		s = \left(0,-ghb_{x_1}\right)^{\top},
	\end{equation*}
	where $v_1$ denotes the flow velocity. The steady state corresponding to still water is given by 
	\begin{equation*}
		h+b = C,\quad v_1=0,
	\end{equation*}
	which satisfy \eqref{eq:Steady_state}  with $\bm{w}(x)=(b,0)^{\top}$ and $\bm{q}(x) = (1,0)^{\top}$.  
	The corresponding functions in the decomposition \eqref{eq:SourceTerm_Decom} become 
	\begin{equation*}
		s_1 = -g(h+b),\quad s_2 = \frac{1}{2}g, \quad \bm{r}_1 = \left(0,b\right)^{\top}, \quad \bm{r}_2 = \left(0,b^2\right)^{\top}.
	\end{equation*}
	The admissible state set is  
	\begin{equation*}
		\mathcal{G}= \left\{\bm{U}=(h,hv_1)^\top\in\mathbf{R}^{2}:~ h \geqslant0\right\}.
	\end{equation*}
	
	\subsection{Outline of 1D structure-preserving finite volume schemes on moving meshes}\label{Sec:1D_case}
	
	Denote the computational domain as $[0,1]$. Assume that a uniform computational mesh is used, given by 
	$\xi_i = i\Delta\xi,~i=0,1,\cdots,N-1,~\Delta\xi = 1/(N-1)$, and the cell $I_{i} = \left[\xi_{i-\frac{1}{2}}, \xi_{i+\frac{1}{2}}\right]$. The semi-discrete finite volume schemes for \eqref{eq:1D_Cur_eq}  can be expressed as
		\begin{align}
			\frac{d(\overline{J\bm{U}})_{i}}{dt} =& 
			-\frac{1}{\Delta \xi}
			\left(
			\bm{\widetilde{\mathcal{F}}}_{i+\frac{1}{2}} 
			-\bm{\widetilde{\mathcal{F}}}_{i-\frac{1}{2}}\right)
			+\dfrac{1}{\Delta \xi}
			\sum\limits_{m=1}^2
			\left(\overline{s}_{m}\right)_{i}
			\left(
			{\left(
				\widetilde{\bm{\mathcal{R}}}_{m}
				\right)_{i+\frac{1}{2}}} 
			-
			{\left(
				\widetilde{\bm{\mathcal{R}}}_{m}
				\right)_{i-\frac{1}{2}}}
			\right)\nonumber\\
			&
			+ \dfrac{1}{\Delta \xi}\int_{I_{i}}
			\sum\limits_{m=1}^2
			\left(
			s_m
			-\left(
			\overline{s}_{m}
			\right)_{i}
			\right)\pd{\bm{\mathcal{R}}_m}{\xi}\mathrm{d}\xi,  \label{eq:1D_dis_U}
			\\
			\frac{d\overline{J}_{i}}{dt} = &
			-\frac{1}{\Delta \xi}
			\left(
			\left(
			{
				J\pd{\xi}{t}}\right)_{i+\frac{1}{2}}
			-\left(
			{
				J\pd{\xi}{t}}\right)_{i-\frac{1}{2}}\right), \label{eq:1D_dis_J}
		\end{align}
	where
	\begin{equation}\label{eq:1D_F_Flux}
		\bm{\widetilde{\mathcal{F}}}_{i\pm\frac{1}{2}} =   \left({J\pd{\xi}{t}}\right)_{i\pm\frac{1}{2}}\mean{\bm{U}}_{i\pm\frac{1}{2}}  - \frac{\alpha}{2}\jump{\bm{U}}_{i\pm\frac{1}{2}} +\widetilde{\bm{F}}_{i\pm\frac{1}{2}},~\left({
			J\pd{\xi}{t}}\right)_{i\pm\frac{1}{2}} = 
		\left(\widehat{
			J\pd{\xi}{t}}\right)_{i\pm\frac{1}{2}}-\frac{\beta_{J}}{2}\jump{J}_{i\pm\frac{1}{2}}.
	\end{equation}
	Here, $\left(\widehat{
		J\frac{\partial\xi}{\partial t}}\right)_{i\pm\frac{1}{2}} = - \Dot{x}_{i\pm\frac{1}{2}}$, the notation $\Dot{x}$ represents the moving mesh velocity which will be defined in Section \ref{Sec:MM}, and  $$
	\alpha = \max\limits_{i}\abs{\Dot{x}_{i+\frac{1}{2}}},\quad \beta_{J} = \max\limits_{i}\abs{\Dot{x}_{i+\frac{1}{2}}/J_{i+\frac{1}{2}}},  \quad \mean{u}_{i+\frac{1}{2}} = \dfrac{1}{2}\left( u_{i+\frac{1}{2}}^{+} + u_{i+\frac{1}{2}}^{-}\right), \quad \jump{u}_{i+\frac{1}{2}} =  u_{i+\frac{1}{2}}^{+} - u_{i+\frac{1}{2}}^{-},$$ 
	where the superscripts ``$+$'' and  ``$-$'' indicate the right- and left-hand side limits at the cell interface. The procedure of obtaining $\bm{U}_{i\pm\frac{1}{2}}^{\pm}$
	will be discussed in Subsections \ref{Sec:1D_Qua_Recon}.   
	The integral in equation \eqref{eq:1D_dis_U} can be evaluated using numerical quadrature rules. Denote the quadrature points over cell $I_{i}$ as $\{ (\xi)_{i_{\mu}} \}$ with the corresponding weights $\{\omega_{\mu}\}$ and $\mu = 1,\dots, Q$. We use the notation $(\cdot)_{i_{\mu}}$ to represent the values at the quadrature points $\{ (\xi)_{i_{\mu}} \}$ within cell $I_{i}$. Unless otherwise specified, the four-point $(Q = 4)$ Gauss--Lobatto quadrature rule is adopted to achieve fifth-order accuracy.
	Then 
	\begin{equation}\label{eq:1D_integral}
		\dfrac{1}{\Delta \xi}\int_{I_{i}}
		\sum\limits_{m=1}^2
		\left(
		s_m
		-\left(
		\overline{s}_{m}
		\right)_{i}
		\right)\pd{\bm{\mathcal{R}}_m}{\xi}\mathrm{d}\xi 
		\approx
		\sum_{\mu=1}^{Q}\sum\limits_{m=1}^2 \omega_{\mu}
		\left((s_{m})_{i_\mu} - \left(\overline{s}_{m}\right)_{i}\right) \left[\pd{\bm{\mathcal{R}}_m}{\xi}\right]_{i_{\mu}}.
	\end{equation}
	For WB considerations, the numerical fluxes in equation \eqref{eq:1D_dis_U} are defined based on the hydrostatic reconstruction approach  \cite{AuDusse2004fast} as
	\begin{equation}\label{eq:1D_num_flux}
		\begin{aligned}
			&\widetilde{\bm{F}}_{i+\frac{1}{2}}
			= \widetilde{\bm{F}}_{i+\frac{1}{2}}^{L}
			= \widetilde{\bm{F}}_{i+\frac{1}{2}}
			(
			\bm{U}_{i+\frac{1}{2}}^{-,*}
			,
			\bm{U}_{i+\frac{1}{2}}^{+,*}
			) 
			+ \left(\begin{array}{cc}
				0  \\
				\frac{g}{2}\left(h_{i+\frac{1}{2}}^{-}\right)^2-\frac{g}{2}\left(h_{i+\frac{1}{2}}^{-,*}\right)^2 
			\end{array}
			\right),
			\\
			&\widetilde{\bm{F}}_{i-\frac{1}{2}}
			= \widetilde{\bm{F}}_{i-\frac{1}{2}}^{R}
			= \widetilde{\bm{F}}_{i-\frac{1}{2}}
			(
			\bm{U}_{i-\frac{1}{2}}^{-,*}
			,
			\bm{U}_{i-\frac{1}{2}}^{+,*}
			) 
			+ \left(\begin{array}{cc}
				0  \\
				\frac{g}{2}\left(h_{i-\frac{1}{2}}^{+}\right)^2-\frac{g}{2}\left(h_{i-\frac{1}{2}}^{+,*}\right)^2 
			\end{array}
			\right),
			\\
			&\left(
			\widetilde{\bm{\mathcal{R}}}_m
			\right)_{i+\frac{1}{2}} 
			= 
			\left(
			\bm{\mathcal{R}}_m
			\right)_{i+\frac{1}{2}}^{-}
			,~\left(
			\widetilde{\bm{\mathcal{R}}}_m
			\right)_{i-\frac{1}{2}} 
			= 
			\left(
			\bm{\mathcal{R}}_m
			\right)_{i-\frac{1}{2}}^{+}
		\end{aligned}
	\end{equation}
	with 
	\begin{equation*}
		\bm{\widetilde{F}}
		\left(\bm{U}_1,\bm{U}_2\right) 
		= \frac{1}{2}\left(\bm{F}
		\left(
		\bm{U}_1
		\right)+\bm{F}
		\left(\bm{U}_2
		\right)
		\right) 
		- \frac{\beta}{2}\left(\bm{U}_2-\bm{U}_1\right),~
		\beta = \max\limits_{i}\left(\abs{(\overline{v}_1)_{i}}+\sqrt{g\overline{h}_{i}}\right),
	\end{equation*}
	\begin{equation*}
		\bm{U}_{i\pm\frac{1}{2}}^{\pm,*} = 
		\left(\begin{array}{cc}
			h_{i\pm\frac{1}{2}}^{\pm,*} \\
			h_{i\pm\frac{1}{2}}^{\pm,*}(v_1)_{i\pm\frac{1}{2}}^{\pm}
		\end{array}\right),~h_{i\pm\frac{1}{2}}^{\pm,*} 
		= \max
		\left(
		0, h_{i\pm\frac{1}{2}}^{\pm}+b_{i\pm\frac{1}{2}}^{\pm} - \max
		\left(
		b_{i\pm\frac{1}{2}}^{+},
		b_{i\pm\frac{1}{2}}^{-}
		\right)
		\right).
	\end{equation*}
	Similar to the findings in \cite{Xing2006High}, when a steady state is reached, we have 
	\begin{equation}\label{eq:1D_Flux_Condition}
		\widetilde{\bm{\bm{F}}}_{i\pm\frac{1}{2}} - \sum\limits_{m=1}^2
		\left(
		\overline{s}_m\right)_{i} \left(\bm{\widetilde{R}}_{m}\right)_{i\pm\frac{1}{2}} \equiv C.
	\end{equation}
	As discussed in Section \ref{sec:22}, the computations of $\bm{w}$ and $\bm{q}$ are crucial for maintaining well-balancedness. The governing equations for $\bm{w}$ and $\bm{q}$ in curvilinear coordinates are given by 
	\begin{equation*}	\frac{\partial J\bm{w}}{ \partial t} +	\pd{}{\xi}	\left(	J\pd{\xi}{t}\bm{w}\right) = 0,~	\frac{\partial J\bm{q}}{ \partial t} +	\pd{}{\xi}	\left(	J\pd{\xi}{t}\bm{q}\right) = 0.
	\end{equation*}
	From these equations, we devise the following semi-discrete scheme for $\bm{w}$ and $\bm{q}$: 
	\begin{align}
		\frac{d(\overline{J\bm{w}})_{i}}{dt} =& 
		-\frac{1}{\Delta \xi}
		\left(
		\widetilde{\bm{w}}_{i+\frac{1}{2}} 
		-\widetilde{\bm{w}}_{i-\frac{1}{2}}\right), \label{eq:1D_dis_p}\\
		\frac{d(\overline{J\bm{q}})_{i}}{dt} =& 
		-\frac{1}{\Delta \xi}
		\left(
		\widetilde{\bm{q}}_{i+\frac{1}{2}} 
		-
		\widetilde{\bm{q}}_{i-\frac{1}{2}}\right),  \label{eq:1D_dis_q}
	\end{align}
	which updates the numerical values of $\bm{w}$ and $\bm{q}$ from the old mesh to the new mesh, with  
	\begin{equation}\label{eq:1D_flux_pq}
		\begin{aligned}
			&\widetilde{\bm{w}}_{i\pm\frac{1}{2}} = \left(
			J\pd{\xi}{t}
			\right)_{i\pm\frac{1}{2}}
			\mean{\bm{w}}_{i\pm\frac{1}{2}}  - \frac{\alpha}{2}
			\jump{\bm{w}}_{i\pm\frac{1}{2}},\\
			&\widetilde{\bm{q}}_{i\pm\frac{1}{2}} = \left(
			J\pd{\xi}{t}
			\right)_{i\pm\frac{1}{2}}
			\mean{\bm{q}}_{i\pm\frac{1}{2}}  - \frac{\alpha}{2}
			\jump{\bm{q}}_{i\pm\frac{1}{2}}.
		\end{aligned}
	\end{equation}
	The cell averages of $\bm{U}$, $\bm{w}$, and $\bm{q}$ in physical domain are then computed by 
	\begin{equation*}
		\overline{\bm{U}}_{i} := 
		\frac
		{\overline{J\bm{U}}_{i}}{\overline{J}_{i}} = \frac{\int_{I_i}J\bm{U}\mathrm{d}\xi}{\int_{I_i}J\mathrm{d}\xi},~
		\overline{\bm{w}}_{i} := 
		\frac
		{\overline{J\bm{w}}_{i}}
		{\overline{J}_{i}}=
		\frac{\int_{I_i}
			J\bm{w}
			\mathrm{d}\xi}
		{\int_{I_i}
			J\mathrm{d}\xi},~ 
		\overline{\bm{q}}_{i} := 
		\frac
		{\overline{J\bm{q}}_{i}}
		{\overline{J}_{i}}
		=
		\frac{\int_{I_i}
			J\bm{q}
			\mathrm{d}\xi}
		{\int_{I_i}
			J\mathrm{d}\xi}.
	\end{equation*}
	The quadrature rules are used here to approximate the integral, e.g., $\overline{\bm{U}}_{i} = \sum\limits_{\mu=1}^{Q}\omega_{\mu} J_{i_{\mu}}\bm{U}_{i_{\mu}}/\sum\limits_{\mu=1}^{Q}\omega_{\mu} J_{i_{\mu}}$.
	\subsection{High-order 1D reconstruction}\label{Sec:1D_Qua_Recon}
	This subsection discusses the 1D WENO reconstruction in the computational domain.  
	It is crucial to obtain high-order accurate approximate values of ${\bm U}$  at the quadrature nodes  
	$\left\{\left(\xi\right)_{i_{\mu}}\right\}$ when calculating the integral in \eqref{eq:1D_integral}. We utilize the WENO-Z  reconstruction technique \cite{Borges2008An}, specifically adapted to accommodate steady-state conditions during mesh movement. 
	In this work, we employ the following reconstruction approach:

	Step 1.  Starting with the cell average values $\overline{J\bm{U}}_{i}$ and $\overline{J}_{i}$, 
	we apply the WENO reconstructions to obtain the high-order approximate values 
	$\left\{\mathring{J\bm{U}}_{i_{\mu}}\right\}$ and $\left\{\mathring{J}_{i_{\mu}}\right\}$ of $J\bm{U}$ and $\bm{U}$ at the quadrature nodes $\left\{\left(\xi\right)_{i_{\mu}}\right\}$.
	Then  $\mathring{\bm{U}}_{i_{\mu}} = \mathring{J\bm{U}}_{i_{\mu}} / \mathring{J}_{i_{\mu}}$ serves as the high-order approximations for  $\bm{U}$ at the quadrature points. To enforce the cell average of $J$, we modify ${J}_{i_{\mu}} = \mathring{J}_{i_{\mu}}+\overline{J}_{i} - \sum\limits_{\mu}\omega_{\mu} \mathring{J}_{i_{\mu}}$ to obtain $\sum\limits_{\mu}\omega_{\mu} {J}_{i_{\mu}} = \overline{J}_{i}$.
	
	Step 2.  
Calculate the approximate cell average $\widetilde{\bm{U}}_{i} = \sum \limits_{\mu} \omega_{\mu} \mathring{\bm{U}}_{i_\mu}$ over cell $I_{i}$ within the computational domain. 
	Using these cell averages $\{\widetilde{\bm{U}}_{i}\}$, employ the WENO reconstruction to ascertain the robust high-order approximate values $\check{\bm{U}}_{i_{\mu}}$ at the quadrature points $\left\{\left(\xi\right)_{i_{\mu}}\right\}$.
	Denote the difference between the cell average $\overline{J\bm{U}}_{i}$ and the values $\sum\limits_{\mu}\omega_{\mu} J_{i_{\mu}} \check{\bm{U}}_{i_{\mu}}$ over cell $I_{i}$ by $\bm{\gamma}_{i} = \overline{J\bm{U}}_{i}-(\sum\limits_{\mu}\omega_{\mu} J_{i_{\mu}} \check{\bm{U}}_{i_{\mu}})$. To maintain the cell average after the reconstruction procedure. The reconstruction values are modified as ${\bm{U}}_{i_{\mu}}  = \check{\bm{U}}_{i_{\mu}}+ \bm{\gamma}_{i}/J_{i_{\mu}}$. Therefore, $\sum\limits_{\mu}\omega_{\mu} J_{i_{\mu}} {\bm{U}}_{i_{\mu}} = \overline{J\bm{U}}_{i}$.
\begin{remark}
	The modifications in Steps 1 and 2 are essential for the PP analysis in Theorems \ref{Pro:1D_SWE_PP_J} and \ref{Pro:1D_SWE_PP} and maintain the accuracy and conservation in the computational domain during the reconstruction procedure. Take the modification in Step 1 as an example. Assume that $J$ is sufficiently smooth. If the fifth-order WENO reconstruction is employed, one has $\mathring{J}_{i_{\mu}} = J_{i_{\mu}}^{exact}+O((\Delta \xi)^5)$, where $J_{i_{\mu}}^{exact}$ represents the exact value of $J$ at the quadrature points $\left(\xi\right)_{i_{\mu}}$ within $I_{i}$. Noting that the exact value $J_{i_{\mu}}^{exact}$ and the numerical cell average  $\overline{J}_{i}$ satisfy the relation $\sum\limits_{\mu}\omega_{\mu}J_{i_{\mu}}^{exact} = \overline{J}_{i}+O((\Delta \xi )^{5})$, one obtains  $ \sum\limits_{\mu}\omega_{\mu}\mathring{J}_{i_{\mu}} = \sum\limits_{\mu}\omega_{\mu}J_{i_{\mu}}^{exact}+O((\Delta \xi)^5) = \overline{J}_{i}+O((\Delta \xi )^{5})$. Therefore, $\overline{J}_{i} - \sum\limits_{\mu}\omega_{\mu} \mathring{J}_{i_{\mu}} = O((\Delta \xi)^{5})$. Consequently, $${J}_{i_{\mu}} = \mathring{J}_{i_{\mu}}+\overline{J}_{i} - \sum\limits_{\mu}\omega_{\mu} \mathring{J}_{i_{\mu}} = J_{i_{\mu}}^{exact}+O((\Delta \xi)^5) + \overline{J}_{i} - \sum\limits_{\mu}\omega_{\mu} \mathring{J}_{i_{\mu}} = J_{i_{\mu}}^{exact}+O((\Delta \xi)^5),$$
		which means this modification does not affect the accuracy. {\color{blue}}From ${J}_{i_{\mu}} = \mathring{J}_{i_{\mu}}+\overline{J}_{i} - \sum\limits_{\mu}\omega_{\mu} \mathring{J}_{i_{\mu}}$, one obtains $\sum\limits_{\mu} \omega_{\mu}J_{i_{\mu}}= \sum\limits_{\mu} \omega_{\mu}\mathring{J}_{i_{\mu}}+\sum\limits_{\mu} \omega_{\mu}\overline{J}_{i} -\sum\limits_{\mu} \omega_{\mu} \sum\limits_{\tilde{\mu}}\omega_{\tilde{\mu}}\mathring{J}_{i_{\tilde{\mu}}} $. Since $\sum\limits_{\mu} \omega_{\mu} = 1$, one has $\sum\limits_{\mu} \omega_{\mu}J_{i_{\mu}}= \sum\limits_{\mu} \omega_{\mu} \mathring{J}_{i_{\mu}}+\overline{J}_{i} - \sum\limits_{\tilde{\mu}}\omega_{\tilde{\mu}}\mathring{J}_{i_{\tilde{\mu}}} = \overline{J}_{i}$, which implies this  modification ensures the conservative property of $J$ in the computational domain. Similar to the analysis for $J$ in Step 1, the modification in Step 2 does not affect the accuracy and conservation in the computational domain. Note that $\overline{\bm{U}}_{i}  = \overline{J\bm{U}}_{i}/\overline{J}_{i} $,
		$ \sum\limits_{\mu} \omega_{\mu} J_{i_{\mu}}\bm{U}_{i_{\mu}} = \overline{J\bm{U}}_{i}$, and $ \sum\limits_{\mu} \omega_{\mu} J_{i_{\mu}} = \overline{J}_{i}$. The modifications do not influence the conservation in the physical domain during the reconstruction procedure, as we merely modify the reconstruction values at the  quadrature points. 
\end{remark}

\begin{remark}
It seems more straightforward to directly use Step 1 to calculate the reconstruction values $\mathring{\bm{U}}_{i_{\mu}}$ at the quadrature points $\left(\xi\right)_{i_{\mu}}$. This strategy also preserves to the free-stream property. However, our numerical experiments indicate that merely performing Step 1 can result in
significant overshoots or undershoots near discontinuities, as illustrated in Figure \ref{fig:1D_Perturbation_Compare}. 
To address
this issue, we use the reconstruction approach consisting of Steps 1 and 2.
\end{remark}
		\begin{remark}\rm
		During mesh movement, it is crucial that the coefficients used for the reconstruction in 
		$ \mathring{J(h+b)}_{i_{\mu}} $ and $ \mathring{J}_{i_\mu} $ remain identical. Specifically, if 
		$ \mathring{J(h+b)}_{i+\frac{1}{2}}^{-} $ is reconstructed as 
		$ \sum\limits_{l=-\gamma+1}^{\gamma-1} w_{l} \overline{J(h+b)}_{i+l} $
		to determine the left limit of $ J(h+b) $, then $\mathring{J}_{i+\frac{1}{2}}^{-} $ must be similarly  reconstructed using the same weights: 
		$ \sum\limits_{l=-\gamma+1}^{\gamma-1}w_{l} \overline{J}_{i+l} $. 
		This consistent treatment ensures that the ratio 
		$ \frac{\mathring{J(h+b)}_{i+\frac{1}{2}}^{-}}{\mathring{J}_{i+\frac{1}{2}}^{-}} = C $, 
		which is crucial under still water equilibrium conditions, where $ (v_1)_{i+\frac{1}{2}}^{-} = 0 $.  Therefore, the relationships 
		$ \mathring{\bm{U}}_{i_{\mu}} + \mathring{\bm{w}}_{i_{\mu}} = C{\bm{q}}_{i_{\mu}}$ and $ \check{\bm{U}}_{i_{\mu}} + \check{\bm{w}}_{i_{\mu}} = C{\bm{q}}_{i_{\mu}}$  
		are maintained, aligning with the principles of hydrostatic balance. We modify ${\bm{w}}_{i_{\mu}} = \check{\bm{w}}_{i_{\mu}}- \bm{\gamma}_{i}/J_{i_{\mu}}$ to ensure ${\bm{U}}_{i_{\mu}} + {\bm{w}}_{i_{\mu}} = C{\bm{q}}_{i_{\mu}}$.  This is important for the WB property of our scheme under mesh movement. 
	\end{remark}
		The condition 
		\begin{equation}\label{eq:PP_Condition}
			J_{i_{\mu}} >0,\quad 
			h_{i_{\mu}}\geqslant 0 \quad \forall i, \forall \mu, 
		\end{equation}
		is essential for the  PP property of our moving mesh finite volume schemes in Theorems \ref{Pro:1D_SWE_PP_J} and \ref{Pro:1D_SWE_PP}. However, the reconstruction procedure does not always guarantee satisfaction of condition \eqref{eq:PP_Condition}.
	Following \cite{Zhang2010On,Zhang2011Maximum,Xing2010Positivity,Zhang2017On}, a PP limiter should be implemented to enforce this condition.
		\begin{remark}[1D PP limiter] \label{rmk:1D_PP_Limiter}
			If $\overline{J}_{i}>0$ and $\overline{h}_{i}\geqslant 0$, the PP limiter modifies the reconstructed values $\{{J}_{i_{\mu}}\}$ and $\{{h}_{i_{\mu}}\}$ as 
			\begin{equation}\label{eq:modifiedhJ}
				\widehat{J}_{i_\mu}
				= {\theta_J}
				\left({J}_{i_\mu}
				-\overline{J}_{i}\right)
				+\overline{J}_{i},\quad \widehat{h}_{i_\mu}
				= \mathcal{\theta}_{h}
				\left({h}_{i_\mu}
				-\overline{h}_{i}\right)
				+\overline{h}_{i} \quad \forall \mu
			\end{equation}
			with
			\begin{equation}
				\mathcal{\theta}_{J}: = 
				\begin{cases}
					~1, &\text{if}\quad \min\limits_{\mu} 
					\left\{
					{J}_{i_\mu}
					\right\} = \overline{J}_{i},\\
					~\min 
					\left\{
					\left|
					\frac{
						\overline{J}_{i}-\epsilon_{0}}
					{\overline{J}_{i}-\min\limits_{\mu} 
						\left\{
						{J}_{i_\mu}
						\right\}}
					\right|, 
					1
					\right\}, &\text{otherwise},
				\end{cases}
			\end{equation}
			\begin{equation}
				\mathcal{\theta}_{h}:=
				\begin{cases}
					~1, &\text{if}\quad \min\limits_{\mu} 
					\left\{
					{h}_{i_\mu}
					\right\} = \overline{h}_{i},\\
					~\min 
					\left\{
					\left|
					\frac{
						\overline{h}_{i}}
					{\overline{h}_{i}-\min\limits_{\mu} 
						\left\{
						{h}_{i_\mu}
						\right\}}
					\right|, 
					1
					\right\}, &\text{otherwise}.
				\end{cases}
			\end{equation}
		\end{remark}
		The value $\epsilon_0$ is a small positive number, which represents the desired lower bound of $J$ and can be taken as $\epsilon_0 =  \min\left\{10^{-13}, \overline{J}_{i}\right\}$. 
		After applying the PP limiter, the modified values, given by \eqref{eq:modifiedhJ}, satisfy 
		$\widehat {J}_{i_\mu} > 0$
		and		$\widehat {h}_{i_\mu} \geqslant 0$. Notice that, if ${h}_{i_\mu} \geqslant 0$, one has $\theta_{h} = 1$ and $\widehat{h}_{i_\mu} = {h}_{i_\mu}$. We use $\widehat{J}_{i_{\mu}}$ to replace $J_{i_{\mu}}$ in Step 2 so as to ensure $\sum\limits_{\mu}\omega_{\mu}\widehat{J}_{i_{\mu}}\bm{U}_{i_{\mu}} = \overline{J\bm{U}}_{i}$.
	
	If $\overline{J}_{i}>0$ and $\overline{h}_{i}\geqslant 0$, the PP limiter maintains high-order accuracy, as discussed in \cite{Zhang2010On,Xing2010Positivity,Zhang2017On}. However, it may impact the WB property under still water equilibrium conditions. To address this issue, following \cite{Zhang2021High}, we update the bottom topography by 
	$\widehat{b}_{i_\mu}  = {h}_{i_\mu} 
	+ {b}_{i_\mu} 
	- \widehat{h}_{i_\mu}.
	$ 
	When a steady state is reached, we have 
	$
	{h}_{i_\mu} 
	+ 
	{b}_{i_\mu}
	\equiv C,~({v}_1)_{i_\mu} 
	= 0,
	$  and then $
	\widehat{h}_{i_\mu} 
	+ 
	\widehat{b}_{i_\mu}
	\equiv C. 
	$ 
	In our computations, we use $\widehat{\bm{U}}_{i_{\mu}} = \left(\widehat{h}_{i_{\mu}}, {\left(hv_1\right)}_{i_\mu}\right)^{\top}$,  $\widehat{\bm{w}}_{i_{\mu}} = \left(\widehat{b}_{i_{\mu}},0\right)^{\top}$ and $\widehat{J}_{i_{\mu}}$ to replace $\bm{U}_{i_{\mu}}$, $\bm{w}_{i_{\mu}}$ and ${J}_{i_{\mu}}$ respectively.
	Hence,  $\widehat{\bm{U}}_{i_\mu}+\widehat{\bm{w}}_{i_{\mu}}  = \bm{U}_{i_\mu}+\bm{w}_{i_{\mu}} = C\bm{q}_{i_\mu}$.
	This modification ensures the maintenance of the WB property while simultaneously preserving the non-negativity of the water depth. 
	The conservative property is preserved, since the cell averages of the conservative variables $\overline{h}_{i}$ and $\left(\overline{hv_1}\right)_{i}$ do not change.	
	Specifically
	\begin{align*}
		&\frac{\sum\limits_{\mu} \omega_{\mu} {\widehat{J}_{i_{\mu}}\widehat{h}_{i_{\mu}}}}{\sum\limits_{\mu}  \omega_{\mu} {\widehat{J}_{i_{\mu}}}} = \frac{\sum\limits_{\mu}  \omega_{\mu} \widehat{J}_{i_{\mu}} \left(\mathcal{\theta}
			\left({h}_{i_\mu}
			-\overline{h}_{i}\right) +\overline{h}_{i}\right)}{{\sum\limits_{\mu}  \omega_{\mu} {\widehat{J}_{i_{\mu}}}}} =  
		\frac{\mathcal{\theta}\left(\sum\limits_{\mu}  \omega_{\mu} \widehat{J}_{i_{\mu}}{h}_{i_\mu}		-\overline{Jh}_{i}\right) + \overline{Jh}_{i}}{\overline{J}_{i}} = \overline{h}_{i},\\
		&
		\frac{\sum\limits_{\mu}  \omega_{\mu} {\widehat{J}_{i_{\mu}}(hv_1)_{i_{\mu}}}}{\sum\limits_{\mu}  \omega_{\mu} {\widehat{J}_{i_{\mu}}}} = \frac{\overline{J(hv_1)}_{i}}{\overline{J}_{i}} =  (\overline{hv_1})_{i}.
	\end{align*}
	Therefore, the conservation is preserved.
	
	\subsection{Analysis of WB property}
	This subsection shows the WB property of the proposed schemes.
	\begin{theorem}\rm\label{Prop:1D_WB}
		
		Consider the semi-discrete schemes \eqref{eq:1D_dis_U}, \eqref{eq:1D_dis_J}, \eqref{eq:1D_dis_p}, and \eqref{eq:1D_dis_q} with the numerical fluxes \eqref{eq:1D_F_Flux} and \eqref{eq:1D_flux_pq}.
		These schemes are WB, in the sense that, under the forward Euler or explicit SSP RK discretizations, when a steady state is reached at the $n$th time level $t^n$, i.e., 
		\begin{equation*}
			\overline{\bm{U}}_{i}^{n}+\overline{\bm{w}}_{i}^{n} = C\overline{\bm{q}}_{i}^{n}\quad \forall i,
		\end{equation*}
		the numerical solution at $t^{n+1}$ also satisfies 
		\begin{equation*}
			\overline{\bm{U}}_{i}^{n+1}+\overline{\bm{w}}_{i}^{n+1} = C\overline{\bm{q}}_{i}^{n+1}\quad \forall i.
		\end{equation*}
	\end{theorem}
	\begin{proof}
		Since an SSP RK method is a convex combination of forward Euler method, we only need to show the conclusion for  the forward Euler discretization. 
		When the steady state is reached at time $t^n$, 
		the condition \eqref{eq:1D_Flux_Condition} holds, and the discrete evolution of $J\bm{U}$ can be expressed as follows: 
		\begin{equation}
			\left(
			\overline{J\bm{U}}
			\right)_{i}^{n+1} - \left(
			\overline{J\bm{U}}
			\right)_{i}^{n} = 
			-\dfrac{\Delta t}
			{\Delta \xi}
			\left(
			\left(
			J\pd{\xi}{t}
			\right)_{i+\frac{1}{2}}
			\mean{\bm{U}}_{i+\frac{1}{2}}  
			- \frac{\alpha}{2}
			\jump{\bm{U}}_{i+\frac{1}{2}}
			-\left(J\pd{\xi}{t}\right)_{i-\frac{1}{2}}
			\mean{\bm{U}}_{i-\frac{1}{2}}  + \frac{\alpha}{2}
			\jump{\bm{U}}_{i-\frac{1}{2}}
			\right),\label{eq:1D_U_full_dis}
		\end{equation}
		where and hereafter the superscript $n$ is omitted for convenience. 
		The fully-discrete schemes of  $J\bm{w}$ and $J\bm{q}$ are given by
		\begin{align}
			\left(
			\overline{J\bm{w}}
			\right)_{i}^{n+1} - \left(
			\overline{J\bm{w}}
			\right)_{i}^{n} &= 
			-\dfrac{\Delta t}
			{\Delta \xi}
			\left(
			\left(
			J\pd{\xi}{t}
			\right)_{i+\frac{1}{2}}
			\mean{\bm{w}}_{i+\frac{1}{2}}  
			- \frac{\alpha}{2}\jump{\bm{w}}_{i+\frac{1}{2}}
			-\left(J\pd{\xi}{t}\right)_{i-\frac{1}{2}}
			\mean{\bm{w}}_{i-\frac{1}{2}}  
			+ \frac{\alpha}{2}
			\jump{\bm{w}}_{i-\frac{1}{2}}
			\right),\label{eq:1D_p_full_dis}\\
			\left(
			\overline{J\bm{q}}
			\right)_{i}^{n+1} - \left(
			\overline{J\bm{q}}
			\right)_{i}^{n} &= 
			-\dfrac{\Delta t}
			{\Delta \xi}
			\left(
			\left(
			J\pd{\xi}{t}
			\right)_{i+\frac{1}{2}}
			\mean{\bm{q}}_{i+\frac{1}{2}}  
			- \frac{\alpha}{2}\jump{\bm{q}}_{i+\frac{1}{2}}
			-\left(J\pd{\xi}{t}\right)_{i-\frac{1}{2}}
			\mean{\bm{q}}_{i-\frac{1}{2}}  
			+ \frac{\alpha}{2}
			\jump{\bm{q}}_{i-\frac{1}{2}}
			\right).\label{eq:1D_q_full_dis}
		\end{align}
		Summing \eqref{eq:1D_U_full_dis} and  \eqref{eq:1D_p_full_dis} yields
		\begin{align*}
			\left(
			\overline{J\left(\bm{U}+\bm{w}\right)}
			\right)_{i}^{n+1} - \left(
			\overline{J\left(\bm{U}+\bm{w}\right)}
			\right)_{i}^{n} = 
			&-\dfrac{\Delta t}
			{\Delta \xi}
			\Bigg(
			\left(
			J\pd{\xi}{t}
			\right)_{i+\frac{1}{2}}
			\mean{\bm{U}+\bm{w}}_{i+\frac{1}{2}}  
			- \frac{\alpha}{2}
			\jump{\bm{U}+\bm{w}}_{i+\frac{1}{2}}
			\\
			&\quad -\left(J\pd{\xi}{t}\right)_{i-\frac{1}{2}}
			\mean{\bm{U}+\bm{w}}_{i-\frac{1}{2}}  + \frac{\alpha}{2}
			\jump{\bm{U}+\bm{w}}_{i-\frac{1}{2}}
			\Bigg).
		\end{align*}
		Based on the	relation  
		$ {\bm{U}}_{i_{\mu}} + {\bm{w}}_{i_{\mu}} = C{\bm{q}}_{i_{\mu}} $, one obtains $\mean{\bm{U}+\bm{w}} = C\mean{\bm{q}}$ and $\jump{\bm{U}+\bm{w}} = C\jump{\bm{q}}$. This implies
		\begin{align*}
			\left(
			\overline{J\left(\bm{U}+\bm{w}\right)}
			\right)_{i}^{n+1} - \left(
			\overline{J\left(C\bm{q}\right)}
			\right)_{i}^{n} = 
			&-\dfrac{\Delta t}
			{\Delta \xi}
			\Bigg(
			\left(
			J\pd{\xi}{t}
			\right)_{i+\frac{1}{2}}
			\mean{C\bm{q}}_{i+\frac{1}{2}}  
			- \frac{\alpha}{2}
			\jump{C\bm{q}}_{i+\frac{1}{2}}
			\\
			&\quad -\left(J\pd{\xi}{t}\right)_{i-\frac{1}{2}}
			\mean{C\bm{q}}_{i-\frac{1}{2}}  
			+ \frac{\alpha}{2}
			\jump{C\bm{q}}_{i-\frac{1}{2}}
			\Bigg).
		\end{align*}
		Plugging \eqref{eq:1D_q_full_dis} into the above equations, one obtains
		\begin{align*}
			\left(
			\overline{
				J\left(
				\bm{U}+\bm{w}
				\right)
			}
			\right)_{i}^{n+1} 
			- \left(
			\overline{
				J\left(C\bm{q}\right)
			}
			\right)_{i}^{n+1} = 0,
		\end{align*}
		which implies 
		\begin{equation*}
			\left(
			\overline{
				\bm{U}+\bm{w}
			}\right)_{i}^{n+1} 
			= C
			\overline{\bm{q}}_{i}^{n+1}.
		\end{equation*}
		The proof is completed.
	\end{proof}
	\subsection{Analysis of PP property} \label{sec:1DPP}
	The semi-discrete schemes \eqref{eq:1D_dis_U} and \eqref{eq:1D_dis_J}  coupled with the forward Euler time discretization can be expressed as 
	\begin{equation*}
		\overline{J\bm{U}}_{i}^{\Delta t}  = \overline{J\bm{U}}_{i}  - \Delta t\bm{L}_{i},\quad \overline{J}_{i}^{\Delta t}  = \overline{J}_{i}  - \Delta t{L}_{i},
	\end{equation*}
	where $\bm{L}_{i}$ and $L_{i}$ denote the right-hand-side terms of  \eqref{eq:1D_dis_U}  and \eqref{eq:1D_dis_J}, respectively; $\overline{J\bm{U}}_{i}$ and $\overline{J}_{i}$ represent the approximations at the $n$th time level.
		\begin{theorem}\rm\label{Pro:1D_SWE_PP_J}
			Assume that $\overline{J}_{i}>0$ for all $i$. If 
			$
			J_{i_{\mu}}>0  
			$
			for all $i$, $\mu$,
			then 
			$
			\overline{J}_{i}^{\Delta t} > 0
			$
			holds under the condition
			\begin{equation}\label{eq:1D_CFL_condition_J}
				\Delta t \leqslant  		\dfrac{
					\omega_1 
					\Delta \xi 
				}
				{
					\beta_{J}
				}. 
			\end{equation}
		\end{theorem}
	\begin{proof}
	Consider the discrete evolution equation for $J$: 
	\begin{equation}\label{eq:1D_PP_J}
		\begin{aligned}
			\overline{J}_{i}^{\Delta t}  - 
			\overline{J}_{i}
			=&  
			- \frac{\Delta t}{\Delta \xi}
			\left[
			\left(\widehat{J \pd{\xi}{t}}\right)
			_{i+\frac{1}{2}}
			-\frac{\beta_{J}}{2}
			\jump{J}
			_{i+\frac{1}{2}}
			-
			\left(\widehat{J \pd{\xi}{t}}\right)
			_{i-\frac{1}{2}}
			+\frac{\beta_{J}}{2}
			\jump{J}
			_{i-\frac{1}{2}}
			\right]
			\\
			=&~\frac{\Delta t}{2\Delta \xi}
			\left[
			\beta_J J
			_{i+\frac{1}{2}}^{+}
			-
			\left(\widehat{J \pd{\xi}{t}}
			\right)
			_{i+\frac{1}{2}}
			+
		\beta_J J
		_{i+\frac{1}{2}}^{-}
		-
		\left(\widehat{J \pd{\xi}{t}}
		\right)
		_{i+\frac{1}{2}}
			\right]
				\\
			& +\frac{\Delta t}{2\Delta \xi}
			\left[
			\beta_J J
			_{i-\frac{1}{2}}^{+}
			+
			\left(\widehat{J \pd{\xi}{t}}
			\right)
			_{i-\frac{1}{2}}
			+
			\beta_J J
			_{i-\frac{1}{2}}^{-}
			+
			\left(\widehat{J \pd{\xi}{t}}
			\right)
			_{i-\frac{1}{2}}
			\right]\\
			& -\frac{\beta_J\Delta t }{\Delta \xi}
				\left(
			J
			_{i+\frac{1}{2}}^{-}
			+
			J
			_{i-\frac{1}{2}}^{+}
			\right).
		\end{aligned}
	\end{equation}
	Noting that
	\begin{equation*}
		\beta_J J_{i+\frac{1}{2}}^{\pm} - 	\left(\widehat{J \pd{\xi}{t}}
		\right)
		_{i+\frac{1}{2}}\geqslant 0,\quad	\beta_J J_{i-\frac{1}{2}}^{\pm} + 	\left(\widehat{J \pd{\xi}{t}}
		\right)
		_{i-\frac{1}{2}}\geqslant 0,\quad 
		\overline{J}_{i}
	=\sum_{\mu = 2}^{Q-1}
	\omega_{\mu}
	J_{i_{\mu}} 
	+ \omega_1
	\left(
	J
	_{i+\frac{1}{2}}^{-}
	+
	J
	_{i-\frac{1}{2}}^{+}
	\right),
\end{equation*}
  a lower bound of the right-hand-side of \eqref{eq:1D_PP_J} is given by  
	\begin{align*}
		\overline{J}_{i}^{\Delta t} \geqslant &\sum_{\mu = 2}^{Q-1}
		\omega_{\mu}
		J_{i_{\mu}} 
		+ \omega_1
			\left(
		J
		_{i+\frac{1}{2}}^{-}
		+
		J
		_{i-\frac{1}{2}}^{+}
		\right) - \frac{\beta_J\Delta t }{\Delta \xi}
			\left(
		J
		_{i+\frac{1}{2}}^{-}
		+
		J
		_{i-\frac{1}{2}}^{+}
		\right)
		\\
		= & \left(\omega_1 - \frac{\beta_J\Delta t }{\Delta \xi}\right)
			\left(
		J
		_{i+\frac{1}{2}}^{-}
		+
		J
		_{i-\frac{1}{2}}^{+}
		\right) + \sum_{\mu = 2}^{Q-1}
		\omega_{\mu}
		J_{i_{\mu}}.
	\end{align*}  
	Under the condition $J_{i_{\mu}}>0$ and the condition \eqref{eq:1D_CFL_condition_J}, one obtains
	\begin{equation*}
		\overline{J}_{i}^{\Delta t}  > 0.
	\end{equation*} 
	The proof is completed.
\end{proof}
		\begin{remark}\label{rm:1D_J_pp}
			The mesh should not interleave or overlap during computation, i.e., if $x_{i-\frac{1}{2}}< x_{i+\frac{1}{2}}$ for $i = 0, \dots, N-1$ then for the new mesh $x_{i-\frac{1}{2}}^{\Delta t}< x_{i+\frac{1}{2}}^{\Delta t}$ still holds, otherwise the numerical simulation will fail.  The treatment to avoid mesh interleaving or overlapping is illustrated in Section \ref{Sec:MM}.
		\end{remark}
	\begin{theorem}\rm\label{Pro:1D_SWE_PP}
		Assume that $\overline{\bm{U}}_{i}\in \mathcal{G}$ and $\overline{J}_{i}>0$ for all $i$. If $\overline{J}_{i}^{\Delta t}>0$, and 
		\begin{equation}\label{eq:1D_PP_condition}
			J_{i_{\mu}}>0, \quad	h_{i_{\mu}}\geqslant0 \qquad \forall i,~\forall \mu, 
		\end{equation}
		then the PP property 
		\begin{equation*}
			\overline{\bm{U}}_{i}^{\Delta t} := \overline{J\bm{U}}_{i}^{\Delta t}/\overline{J}_{i}^{\Delta t} \in \mathcal{G},
		\end{equation*}
		holds under the CFL condition
		\begin{equation}\label{eq:1D_CFL_condition}
			\Delta t \leqslant  		\dfrac{
				\omega_1 
				\Delta \xi 
			}
			{\alpha
				+
				\beta
			} \min_{i} \left\{ J_{i\pm\frac{1}{2}}^{\mp} \right \}. 
		\end{equation}
	\end{theorem}
	\begin{proof}
		Consider the discrete evolution equation for $Jh$: 
		\begin{equation}\label{eq:1D_PP_h}
			\begin{aligned}
				\overline{Jh}_{i}^{\Delta t}  - 
				\overline{Jh}_{i}
				=&  
				-
				\frac{\Delta t}{\Delta \xi}
				\left[
				\left(J \pd{\xi}{t}\right)
				_{i+\frac{1}{2}}
				\mean{h}
				_{i+\frac{1}{2}}
				-\frac{\alpha}{2}
				\jump{h}
				_{i+\frac{1}{2}}
				-
				\left(J \pd{\xi}{t}\right)
				_{i-\frac{1}{2}}
				\mean{h}
				_{i-\frac{1}{2}}
				+\frac{\alpha}{2}
				\jump{h}
				_{i-\frac{1}{2}}
				\right]
				\\
				&-\frac{\Delta t}{\Delta \xi}
				\Bigg[
				\frac{1}{2}\left(h_{i+\frac{1}{2}}^{-,*}(v_1)_{i+\frac{1}{2}}^{-}
				+h_{i+\frac{1}{2}}^{+,*}(v_1)_{i+\frac{1}{2}}^{+}
				\right)-
				\frac{\beta}{2}
				\jump{h}
				_{i+\frac{1}{2}}^{*} 
				\\&\quad\quad\quad\quad\quad-
				\frac{1}{2}\left(h_{i-\frac{1}{2}}^{-,*}(v_1)_{i-\frac{1}{2}}^{-}
				+h_{i-\frac{1}{2}}^{+,*}(v_1)_{i-\frac{1}{2}}^{+}
				\right)
				+
				\frac{\beta}{2}
				\jump{h}
				_{i-\frac{1}{2}}^{*}
				\Bigg],
			\end{aligned}
		\end{equation}
		Based on the relations 
		\begin{align*}
			&\pm\frac{1}{2}\left(h_{i\pm\frac{1}{2}}^{-,*}(v_1)_{i\pm\frac{1}{2}}^{-}
			+h_{i\pm\frac{1}{2}}^{+,*}(v_1)_{i\pm\frac{1}{2}}^{+}
			\right)  \leqslant 
			\beta \mean{h}
			_{i\pm\frac{1}{2}}^{*},
			\\
			& \frac{\alpha}{2}
			\jump{h}
			_{i+\frac{1}{2}}
			-\frac{\alpha}{2}
			\jump{h}
			_{i-\frac{1}{2}}
			=
			\alpha
			\mean{h}
			_{i+\frac{1}{2}}
			+\alpha
			\mean{h}
			_{i-\frac{1}{2}}
			-
			\alpha
			{h}
			_{i+\frac{1}{2}}^{-} 
			-\alpha
			{h}
			_{i-\frac{1}{2}}^{+},
		\end{align*}
		the lower bound of the right-hand side of \eqref{eq:1D_PP_h} can be estimated as
		\begin{equation}\label{eq:JHDelta_1D}
			\overline{Jh}_{i}^{\Delta t}  - 
			\overline{Jh}_{i}
			\geqslant I_1+I_2
		\end{equation}
		with
		\begin{equation*}
			\begin{aligned}
				I_1 &= \frac{\Delta t}{\Delta \xi}
				\left[
				\left(
				\alpha
				-
				J \pd{\xi}{t}
				\right)
				_{i+\frac{1}{2}}
				\mean{h}
				_{i+\frac{1}{2}}
				+
				\left(
				\alpha
				+
				J \pd{\xi}{t}
				\right)
				_{i-\frac{1}{2}}
				\mean{h}
				_{i-\frac{1}{2}}
				\right],
				\\
				I_2 &= -\frac{\Delta t}{\Delta \xi}
				\left[	
				\beta{h}
				_{i+\frac{1}{2}}^{-,*}
				+
				\beta{h}
				_{i-\frac{1}{2}}^{+,*}
				+\alpha
				{h}
				_{i+\frac{1}{2}}^{-}+\alpha
				{h}
				_{i-\frac{1}{2}}^{+}\right].\\
			\end{aligned}
		\end{equation*}
		Noting that $I_1\geqslant0$ and
		$h_{i\pm\frac{1}{2}}^{\mp,*}\leqslant h_{i\pm\frac{1}{2}}^{\mp}$, one has
		\begin{align}\label{eq:1D_h_pp_delta_t}
			\overline{h}_{i}^{\Delta t}  - \dfrac{\overline{Jh}_{i}}{\overline{J}_{i}^{\Delta t}} \geqslant
			&
			-\frac{\Delta t}{\overline{J}_{i}^{\Delta t}\Delta \xi}
			\left[
			\left(
			\alpha
			+ 
			\beta
			\right)
			{h}
			_{i+\frac{1}{2}}^{+}
			+
			\left(
			\alpha 
			+ 
			\beta
			\right)
			{h}
			_{i-\frac{1}{2}}^{+}
			\right].
		\end{align}
		Furthermore, by using the fact
		\begin{align*}
			\overline{Jh}_{i}
			=&\sum_{\mu = 2}^{Q-1}
			\omega_{\mu}
			J_{i_{\mu}} h_{i_{\mu}}
			+ \omega_1
			\left(
			J
			_{i+\frac{1}{2}}^{-}
			{h}
			_{i+\frac{1}{2}}^{-}
			+
			J
			_{i-\frac{1}{2}}^{+}
			{h}
			_{i-\frac{1}{2}}^{+}
			\right),
		\end{align*}
		the inequality \eqref{eq:1D_h_pp_delta_t} can be reformulated as
		\begin{equation*}
			\begin{aligned}
				\overline{h}_{i}^{\Delta t}   \geqslant
				&
				\frac{1}{\overline{J}_{i}^{\Delta t}}
				\left[
				\left(
				\omega_1
				{J}
				_{i+\frac{1}{2}}^{-}
				-\frac{\Delta t}{\Delta \xi}
				\left(
				\alpha
				+ 
				\beta
				\right)
				\right)
				{h}
				_{i+\frac{1}{2}}^{-}
				+
				\left(
				\omega_1
				{J}
				_{i-\frac{1}{2}}^{+}
				-\frac{\Delta t}{\Delta \xi}
				\left(
				\alpha
				+ 
				\beta
				\right)
				\right)
				{h}
				_{i-\frac{1}{2}}^{+}
				\right]+
				\frac{1}{\overline{J}_{i}^{\Delta t}}\sum_{\mu=2} ^{Q-1}\omega_{\mu} J_{i_{\mu}} h_{i_{\mu}},
			\end{aligned}
		\end{equation*}
		where we have used the facts ${h}
		_{i-\frac{1}{2}}^{+} = h_{i_{1}}$ and ${h}
		_{i+\frac{1}{2}}^{-} = h_{i_{Q}}$ for the adopted Gauss--Lobatto quadrature.   
		Under the condition \eqref{eq:1D_PP_condition} and the CFL condition \eqref{eq:1D_CFL_condition}, one obtains
		\begin{equation*}
			\overline{h}_{i}^{\Delta t}  \geqslant 0,
		\end{equation*}
		which means $\overline{\bm{U}}_{i}^{\Delta t} \in \mathcal{G}$. 
		The proof is completed.
	\end{proof}

	\section{Structure-preserving schemes for 2D SWEs}\label{Sec:2D_case}
	This section presents the high-order structure-preserving finite volume schemes for the 2D SWEs in curvilinear coordinates. In the 2D scenario, the SCLs in \eqref{eq:GCL_2D} become nontrivial and must be preserved under appropriate discretization.

	\subsection{Outline of 2D structure-preserving finite volume schemes on moving meshes}
	Denote the computational domain by $[0,1]\times[0,1]$. Assume a uniform mesh configuration, defined as follows:
	\begin{align*}
		&(\xi_1)_i =  i\Delta\xi_1, \quad i=0,1,\cdots,N_1-1, \quad \Delta\xi_1 = 1/(N_1-1),
		\\
		& (\xi_2)_j = j\Delta\xi_2,\quad j=0,1,\cdots,N_2-1, \quad \Delta\xi_2 = 1/(N_2-1). 
	\end{align*}
	Define the cell $I_{i,j} := \big[(\xi_1)_{i-\frac{1}{2}},(\xi_1)_{i+\frac{1}{2}}\big]\times\big[(\xi_2)_{j-\frac{1}{2}},(\xi_2)_{j+\frac{1}{2}}\big]$.

	Our semi-discrete finite volume schemes for \eqref{eq:2D_cur_eq} and the GCLs \eqref{eq:GCL_2D} are given by
	\begin{align}
		\frac{d(\overline{J\bm{U}})_{i,j}}{dt} =& 
		-\frac{1}{\Delta \xi_1}
		\left(
		\left(\bm{\widetilde{\mathcal{F}}}_1\right)_{i+\frac{1}{2},j} 
		-
		\left(\bm{\widetilde{\mathcal{F}}}_1\right)_{i-\frac{1}{2},j}\right)
		+\dfrac{1}{\Delta \xi_1}
		\sum\limits_{m=1}^2
		\left(\overline{s}_{m}\right)_{i,j}
		\left(
		{
			\left(
			\widetilde{\bm{\mathcal{R}}}_{m1}
			\right)_{i+\frac{1}{2},j}} 
		-
		{\left(
			\widetilde{\bm{\mathcal{R}}}_{m1}
			\right)_{i-\frac{1}{2},j}}
		\right)\nonumber\\
		& 
		-\frac{1}{\Delta \xi_2}
		\left(
		\left(\bm{\widetilde{\mathcal{F}}}_2\right)_{i,j+\frac{1}{2}} 
		-
		\left(\bm{\widetilde{\mathcal{F}}}_2\right)_{i,j-\frac{1}{2}}\right)
		+\dfrac{1}{\Delta \xi_2}
		\sum\limits_{m=1}^2
		\left(\overline{s}_{m}\right)_{i,j}
		\left(
		{
			\left(
			\widetilde{\bm{\mathcal{R}}}_{m2}
			\right)_{i,j+\frac{1}{2}}} 
		-
		{\left(
			\widetilde{\bm{\mathcal{R}}}_{m2}
			\right)_{i,j-\frac{1}{2}}}
		\right)
		\nonumber\\
		& +
		\dfrac{1}{\Delta \xi_1 \Delta \xi_2}
		\int_{I_{i,j}} 
		\sum\limits_{m=1}^2
		\left(
		s_m
		-\left(
		\overline{s}_{m}
		\right)_{i,j}
		\right)
		\nabla _{\bm{\xi}} \cdot \bm{\mathcal{R}}_m,  \label{eq:2D_dis_U}\\
		\frac{d\overline{J}_{i,j}}{dt} =& 
		-\frac{1}{\Delta \xi_1}
		\left(
		\left(
		\widetilde{J\frac{\partial \xi_1}{\partial t}}
		\right)_{i+\frac{1}{2},j} 
		-{\left(
			\widetilde{J\frac{\partial \xi_1}{\partial t}}
			\right)}_{i-\frac{1}{2},j}\right)
		-\frac{1}{\Delta \xi_2}
		\left(
		\left(
		\widetilde{J\frac{\partial \xi_2}{\partial t}}
		\right)_{i,j+\frac{1}{2}} 
		\left(
		\widetilde{J\frac{\partial \xi_2}{\partial t}}
		\right)_{i,j-\frac{1}{2}}\right)
		, 
		\label{eq:2D_dis_J}
		\\
		\frac{1}{
			\Delta\xi_{1}
		}
		&
		\left(
		\left(\widetilde{J \frac{\partial \xi_{1}} 
			{\partial x_{1}}
		}\right)
		_{i+\frac{1}{2},j}
		- \left(\widetilde{J \frac{\partial \xi_{1}} 
			{\partial x_{1}}
		}\right)
		_{i-\frac{1}{2},j}
		\right)
		+
		\frac{1}{
			\Delta\xi_{2}
		}
		\left(
		\left(\widetilde{J \frac{\partial \xi_{2}} 
			{\partial x_{1}}
		}
		\right)
		_{i,j+\frac{1}{2}}
		- \left(\widetilde{J \frac{\partial \xi_{2}} 
			{\partial x_{1}}
		}\right)
		_{i,j-\frac{1}{2}}
		\right)=0,\label{eq:2D_dis_SCL1}
		\\
		\frac{1}{
			\Delta\xi_{1}
		}
		&
		\left(
		\left(\widetilde{J \frac{\partial \xi_{1}} 
			{\partial x_{2}}
		}\right)
		_{i+\frac{1}{2},j}
		- \left(\widetilde{J \frac{\partial \xi_{1}} 
			{\partial x_{2}}
		}\right)
		_{i-\frac{1}{2},j}
		\right)
		+
		\frac{1}{
			\Delta\xi_{2}
		}
		\left(
		\left(\widetilde{J \frac{\partial \xi_{2}} 
			{\partial x_{2}}
		}
		\right)
		_{i,j+\frac{1}{2}}
		- \left(\widetilde{J \frac{\partial \xi_{2}} 
			{\partial x_{2}}
		}\right)
		_{i,j-\frac{1}{2}}
		\right)=0,\label{eq:2D_dis_SCL2}
	\end{align}
	where
	\begin{equation}\label{eq:2D_F_1_Flux}
		\left(
		\bm{
			\widetilde{\mathcal{F}}
		}_1
		\right)
		_{i\pm\frac{1}{2},j} 
		=   \sum_{\mu=1}^{Q}
		\omega_{\mu}
		\left[
		\left(
		{
			J\pd{\xi_1}{t}}
		\right)
		_{i\pm\frac{1}{2},j}^{\mu}
		\mean{\bm{U}}_{i\pm\frac{1}{2},j}^{\mu}
		- \frac{\alpha_1}{2}
		\jump{\bm{U}}_{i\pm\frac{1}{2},j}^{\mu} 
		+
		\left(
		L_{1}
		\right)
		_{i\pm\frac{1}{2},j}^{\mu}
		\left(
		\widetilde{\bm{F}}_{n_{1}}
		\right)
		_{i\pm\frac{1}{2},j}^{\mu}
		\right],
	\end{equation}
	\begin{equation}\label{eq:2D_F_2_Flux}
		\left(
		\bm{
			\widetilde{\mathcal{F}}
		}_2
		\right)
		_{i,j\pm\frac{1}{2}} 
		=   \sum_{\mu=1}^{Q}
		\omega_{\mu}
		\left[
		\left(
		{
			J\pd{\xi_2}{t}}
		\right)
		_{i,j\pm\frac{1}{2}}^{\mu}
		\mean{\bm{U}}_{i,j\pm\frac{1}{2}}^{\mu}
		- \frac{\alpha_2}{2}
		\jump{\bm{U}}_{i,j\pm\frac{1}{2}}^{\mu} 
		+
		\left(
		L_{2}
		\right)
		_{i,j\pm\frac{1}{2}}^{\mu}
		\left(
		\widetilde{\bm{F}}_{n_{2}}
		\right)
		_{i,j\pm\frac{1}{2}}^{\mu}
		\right],
	\end{equation}
	\begin{equation}\label{eq:2D_J_dis_flux}
		\left({J\frac{\partial \xi_1}{\partial t}}\right)_{i\pm\frac{1}{2},j}^{\mu} = 	\left(\widehat
		{J\frac{\partial \xi_1}{\partial t}}
		\right)_{i\pm\frac{1}{2},j}^{\mu}
		-\frac{(\beta_1)_{J}}{2}\jump{J}_{i\pm\frac{1}{2},j}^{\mu},
		~	
		\left({J\frac{\partial \xi_2}{\partial t}}\right)_{i,j\pm\frac{1}{2}}^{\mu} = 	\left(\widehat
		{J\frac{\partial \xi_2}{\partial t}}
		\right)_{i,j\pm\frac{1}{2}}^{\mu}
		-\frac{(\beta_2)_{J}}{2}\jump{J}_{i,j\pm\frac{1}{2}}^{\mu},
	\end{equation}
	\begin{equation}\label{eq:2D_J_dis}
		\begin{small}
			\left(
			\widetilde{J\frac{\partial \xi_1}{\partial t}}
			\right)
			_{i\pm\frac{1}{2},j} 
			= 
			\sum_{\mu=1}^{Q}
			\omega_{\mu}\left(
			{
				{J\frac{\partial \xi_1}{\partial t}}
			}
			\right)
			_{i\pm\frac{1}{2},j}^{\mu} ,
			~
			\left(
			\widetilde{J\frac{\partial \xi_2}{\partial t}}
			\right)
			_{i,j\pm\frac{1}{2}} 
			= 
			\sum_{\mu=1}^{Q}
			\omega_{\mu}
			\left(
			{J\frac{\partial \xi_2}{\partial t}}
			\right)
			_{i,j\pm\frac{1}{2}}^{\mu},
		\end{small}
	\end{equation}
	\begin{equation}
		\left(
		\widetilde{J\frac{\partial \xi_1}{\partial x_k}}
		\right)
		_{i\pm\frac{1}{2},j} 
		= 
		\sum_{\mu=1}^{Q}
		\omega_{\mu}
		\left(
		{J\frac{\partial \xi_1}{\partial x_k}}
		\right)
		_{i\pm\frac{1}{2},j}^{\mu},
		~
		\left(
		\widetilde{J\frac{\partial \xi_2}{\partial x_k}}
		\right)
		_{i,j\pm\frac{1}{2}} 
		= 
		\sum_{\mu=1}^{Q}
		\omega_{\mu}
		\left(
		{J\frac{\partial \xi_2}{\partial x_k}}
		\right)
		_{i,j\pm\frac{1}{2}}^{\mu},\label{eq:2D_dis_SCL_mesh_metrics}
	\end{equation}
	\begin{equation}\label{eq:2D_R_1_Flux}
		\left(
		\bm{
			\widetilde{\mathcal{R}}
		}_{m1}
		\right)
		_{i\pm\frac{1}{2},j} 
		=   \sum_{\mu=1}^{Q}
		\omega_{\mu}
		\left[
		\left(
		L_{1}
		\right)
		_{i\pm\frac{1}{2},j}^{\mu}
		\left(
		\left(
		\widetilde{\bm{r}}_m
		\right)_{n_1}
		\right)
		_{i\pm\frac{1}{2},j}^{\mu}
		\right],
	\end{equation}
	\begin{equation}\label{eq:2D_R_2_Flux}
		\left(
		\bm{
			\widetilde{\mathcal{R}}
		}_{m2}
		\right)
		_{i,j\pm\frac{1}{2}} 
		=   \sum_{\mu=1}^{Q}
		\omega_{\mu}
		\left[
		\left(
		L_{2}
		\right)
		_{i,j\pm\frac{1}{2}}^{\mu}
		\left(
		\left(
		\widetilde{\bm{r}}_m
		\right)_{n_2}
		\right)
		_{i,j\pm\frac{1}{2}}^{\mu}
		\right]. 
	\end{equation}
	Here, $(\cdot)_{i\pm\frac{1}{2},j}^{\mu}$ and $(\cdot)_{i,j\pm\frac{1}{2}}^{\mu}$  represent the approximations at the quadrature points $\left((\xi_1)_{i\pm\frac{1}{2}}, (\xi_2)_{j_\mu}\right)$ and $\left((\xi_1)_{i_{\mu}}, (\xi_2)_{j\pm\frac{1}{2}}\right)$, respectively, with the corresponding quadrature weights denoted as $\{ \omega_{\mu}\}_{\mu=1}^Q$. The procedures of obtaining $\bm{U}_{i\pm\frac{1}{2},j}^{\pm,\mu}$ and $\bm{U}_{i,j\pm\frac{1}{2}}^{\pm,\mu}$ will be given in Subsections \ref{Sec:Qua_Recon}. 
		The parameters $\alpha_{1}$, $\alpha_{2}$, $(\beta_1)_{J}$ and $(\beta_2)_{J}$ in \eqref{eq:2D_F_1_Flux}, \eqref{eq:2D_F_2_Flux} and \eqref{eq:2D_J_dis_flux} are defined as 
		\begin{equation}\label{eq:alpha_def}
			\begin{aligned}
				&\alpha_{1} = \max_{i,j,\mu}
				\left\{
				\left | 
				\left(
				J\pd{\xi_1}{t}
				\right)
				_{i\pm\frac{1}{2},j}^{\mu}
				\right |\right\},~
				\alpha_{2} = \max_{i,j,\mu}
				\left\{
				\left | 
				\left({
				J\pd{\xi_2}{t}
			}\right)
				_{i,j\pm\frac{1}{2}}^{\mu}
				\right| \right\},
				\\
				&(\beta_{1})_{J} = \max_{i,j,\mu}
				\left\{
				\left | 
				\left(\widehat{
				J\pd{\xi_1}{t}}/J
				\right)
				_{i\pm\frac{1}{2},j}^{\mu}
				\right |\right\},~
				(\beta_{2})_{J} = \max_{i,j,\mu}
				\left\{
				\left | 
				\left(\widehat{
				J\pd{\xi_2}{t}}/J
				\right)
				_{i,j\pm\frac{1}{2}}^{\mu}
				\right| \right\}.
			\end{aligned}
		\end{equation}
	The computations of
	$
	\displaystyle
	\left(\widehat{J \frac{\partial \xi_\ell}{\partial t}}\right)_{i\pm\frac{1}{2},j}^{\mu}$, $
	\displaystyle
	\left(\widehat{J \frac{\partial \xi_\ell}{\partial t}}\right)_{i,j\pm\frac{1}{2}}^{\mu}$, $
	\displaystyle
	\left(J \frac{\partial \xi_\ell}{\partial x_k}\right)_{i\pm\frac{1}{2},j}^\mu
	$, and $
	\displaystyle
	\left(J \frac{\partial \xi_\ell}{\partial x_k}\right)_{i,j\pm\frac{1}{2}}^\mu
	$ will be given in Section \ref{Sec:Mesh_coeff}.
	
	The integral term in \eqref{eq:2D_dis_U} should be approximated by a proper quadrature rule on $I_{i,j}$:
	\begin{equation}\label{eq:2D_Source_Term_Dis}
		\dfrac{1}{\Delta \xi_1 \Delta \xi_2}
		\int_{I_{i,j}} 
		\sum_{m=1}^{2}
		\left(
		s_m
		-\left(
		\overline{s}_{m}
		\right)_{i,j}
		\right)
		\nabla _{\bm{\xi}} \cdot \bm{\mathcal{R}}_m\approx
		\sum_{\mu,\nu= 1}^{Q} \omega_{\mu,\nu}\sum_{m=1}^{2}\left((s_{m})_{i_{\mu},j_{\nu}} - (\overline{s}_m)_{i,j}\right)\left[\nabla _{\bm{\xi}} \cdot \bm{\mathcal{R}}_m\right]_{i_{\mu},j_{\nu}},
	\end{equation}
	where $(\cdot)_{i_{\mu},j_{\nu}}$ denotes the approximate values at the quadrature point  $\left((\xi_1)_{i_{\mu}},(\xi_2)_{j_{\nu}}\right)$, and $\{\omega_{\mu,\nu}\}$ denote the corresponding quadrature weights. 
	For the WB consideration, the numerical fluxes 	$\left(\widetilde{\bm{F}}_{n_1}\right)_{i+\frac{1}{2},j}^{\mu}$ and $\left(
	\bm{
		\widetilde{\mathcal{R}}
	}_{m1}
	\right)
	_{i+\frac{1}{2},j}$  in equations \eqref{eq:2D_F_2_Flux} and \eqref{eq:2D_R_1_Flux}  can be defined based on the hydrostatic reconstruction \cite{AuDusse2004fast}  as
	\begin{equation*}
		\begin{aligned}
			&
			\left(\widetilde{\bm{F}}_{n_1}\right)_{i+\frac{1}{2},j}^{\mu}
			= \left(\widetilde{\bm{F}}_{n_1}\right)_{i+\frac{1}{2},j}^{L,\mu}
			= \langle
			\left(
			\bm{n}_1
			\right)
			_{i+\frac{1}{2},j}^{\mu}
			,
			\widetilde{\bm{F}}
			_{i+\frac{1}{2},j}^{L,\mu}
			\rangle
			-
			\frac{\beta}{2}
			(
			\bm{U}_{i+\frac{1}{2},j}^{+,\mu,*}
			-
			\bm{U}_{i+\frac{1}{2},j}^{-,\mu,*}
			),
			\\
			&
			\left(
			\bm{
				\widetilde{\mathcal{R}}
			}_{m1}
			\right)
			_{i+\frac{1}{2},j} 
			=   \sum_{\mu=1}^{Q}
			\omega_{\mu}
			\left[
			\left(
			L_{1}
			\right)
			_{i+\frac{1}{2},j}^{\mu}
			\left(
			\left(
			\widetilde{\bm{r}}_m
			\right)_{n_1}
			\right)
			_{i+\frac{1}{2},j}^{\mu}
			\right],\\
			&
			\left(
			\left(
			\widetilde{\bm{r}}_m
			\right)_{n_1}
			\right)
			_{i+\frac{1}{2},j}^{\mu} = 
			\langle
			\left(
			\bm{n}_1
			\right)
			_{i+\frac{1}{2},j}^{\mu}
			,\left(
			{\bm{r}}_m
			\right)_{i+\frac{1}{2},j}^{-,\mu}
			\rangle,
		\end{aligned}
	\end{equation*}
	with 
	\begin{equation*}
		\left(\widetilde{\bm{F}}\right)
		_{i+\frac{1}{2},j}^{L,\mu}
		= \left(
		\left(
		\bm{\widetilde{F}}_1
		\right)
		_{i+\frac{1}{2},j}^{L,\mu}
		, 
		\left(
		\bm{\widetilde{F}}_2
		\right)
		_{i+\frac{1}{2},j}^{L,\mu}
		\right)^{\top}, \qquad  \beta = \max_{i,j}\left(\abs{\overline{\bm{v}}_{i,j}}+\sqrt{g\overline{h}_{i,j}}\right),
	\end{equation*}
	where
	\begin{equation*}
		\left(
		\widetilde{\bm{F}}_1
		\right)
		_{i+\frac{1}{2},j}^{L,\mu}    
		= \dfrac{1}{2}
		\left(
		{\bm{F}}_1
		(
		\bm{U}_{i+\frac{1}{2},j}^{-,\mu,*}
		)
		+
		{\bm{F}}_1
		(
		\bm{U}_{i+\frac{1}{2},j}^{+,\mu,*}
		)
		\right)
		+ \left(\begin{array}{cc}
			0  \\
			\frac{g}{2}\left(h_{i+\frac{1}{2},j}^{-,\mu}\right)^2-\frac{g}{2}
			\left(h_{i+\frac{1}{2},j}
			^{-,\mu,*}
			\right)^2 \\
			0
		\end{array}
		\right),
	\end{equation*}
	\begin{equation*}
		\left(
		\widetilde{\bm{F}}_2
		\right)
		_{i+\frac{1}{2},j}^{L,\mu}    
		= \dfrac{1}{2}
		\left(
		{\bm{F}}_2
		(
		\bm{U}_{i+\frac{1}{2},j}^{-,\mu,*}
		)
		+
		{\bm{F}}_2
		(
		\bm{U}_{i+\frac{1}{2},j}^{+,\mu,*}
		)
		\right)
		+ \left(\begin{array}{cc}
			0  \\
			0  \\
			\frac{g}{2}\left(h_{i+\frac{1}{2},j}^{-,\mu}\right)^2-\frac{g}{2}
			\left(h_{i+\frac{1}{2},j}
			^{-,\mu,*}
			\right)^2  
		\end{array}
		\right),
	\end{equation*}
	\begin{equation*}
		\bm{U}_{i+\frac{1}{2},j}^{\pm,\mu,*} = 
		\left(\begin{array}{cc}
			h_{i+\frac{1}{2},j}^{\pm,\mu,*} \\
			h_{i+\frac{1}{2},j}^{\pm,\mu,*}(v_1)_{i+\frac{1}{2},j}^{\pm,\mu}
			\\
			h_{i+\frac{1}{2},j}^{\pm,\mu,*}(v_2)_{i+\frac{1}{2},j}^{\pm,\mu}
		\end{array}\right), \quad 
		h_{i+\frac{1}{2},j}^{\pm,\mu,*} 
		= \max
		\left(
		0,  h_{i+\frac{1}{2},j}^{\pm,\mu}+b_{i+\frac{1}{2},j}^{\pm,\mu} - \max
		\left(
		b_{i+\frac{1}{2},j}^{+,\mu},
		b_{i+\frac{1}{2},j}^{-,\mu}
		\right)
		\right).
	\end{equation*}
	The numerical fluxes in \eqref{eq:2D_F_1_Flux}-\eqref{eq:2D_R_2_Flux} maintain 
	\begin{equation}\label{eq:2D_condition}
		\begin{aligned}
			&\left(
			\widetilde{\bm{F}}_{n_{1}}
			\right)
			_{i\pm\frac{1}{2},j}^{\mu} - \sum\limits_{m=1}^2(\overline{s}_m)_{i,j}
			\left(
			\left(
			\widetilde{\bm{r}}_m
			\right)_{n_1}
			\right)
			_{i\pm\frac{1}{2},j}^{\mu}  = 
			\langle
			\left(
			\bm{n}_1
			\right)
			_{i\pm\frac{1}{2},j}^{\mu}
			,\bm{C}\rangle,\\
			&\left(
			\widetilde{\bm{F}}_{n_{2}}
			\right)
			_{i,j\pm\frac{1}{2}}^{\mu} - \sum\limits_{m=1}^2(\overline{s}_m)_{i,j}
			\left(
			\left(
			\widetilde{\bm{r}}_m
			\right)_{n_2}
			\right)
			_{i,j\pm\frac{1}{2}}^{\mu}  = 
			\langle
			\left(
			\bm{n}_2
			\right)
			_{i,j\pm\frac{1}{2}}^{\mu}
			,\bm{C}\rangle,
		\end{aligned}
	\end{equation}
	when a steady state is reached.
	\ref{section:Appendix} details the fluxes $\left(\widetilde{\bm{F}}_{n_1}\right)_{i-\frac{1}{2},j}^{\mu}$, $\left(
	\bm{
		\widetilde{\mathcal{R}}
	}_{m1}
	\right)
	_{i-\frac{1}{2},j}$, $\left(\widetilde{\bm{F}}_{n_2}\right)_{i,j\pm\frac{1}{2}}^{\mu}$, $\left(
	\bm{
		\widetilde{\mathcal{R}}
	}_{m2}
	\right)
	_{i,j\pm\frac{1}{2}}$ and the proof for the condition \eqref{eq:2D_condition}.
	As discussed in Section \ref{sec:22}, the governing equations of $\bm{w}$ and $\bm{q}$ have been expressed as
	\begin{equation*}
		\pd{J\bm{w}}{t}
		+\sum_{\ell = 1}^{2}
		\pd{}{\xi_{\ell}}
		\left(
		J\pd{\xi_{\ell}}{t} \bm{w}
		\right) = 0, \quad	\pd{J\bm{q}}{t}
		+\sum_{\ell = 1}^{2}
		\pd{}{\xi_{\ell}}
		\left(
		J\pd{\xi_{\ell}}{t} \bm{q}
		\right) = 0.
	\end{equation*}
	Our semi-discrete schemes for  $J\bm{w}$ and $J\bm{q}$ are given by 
	\begin{align}
		\frac{d(\overline{J\bm{w}})_{i,j}}{dt} =& 
		-\frac{1}{\Delta \xi_1}
		\left(
		\widetilde{\bm{w}}_{i+\frac{1}{2},j} 
		-\widetilde{\bm{w}}_{i-\frac{1}{2},j}\right)
		-\frac{1}{\Delta \xi_2}
		\left(
		\widetilde{\bm{w}}_{i,j+\frac{1}{2}} 
		-\widetilde{\bm{w}}_{i,j-\frac{1}{2}}\right)
		, 
		\label{eq:2D_dis_p}\\
		\frac{d(\overline{J\bm{q}})_{i,j}}{dt} =& 
		-\frac{1}{\Delta \xi_1}
		\left(
		\widetilde{\bm{q}}_{i+\frac{1}{2},j} 
		-
		\widetilde{\bm{q}}_{i-\frac{1}{2},j}\right)
		-\frac{1}{\Delta \xi_2}
		\left(
		\widetilde{\bm{q}}_{i,j+\frac{1}{2}} 
		-
		\widetilde{\bm{q}}_{i,j-\frac{1}{2}}
		\right),  \label{eq:2D_dis_q}
	\end{align}
	with 
	\begin{equation}\label{eq:2D_flux_pq}
		\begin{aligned}
			&\widetilde{\bm{w}}_{i\pm\frac{1}{2},j} = \sum_{\mu=1}^{Q} \omega_{\mu}
			\left[
			\left(
			J\pd{\xi}{t}
			\right)_{i\pm\frac{1}{2},j}^{\mu}
			\mean{\bm{w}}_{i\pm\frac{1}{2},j}^{\mu}  
			- \frac{\alpha_1}{2}
			\jump{\bm{w}}_{i\pm\frac{1}{2},j}^{\mu}
			\right],\\
			&\widetilde{\bm{w}}_{i,j\pm\frac{1}{2}} = \sum_{\mu=1}^{Q} \omega_{\mu}
			\left[
			\left(
			J\pd{\xi}{t}
			\right)_{i,j\pm\frac{1}{2}}^{\mu}
			\mean{\bm{w}}_{i,j\pm\frac{1}{2}}^{\mu}  
			- \frac{\alpha_2}{2}
			\jump{\bm{w}}_{i,j\pm\frac{1}{2}}^{\mu}
			\right],\\
			&\widetilde{\bm{q}}_{i\pm\frac{1}{2},j} = \sum_{\mu=1}^{Q} \omega_{\mu}
			\left[
			\left(
			J\pd{\xi}{t}
			\right)_{i\pm\frac{1}{2},j}^{\mu}
			\mean{\bm{q}}_{i\pm\frac{1}{2},j}^{\mu}  
			- \frac{\alpha_1}{2}
			\jump{\bm{q}}_{i\pm\frac{1}{2},j}^{\mu}
			\right],
			\\
			&\widetilde{\bm{q}}_{i,j\pm\frac{1}{2}} = \sum_{\mu=1}^{Q} \omega_{\mu}
			\left[
			\left(
			J\pd{\xi}{t}
			\right)_{i,j\pm\frac{1}{2}}^{\mu}
			\mean{\bm{q}}_{i,j\pm\frac{1}{2}}^{\mu}  
			- \frac{\alpha_2}{2}
			\jump{\bm{q}}_{i,j\pm\frac{1}{2}}^{\mu}
			\right].
		\end{aligned}
	\end{equation}
	The cell averages of $\bm{U}$, $\bm{w}$, and $\bm{q}$ in the physical domain are then computed by 
	\begin{equation*}
		\overline{\bm{U}}_{i,j} := 
		\frac
		{\overline{J\bm{U}}_{i,j}}{\overline{J}_{i,j}} = \frac{\int_{I_{i,j}}J\bm{U}\mathrm{d}\bm{\xi}}{\int_{I_{i,j}}J\mathrm{d}\bm{\xi}}, \quad 
		\overline{\bm{w}}_{i,j} := 
		\frac
		{\overline{J\bm{w}}_{i,j}}
		{\overline{J}_{i,j}}=
		\frac{\int_{I_{i,j}}
			J\bm{w}
			\mathrm{d}\bm{\xi}}
		{\int_{I_{i,j}}
			J\mathrm{d}\bm{\xi}}, \quad 
		\overline{\bm{q}}_{i,j} := 
		\frac
		{\overline{J\bm{q}}_{i,j}}
		{\overline{J}_{i,j}}
		=
		\frac{\int_{I_{i,j}}
			J\bm{q}
			\mathrm{d}\bm{\xi}}
		{\int_{I_{i,j}}
			J\mathrm{d}\bm{\xi}}.
	\end{equation*}
	One can use the numerical quadrature rules to calculate the integral above.
	\subsection{Discretization of mesh metrics}\label{Sec:Mesh_coeff}

	This subsection first details the computations of $\left\{ \left(J \dfrac{\partial \xi_\ell}{\partial x_k}\right)_{i_{\mu},j_{\nu}} \right\}$ at the quadrature points $\left\{\left((\xi_1)_{i_{\mu}},(\xi_2)_{j_\nu}\right)\right\}$ for the calculations in  \eqref{eq:2D_Source_Term_Dis}. 
	We then give the high-order approximations of $
	\displaystyle
	\left(\widehat{J \frac{\partial \xi_\ell}{\partial t}}\right)_{i\pm\frac{1}{2},j}^{\mu}$, $
	\displaystyle
	\left(\widehat{J \frac{\partial \xi_\ell}{\partial t}}\right)_{i,j\pm\frac{1}{2}}^{\mu}$, $
	\displaystyle
	\left(J \frac{\partial \xi_\ell}{\partial x_k}\right)_{i\pm\frac{1}{2},j}^\mu
	$, and $
	\displaystyle
	\left(J \frac{\partial \xi_\ell}{\partial x_k}\right)_{i,j\pm\frac{1}{2}}^\mu
	$ at the quadrature points $\left((\xi_1)_{i\pm\frac{1}{2}},(\xi_2)_{j_\mu}\right)$ and $\left((\xi_1)_{i_\mu},(\xi_2)_{j\pm\frac{1}{2}}\right)$.

	The high-order approximation for the mesh metrics satisfying the discrete version of the SCLs can be constructed by using the method proposed in \cite{Xu_freestream}. Note that the mesh metrics satisfy the following relations
	\begin{equation}\label{eq:Mesh_metrics_relation}
		\begin{aligned}
			&J \frac{\partial \xi_1}{\partial x_1} = \frac{\partial x_2}{\partial \xi_2},\quad\,\,\,\,
			J \frac{\partial \xi_1}{\partial x_2} = -\frac{\partial x_1}{\partial \xi_2},\\
			&J \frac{\partial \xi_2}{\partial x_1} = -\frac{\partial x_2}{\partial \xi_1},\quad
			J \frac{\partial \xi_2}{\partial x_2} = \frac{\partial x_1}{\partial \xi_1}.\\
		\end{aligned}
	\end{equation}
	The values $\displaystyle \left(J \frac{\partial \xi_\ell}{\partial x_k}\right)_{i_{\mu},j_{\nu}} $
	can be obtained by using the polynomial interpolation based on the mesh point $\bm{x}$. 
	Take the fifth-order approximation to
	$\left(J\dfrac{\partial \xi_1}{\partial x_1}\right)_{i_{\mu},j_{\nu}}$ as an example. Notice that the $J \dfrac{\partial {\xi_1}}{\partial x_1} = \dfrac{\partial x_2}{\partial \xi_{2}}$. If one has a sixth-order approximation $\widehat{x}_2(\bm{\xi})$ to $x_2(\bm{\xi})$ on the cell $I_{i,j}$, then $\left(J\dfrac{\partial \xi_1}{\partial x_1}\right)_{i_{\mu},j_{\nu}}$ can be evaluated by $\dfrac{\partial \widehat{x}_2}{\partial \xi_2}\left((\xi_1)_{i_{\mu}},(\xi_2)_{j_{\nu}}\right)$.
	To obtain the values $\dfrac{\partial \widehat{x}_2}{\partial \xi_2}\left((\xi_1)_{i_{\mu}},(\xi_2)_{j_{\nu}}\right)$, we use the dimension-by-dimension approach. Let $\mathbb{P}^{5}$  denote the space of 1D  polynomials of degree up to five.  One can derive the polynomial $\widehat{x}_2(\xi_1)|_{{\mathcal E}_{i,j+\frac{1}{2}}} \in \mathbb{P}^{5}$ at the edge ${\mathcal E}_{i,j+\frac{1}{2}}$ connecting $\left((x_1)_{i-\frac{1}{2},j+\frac{1}{2}},(x_2)_{i-\frac{1}{2},j+\frac{1}{2}}\right)$ and $\left((x_1)_{i+\frac{1}{2},j+\frac{1}{2}},(x_2)_{i+\frac{1}{2},j+\frac{1}{2}}\right)$  by using a 1D polynomial interpolation. This gives the point values $\widehat{x}_2((\xi_1)_{i_{\mu},j+\frac{1}{2}})|_{{\mathcal E}_{i,j+\frac{1}{2}}}$ at $(\xi_1)_{i_{\mu},j+\frac{1}{2}}$. We interpolate these values $\left\{\widehat{x}_2((\xi_1)_{i_{\mu},j+\frac{1}{2}})|_{{\mathcal E}_{i,j+\frac{1}{2}}}\right\}$ to obtain
	$\widehat{x}_2(\xi_2)|_{{\mathcal E}_{i_{\mu},j}} \in \mathbb{P}^{5}$ at the edge ${\mathcal E}_{i_{\mu},j}$ connecting $\left((x_1)_{i_\mu,j-\frac{1}{2}},(x_2)_{i_{\mu},j-\frac{1}{2}}\right)$ and $\left((x_1)_{i_\mu,j+\frac{1}{2}},(x_2)_{i_{\mu},j+\frac{1}{2}}\right)$, which gives the point values $ \dfrac{\partial \widehat{x}_2}{\partial {\xi_2}}\left((\xi_2)_{j_{\nu}}\right)|_{{\mathcal E}_{i_{\mu},j}}$.
	Specifically, the 1D polynomial $\widehat{x}_2(\xi_1)|_{{\mathcal E}_{i,j+\frac{1}{2}}}$ can be obtained by using the stencil $\left\{(x_2)_{i\pm\frac{r}{2},j+\frac{1}{2}}\right\}_{r=1}^{3}$. Assume that $\widehat{x}_2(\xi_1)|_{{\mathcal E}_{i,j+\frac{1}{2}}}$ has the following form
	\begin{equation*}
		\begin{aligned}
			x_2(\xi_1)|_{{\mathcal E}_{i,j+\frac{1}{2}}} =&~a_0 \left(\xi_1-(\xi_1)_{i,j+\frac{1}{2}}\right)^5+a_1 \left(\xi_1-(\xi_1)_{i,j+\frac{1}{2}}\right)^4+a_2 \left(\xi_1-(\xi_1)_{i,j+\frac{1}{2}}\right)^3\\
			&+a_3 \left(\xi_1-(\xi_1)_{i,j+\frac{1}{2}}\right)^4+a_4 \left(\xi_1-(\xi_1)_{i,j+\frac{1}{2}}\right)
			+a_5.
		\end{aligned}
	\end{equation*}
	Using the interpolation condition $\widehat{x}_2((\xi_1)_{i\pm\frac{r}{2},j+\frac{1}{2}})|_{{\mathcal E}_{i,j+\frac{1}{2}}} = 
	\left(x_2\right)_{i\pm\frac{r}{2},j+\frac{1}{2}},~r = 1,2,3$, and solving the system of the linear equations, we obtain the  coefficients
	\begin{align*}
		a_0 = &-(10 \left(x_2\right)_{i-\frac{1}{2},j+\frac{1}{2}} - 5 \left(x_2\right)_{i-\frac{3}{2},j+\frac{1}{2}} +  \left(x_2\right)_{i-\frac{5}{2},j+\frac{1}{2}} - 10 \left(x_2\right)_{i+\frac{1}{2},j+\frac{1}{2}} + 5 \left(x_2\right)_{i+\frac{3}{2},j+\frac{1}{2}} \\
		&-  \left(x_2\right)_{i+\frac{5}{2},j+\frac{1}{2}})/(120(\Delta \xi_1) ^5),\\
		a_1 =& ~(2 \left(x_2\right)_{i-\frac{1}{2},j+\frac{1}{2}} - 3 \left(x_2\right)_{i-\frac{3}{2},j+\frac{1}{2}} +  \left(x_2\right)_{i-\frac{5}{2},j+\frac{1}{2}} + 2 \left(x_2\right)_{i+\frac{1}{2},j+\frac{1}{2}} - 3 \left(x_2\right)_{i+\frac{3}{2},j+\frac{1}{2}} \\
		&+  \left(x_2\right)_{i+\frac{5}{2},j+\frac{1}{2}})/(48(\Delta \xi_1)^4),\\
		a_2 =& ~(34 \left(x_2\right)_{i-\frac{1}{2},j+\frac{1}{2}} - 13 \left(x_2\right)_{i-\frac{3}{2},j+\frac{1}{2}} +  \left(x_2\right)_{i-\frac{5}{2},j+\frac{1}{2}} - 34 \left(x_2\right)_{i+\frac{1}{2},j+\frac{1}{2}} + 13 \left(x_2\right)_{i+\frac{3}{2},j+\frac{1}{2}} \\
		&-  \left(x_2\right)_{i+\frac{5}{2},j+\frac{1}{2}})/(48(\Delta \xi_1)^3),\\
		a_3 =& -(34 \left(x_2\right)_{i-\frac{1}{2},j+\frac{1}{2}} - 39 \left(x_2\right)_{i-\frac{3}{2},j+\frac{1}{2}} + 5 \left(x_2\right)_{i-\frac{5}{2},j+\frac{1}{2}} + 34 \left(x_2\right)_{i+\frac{1}{2},j+\frac{1}{2}} - 39 \left(x_2\right)_{i+\frac{3}{2},j+\frac{1}{2}} \\
		&+ 5 \left(x_2\right)_{i+\frac{5}{2},j+\frac{1}{2}})/(96(\Delta \xi_1)^2),\\
		a_4 = &-(2250 \left(x_2\right)_{i-\frac{1}{2},j+\frac{1}{2}} - 125 \left(x_2\right)_{i-\frac{3}{2},j+\frac{1}{2}} + 9 \left(x_2\right)_{i-\frac{5}{2},j+\frac{1}{2}} - 2250 \left(x_2\right)_{i+\frac{1}{2},j+\frac{1}{2}} + 125 \left(x_2\right)_{i+\frac{3}{2},j+\frac{1}{2}} \\
		& - 9 \left(x_2\right)_{i+\frac{5}{2},j+\frac{1}{2}})/1920(\Delta \xi_1),
		\\
		a_5 = &~(75 \left(x_2\right)_{i-\frac{1}{2},j+\frac{1}{2}})/128 - (25 \left(x_2\right)_{i-\frac{3}{2},j+\frac{1}{2}})/256 + (3 \left(x_2\right)_{i-\frac{5}{2},j+\frac{1}{2}})/256 + (75 \left(x_2\right)_{i+\frac{1}{2},j+\frac{1}{2}})/128 \\
		&- (25 \left(x_2\right)_{i+\frac{3}{2},j+\frac{1}{2}})/256
		+ (3 \left(x_2\right)_{i+\frac{5}{2},j+\frac{1}{2}})/256.
	\end{align*}
	It follows that $(x_2)_{i_{\mu},j+\frac{1}{2}} = \widehat{x}_2\left((\xi_1)_{i_\mu,j+\frac{1}{2}}\right)$ at the edge ${\mathcal E}_{i,j+\frac{1}{2}}$. The 1D polynomial $\widehat{x}_2(\xi_2)|_{{\mathcal E}_{i_{\mu},j}}$ 
	can be obtained by using the stencil $\left\{(x_2)_{i_\mu,j\pm\frac{r}{2}}\right\}_{r=1}^{3}$ in a similar way. Assume the polynomials can be represented as 
	\begin{equation*}
		\begin{aligned}
			\widehat{x}_2(\xi_2)|_{{\mathcal E}_{i_{\mu},j}} = &~\widetilde{a}_0 \left(\xi_2-(\xi_2)_{i_{\mu},j}\right)^5+\widetilde{a}_1 \left(\xi_2-(\xi_2)_{i_{\mu},j}\right)^4+\widetilde{a}_2 \left(\xi_2-(\xi_2)_{i_{\mu},j}\right)^3
			\\
			&+\widetilde{a}_3 \left(\xi_2-(\xi_2)_{i_{\mu},j}\right)^4+\widetilde{a}_4 \left(\xi_2-(\xi_2)_{i_{\mu},j}\right)
			+\widetilde{a}_5.
		\end{aligned}
	\end{equation*}
	Based on the interpolation conditions $\widehat{x}_2((\xi_2)_{i_{\mu},j+\frac{r}{2}})|_{{\mathcal E}_{i_{\mu},j}} = 
	\left(x_2\right)_{i_{\mu},j+\frac{r}{2}}$, the coefficients are computed as  
	\begin{align*}
		\widetilde{a}_0 = &-(10 \left(x_2\right)_{i_{\mu},j-\frac{1}{2}} - 5 \left(x_2\right)_{i_{\mu},j-\frac{3}{2}} +  \left(x_2\right)_{i_{\mu},j-\frac{5}{2}} - 10 \left(x_2\right)_{i_{\mu},j+\frac{1}{2}} + 5 \left(x_2\right)_{i_{\mu},j+\frac{3}{2}} \\
		&-  \left(x_2\right)_{i_{\mu},j+\frac{5}{2}})/(120(\Delta \xi_1) ^5),\\
		\widetilde{a}_1 =& ~(2 \left(x_2\right)_{i_{\mu},j-\frac{1}{2}} - 3 \left(x_2\right)_{i_{\mu},j-\frac{3}{2}} +  \left(x_2\right)_{i_{\mu},j-\frac{5}{2}} + 2 \left(x_2\right)_{i_{\mu},j+\frac{1}{2}} - 3 \left(x_2\right)_{i_{\mu},j+\frac{3}{2}} \\
		&+  \left(x_2\right)_{i_{\mu},j+\frac{5}{2}})/(48(\Delta \xi_1)^4),\\
		\widetilde{a}_2 =& ~(34 \left(x_2\right)_{i_{\mu},j-\frac{1}{2}} - 13 \left(x_2\right)_{i_{\mu},j-\frac{3}{2}} +  \left(x_2\right)_{i_{\mu},j-\frac{5}{2}} - 34 \left(x_2\right)_{i_{\mu},j+\frac{1}{2}} + 13 \left(x_2\right)_{i_{\mu},j+\frac{3}{2}} \\
		&-  \left(x_2\right)_{i_{\mu},j+\frac{5}{2}})/(48(\Delta \xi_1)^3),\\
		\widetilde{a}_3 =& -(34 \left(x_2\right)_{i_{\mu},j-\frac{1}{2}} - 39 \left(x_2\right)_{i_{\mu},j-\frac{3}{2}} + 5 \left(x_2\right)_{i_{\mu},j-\frac{5}{2}} + 34 \left(x_2\right)_{i_{\mu},j+\frac{1}{2}} - 39 \left(x_2\right)_{i_{\mu},j+\frac{3}{2}} \\
		&+ 5 \left(x_2\right)_{i_{\mu},j+\frac{5}{2}})/(96(\Delta \xi_1)^2),\\
		\widetilde{a}_4 = &-(2250 \left(x_2\right)_{i_{\mu},j-\frac{1}{2}} - 125 \left(x_2\right)_{i_{\mu},j-\frac{3}{2}} + 9 \left(x_2\right)_{i_{\mu},j-\frac{5}{2}} - 2250 \left(x_2\right)_{i_{\mu},j+\frac{1}{2}} + 125 \left(x_2\right)_{i_{\mu},j+\frac{3}{2}} \\
		& - 9 \left(x_2\right)_{i_{\mu},j+\frac{5}{2}})/1920(\Delta \xi_1),
		\\
		\widetilde{a}_5 = &~(75 \left(x_2\right)_{i_{\mu},j-\frac{1}{2}})/128 - (25 \left(x_2\right)_{i_{\mu},j-\frac{3}{2}})/256 + (3 \left(x_2\right)_{i_{\mu},j-\frac{5}{2}})/256 + (75 \left(x_2\right)_{i_{\mu},j+\frac{1}{2}})/128 \\
		&- (25 \left(x_2\right)_{i_{\mu},j+\frac{3}{2}})/256
		+ (3 \left(x_2\right)_{i_{\mu},j+\frac{5}{2}})/256.
	\end{align*}
	Using the relations \eqref{eq:Mesh_metrics_relation} gives $\left(J\dfrac{\partial \xi_1}{\partial x_1}\right)_{i_{\mu},j_{\nu}} = \dfrac{\partial \widehat{x}_2}{\partial {\xi_2}}\left((\xi_2)_{j_{\nu}}\right)|_{{\mathcal E}_{i_{\mu},j}}$. Similarly, we can obtain the mesh metrics $
	\displaystyle
	\left(J \frac{\partial \xi_\ell}{\partial x_k}\right)_{i\pm\frac{1}{2},j}^\mu
	$ and $
	\displaystyle
	\left(J \frac{\partial \xi_\ell}{\partial x_k}\right)_{i,j\pm\frac{1}{2}}^\mu
	$ in \eqref{eq:2D_dis_SCL_mesh_metrics}.
	As illustrated in \cite{Xu_freestream}, such approximations maintain the discrete SCLs.
	
	Based on the relation
	\begin{equation*}
		J\pd{\xi_{\ell} }{t} = -\sum_{k =1}^{2}\left(\pd{x_{k}}{t}\right)\left( J\pd{\xi_{\ell} }{x_{k}}\right),
	\end{equation*}
	the values of the temporal metric at the quadrature points can be evaluated by
	\begin{align}\label{eq:temporal_metric}
		\left(\widehat{J\pd{\xi_{k} }{t}}\right)_{i_\mu,j_\nu} 
		= -\sum_{\ell =1}^{d}\left[\left(\pd{x_{\ell}}{t}\right)\left( J\pd{\xi_{k} }{x_{\ell}}\right)\right]_{i_{\mu},j_{\nu}},
	\end{align}
	where 
	$\left(\pd{x_{\ell}}{t}\right)_{i_{\mu},j_{\nu}}$
	is the mesh velocity determined by the adaptive moving mesh strategy detailed in Section \ref{Sec:MM}. 
	
	\subsection{High-order 2D reconstruction}\label{Sec:Qua_Recon}
	This section discusses the 2D WENO reconstruction in the computational domain.
	For the computations in \eqref{eq:2D_Source_Term_Dis}, the high-order approximate values of $\bm{U}$ at $\left\{\left(\xi_1)_{i_{\mu}},(\xi_2)_{j_\nu}\right)\right\}$ are required. These values are obtained by generalizing the 1D reconstruction approach in Subsection \ref{Sec:1D_Qua_Recon} to the 2D case in a dimension-by-dimension fashion.

	Step 1. Starting from the cell average values $\overline{J\bm{U}}_{i,j}$ and $\overline{J}_{i,j}$, we apply the 1D WENO reconstruction to obtain the edge average values $\left\{\overline{J\bm{U}}_{i_{\mu},j}\right\}$ and $\left\{\overline{J}_{i_{\mu},j}\right\}$. For example, we perform 
	\begin{align*}
		\left\{\overline{J\bm{U}}_{i,j}\right\} \xrightarrow{\tt WENO} \left\{\overline{J\bm{U}}_{i_{\mu},j}\right\}~\text{for~fixed}~j,~~
		\left\{\overline{J}_{i,j}\right\}\xrightarrow{\tt WENO} \left\{\overline{J}_{i_{\mu},j}\right\}~\text{for~fixed}~j. 
	\end{align*}
	
	Step 2. Based on the edge values $\left\{\overline{J\bm{U}}_{i_{\mu},j}\right\}$ and $\left\{\overline{J}_{i_{\mu},j}\right\}$, we perform the 1D WENO reconstruction along the other direction to get the point values $\mathring{J\bm{U}}_{i_{\mu},j_{\nu}}$,  and $\mathring{J}_{i_{\mu},j_{\nu}}$, i.e.,
	\begin{align*}	
		\left\{\overline{J\bm{U}}_{i_{\mu},j}\right\} \xrightarrow{\tt WENO} \left\{ \mathring{J\bm{U}}_{i_{\mu},j_{\nu}} \right\},~~\left\{\overline{J}_{i_{\mu},j}\right\} \xrightarrow{\tt WENO} \left\{ \mathring{J}_{i_{\mu},j_{\nu}} \right\} ~\text{for~fixed}~i.
	\end{align*} 
	Then  $\mathring{\bm{U}}_{i_{\mu},j_{\nu}} = \mathring{J\bm{U}}_{i_{\mu},j_{\nu}} / \mathring{J}_{i_{\mu},j_{\nu}}$.  To ensure the cell average of $J$, we modify ${J}_{i_{\mu},j_{\nu}} = \overline{J}_{i} - \sum\limits_{\mu,\nu}\omega_{\mu,\nu} \mathring{J}_{i_{\mu},j_{\nu}}$ to obtain $\sum\limits_{\mu,\nu}\omega_{\mu,\nu} {J}_{i_{\mu},j_{\nu}} = \overline{J}_{i,j}$.
	
	Step 3. Define the approximate cell average over $I_{i,j}$ in the computational domain by  $$\widetilde{\bm{U}}_{i,j} = \sum \limits_{\mu,\nu=1}^{Q} \omega_{\mu,\nu} \mathring{\bm{U}}_{\mu,\nu},$$
	where $\omega_{\mu,\nu}$ is the corresponding quadrature weights.
	
	Step 4. Based on the cell averages $\left\{\widetilde{\bm{U}}_{i,j}\right\}$, employ the dimension-by-dimension WENO reconstruction to obtain the high-order approximate values
	$\{\check{\bm{U}}_{i_{\mu},j_{\nu}}\}$ at the quadrature points $\left\{\left(\left(\xi_{1}\right)_{i_\mu},\left(\xi_{2}\right)_{j_\nu}\right)\right\}$.
	Denote $\bm{\gamma}_{i,j} = \overline{J\bm{U}}_{i,j}-\sum \limits_{\mu,\nu}\omega_{\mu,\nu} J_{i_{\mu}, j_{\nu}} \check{\bm{U}}_{i_{\mu},j_{\nu}}$. We modify $\bm{U}_{i_{\mu},j_{\nu}} = \check{\bm{U}}_{i_{\mu},j_{\nu}}+\bm{\gamma}_{i,j}/J_{i_{\mu},j_{\nu}}$ to ensure that $\sum\limits_{\mu,\nu}\omega_{\mu,\nu} J_{i_{\mu},j_{\nu}}{\bm{U}}_{i_{\mu},j_{\nu}}=\overline{J\bm{U}}_{i,j}
	$.
	One can obtain the values $\bm{U}_{i+\frac{1}{2},j}^{-,\mu} = \bm{U}_{i_Q,j_{\mu}},~\bm{U}_{i-\frac{1}{2},j}^{+,\mu} = \bm{U}_{i_1,j_{\mu}}$,  $\bm{U}_{i,j+\frac{1}{2}}^{-,\mu} = \bm{U}_{i_{\mu},j_{Q}}$, and $\bm{U}_{i,j-\frac{1}{2}}^{+,\mu} = \bm{U}_{i_{\mu},j_{1}}$.
	As the 1D discussions in Section \ref{Sec:1D_case}, the coefficients utilized for the reconstructions in  $\mathring{J(h+b)}_{i_{\mu},j_{\nu}}$ and $\mathring{J}_{i_\mu,j_{\nu}}$ need to be the same, so that  ${\bm{U}}_{i_{\mu},j_{\nu}}+{\bm{w}}_{i_{\mu},j_{\nu}} = C{\bm{q}}_{i_{\mu},j_{\nu}}$.
	
		The condition 
		\begin{equation}\label{eq:Condition_PP_2D}
			J_{i_{\mu},j_{\nu}}>0,	\quad 
			h_{i_{\mu},j_{\nu}}\geqslant0, \quad \forall i,j,\mu,\nu,
		\end{equation}
		is important for the PP property of our 2D moving mesh finite volume schemes in Theorems \ref{Pro:SWE_PP_J} and \ref{Pro:SWE_PP}.
	 However, the WENO reconstruction procedure does not guarantee satisfaction of condition \eqref{eq:Condition_PP_2D}. As the 1D case, the following 2D PP limiter is performed to enforce this condition. 
	\begin{remark}[2D PP limiter]
		If $\overline{J}_{i,j}>0$ and $\overline{h}_{i,j}\geqslant 0 $, the 2D PP limiter modifies ${J}_{i_{\mu},j_{\nu}}$ and ${h}_{i_{\mu},j_{\nu}}$ to  
		\begin{equation*}
			\widehat{J}_{i_\mu,j_\nu}
			= \mathcal{\theta}_{J}
			\left({J}_{i_\mu,j_\nu}
			-\overline{J}_{i,j}\right)
			+\overline{J}_{i,j}
			,\quad
			\widehat{h}_{i_\mu,j_\nu}
			= \mathcal{\theta}_{h}
			\left({h}_{i_\mu,j_\nu}
			-\overline{h}_{i,j}\right)
			+\overline{h}_{i,j} \qquad \forall \mu, \nu
		\end{equation*}
		with
		\begin{equation*}
			\theta_{J}:= \begin{cases}
				~1, & \text{if}\quad \min\limits_{\mu,\nu} 
				\left\{
				{J}_{i_\mu,j_\nu}
				\right\} = \overline{J}_{i,j},\\
				\min 
				\left\{
				\left|
				\frac{
					\overline{J}_{i,j} - \epsilon_{0}}
				{\overline{J}_{i,j}-\min\limits_{\mu,\nu} 
					\left\{
					{J}_{i_\mu,j_\nu}
					\right\}}
				\right|, 
				1
				\right\}, & \text{otherwise},
			\end{cases}
		\end{equation*}
		\begin{equation*}
			\theta_{h}:= \begin{cases}
				~1, & \text{if}\quad \min\limits_{\mu,\nu} 
				\left\{
				{h}_{i_\mu,j_\nu}
				\right\} = \overline{h}_{i,j},\\
				\min 
				\left\{
				\left|
				\frac{
					\overline{h}_{i,j}}
				{\overline{h}_{i,j}-\min\limits_{\mu,\nu} 
					\left\{
					{h}_{i_\mu,j_\nu}
					\right\}}
				\right|, 
				1
				\right\}, & \text{otherwise}.
			\end{cases}
		\end{equation*}
		Here, $\epsilon_{0} = \min\{10^{-13},\overline{J}_{i,j}\}$. The condition \eqref{eq:Condition_PP_2D} is guaranteed after the PP limiter. When $J_{i_{\mu},j_{\nu}}>0$ and $ {h}_{i_\mu,j_\nu}\geqslant0 $, one has
		$\widehat{J}_{i_\mu,j_\nu} = {J}_{i_\mu,j_\nu}$
		and  $\widehat{h}_{i_\mu,j_\nu} = {h}_{i_\mu,j_\nu}$. We use $\widehat{J}_{i_\mu,j_\nu}$ to replace ${J}_{i_\mu,j_\nu}$ in Step 4 to ensure $\sum\limits_{\mu,\nu}\omega_{\mu,\nu} \widehat{J}_{i_{\mu},j_{\nu}}{\bm{U}}_{i_{\mu},j_{\nu}}=\overline{J\bm{U}}_{i,j}
		$.
		To maintain the WB property after performing the PP limiting procedure, the bottom topography is updated by
		\begin{equation*}
			\widehat{b}_{i_\mu,j_\nu} = {h}_{i_\mu,j_\nu} 
			+ {b}_{i_\mu,j_\nu} 
			- \widehat{h}_{i_\mu,j_\nu}.
		\end{equation*}
		In the computation, we employ $\widehat{\bm{U}}_{i_{\mu},j_{\nu}} = \left(\widehat{h}_{i_{\mu},j_{\nu}}, \left(hv_1\right)_{i_{\mu},j_{\nu}},\left(hv_2\right)_{i_{\mu},j_{\nu}}\right)^{\top}$, $\widehat{\bm{w}}_{i_{\mu},j_{\nu}} = \left(\widehat{b}_{i_{\mu},j_{\nu}},0,0\right)^{\top}$ and $J_{i_{\mu},j_{\nu}}$ to replace ${\bm{U}}_{i_{\mu},j_{\nu}}$, ${\bm{w}}_{i_{\mu},j_{\nu}}$ and $J_{i_{\mu},j_{\nu}}$, respectively.  This ensures 
		$\widehat{\bm{U}}_{i_{\mu},j_{\nu}}+\widehat{\bm{w}}_{i_{\mu},j_{\nu}}  = C\bm{q}_{i_{\mu},j_{\nu}}$. 
		It also preserves the conservative property.
	\end{remark}

	\subsection{Analysis of WB property}
	The subsection proves the WB property of our 2D schemes.
	\begin{theorem}\rm\label{2D:Prove_WB}
		Consider the semi-discrete schemes  \eqref{eq:2D_dis_U}--\eqref{eq:2D_dis_SCL2},  \eqref{eq:2D_dis_p}--\eqref{eq:2D_dis_q} with the numerical fluxes \eqref{eq:2D_F_1_Flux}--\eqref{eq:2D_R_2_Flux} and \eqref{eq:2D_flux_pq}. 		These schemes are WB, in the sense that, under the forward Euler or explicit SSP RK discretizations, when a steady state is reached at the $n$th time level $t^n$, i.e., 
		\begin{equation*}
			\overline{\bm{U}}_{i,j}^{n}+\overline{\bm{w}}_{i,j}^{n} = C\overline{\bm{q}}_{i,j}^{n}\quad\forall i,j,
		\end{equation*}
		the numerical solution at $t^{n+1}$ also satisfies 
		\begin{equation*}
			\overline{\bm{U}}_{i,j}^{n+1}+\overline{\bm{w}}_{i,j}^{n+1} = C\overline{\bm{q}}_{i,j}^{n+1}\quad\forall i,j. 
		\end{equation*}
	\end{theorem}
	\begin{proof}
		Since an SSP RK method is a convex combination of forward Euler method, we only need to show the conclusion for  the forward Euler discretization without loss of generality. 
		The discrete SCLs \eqref{eq:2D_dis_SCL1} and \eqref{eq:2D_dis_SCL2} give
		\begin{equation*}
			\sum
			\limits_{\mu=1}^{Q}
			\omega_{\mu}
			\left[
			\frac{1}{
				\Delta\xi_{1}
			}
			\left(
			\left(J \frac{\partial \xi_{1} }
			{\partial x_{1}}
			\right)
			_{i+\frac{1}{2},j}
			^{\mu} 
			- \left(
			J \frac{
				\partial \xi_{1} 
			}{
				\partial x_{1}
			}\right)
			_{i-\frac{1}{2},j}^{\mu}
			\right)
			+
			\frac{1}{
				\Delta\xi_{2}
			}
			\left(
			\left(J \frac{\partial \xi_{2} }
			{\partial x_{1}}
			\right)
			_{i,j+\frac{1}{2}}
			^{\mu} 
			- \left(
			J \frac{
				\partial \xi_{2} 
			}{
				\partial x_{1}
			}\right)
			_{i,j-\frac{1}{2}}^{\mu}
			\right)
			\right]=0,
		\end{equation*}
		\begin{equation*}
			\sum
			\limits_{\mu=1}^{Q}
			\omega_{\mu}
			\left[
			\frac{1}{
				\Delta\xi_{1}
			}
			\left(
			\left(J \frac{\partial \xi_{1} }
			{\partial x_{2}}
			\right)
			_{i+\frac{1}{2},j}
			^{\mu} 
			- \left(
			J \frac{
				\partial \xi_{1} 
			}{
				\partial x_{2}
			}\right)
			_{i-\frac{1}{2},j}^{\mu}
			\right)
			+
			\frac{1}{
				\Delta\xi_{2}
			}
			\left(
			\left(J \frac{\partial \xi_{2} }
			{\partial x_{2}}
			\right)
			_{i,j+\frac{1}{2}}
			^{\mu} 
			- \left(
			J \frac{
				\partial \xi_{2} 
			}{
				\partial x_{2}
			}\right)
			_{i,j-\frac{1}{2}}^{\mu}
			\right)
			\right]=0.
		\end{equation*}
		Summing the above two equations and using the definitions of $L_1$ and $L_2$ in \eqref{eq:L_N_Def}, we obtain the following equality
		\begin{align}
			&\sum
			\limits_{\mu=1}^{Q}
			\omega_{\mu} 
			\frac{1}{\Delta \xi_1}
			\left[
			\left(
			L_1
			\right)
			_{i+\frac{1}{2},j}
			^{\mu} 
			\langle
			\left(
			\bm{n}_1
			\right)
			_{i+\frac{1}{2},j}
			^{\mu}
			,\bm{C}
			\rangle
			- \left(
			L_1
			\right)
			_{i-\frac{1}{2},j}
			^{\mu} 
			\langle
			\left(
			\bm{n}_1
			\right)
			_{i-\frac{1}{2},j}
			^{\mu}
			,\bm{C}
			\rangle
			\right]\nonumber
			\\
			&+\sum
			\limits_{\mu=1}^{Q}
			\omega_{\mu} 
			\frac{1}{\Delta \xi_2}
			\left[
			\left(
			L_2
			\right)
			_{i,j+\frac{1}{2}}
			^{\mu} 
			\langle
			\left(
			\bm{n}_2
			\right)
			_{i,j+\frac{1}{2}}
			^{\mu}
			,\bm{C}
			\rangle
			- \left(
			L_2
			\right)
			_{i,j-\frac{1}{2}}
			^{\mu} 
			\langle
			\left(
			\bm{n}_2
			\right)
			_{i,j-\frac{1}{2}}
			^{\mu}
			,\bm{C}
			\rangle
			\right]=0.\label{eq:2D_SCLs_dis_C}
		\end{align}
		When a steady state is reached, the condition \eqref{eq:2D_condition} holds. Based on the condition \eqref{eq:2D_condition} and the equality \eqref{eq:2D_SCLs_dis_C}, 
		the discrete evolution equations \eqref{eq:2D_dis_U} of $J\bm{U}$  degenerate to
		\begin{align}\label{eq:2D_prove_U}
			\left(
			\overline{J\bm{U}}
			\right)_{i,j}^{n+1} 
			-
			\left(
			\overline{J\bm{U}}
			\right)_{i,j}^{n}
			=& 
			-\frac{\Delta t}{\Delta \xi_1}
			\left(
			\left(
			\bm{\widetilde{\mathcal{F}}}_1
			\right)_{i+\frac{1}{2},j} 
			-
			\left(
			\bm{\widetilde{\mathcal{F}}}_1
			\right)_{i-\frac{1}{2},j}
			\right)
			-\frac{\Delta t}{\Delta \xi_2}
			\left(
			\left(
			\bm{\widetilde{\mathcal{F}}}_2
			\right)_{i,j+\frac{1}{2}} 
			-
			\left(
			\bm{\widetilde{\mathcal{F}}}_2
			\right)_{i,j-\frac{1}{2}}
			\right)
		\end{align}
		with 
		\begin{align*}
			\left(
			\bm{
				\widetilde{\mathcal{F}}
			}_1
			\right)
			_{i+\frac{1}{2},j} 
			=  \sum
			\limits_{\mu=1}^{Q}
			\omega_{\mu}
			\left[
			\left(
			J\pd{\xi_1}{t}
			\right)
			_{i+\frac{1}{2},j}^{\mu}
			\mean{\bm{U}}_{i+\frac{1}{2},j}^{\mu}
			- 
			\frac{\alpha_1}{2}
			\jump{\bm{U}}_{i+\frac{1}{2},j}^{\mu}
			\right],
			\\
			\left(
			\bm{
				\widetilde{\mathcal{F}}
			}_1
			\right)
			_{i-\frac{1}{2},j} 
			=   \sum
			\limits_{\mu=1}^{Q}
			\omega_{\mu}
			\left[
			\left(
			J\pd{\xi_1}{t}
			\right)
			_{i-\frac{1}{2},j}^{\mu}
			\mean{\bm{U}}_{i-\frac{1}{2},j}^{\mu}
			- \frac{\alpha_1}{2}
			\jump{\bm{U}}_{i-\frac{1}{2},j}^{\mu}
			\right],
			\\
			\left(
			\bm{
				\widetilde{\mathcal{F}}
			}_2
			\right)
			_{i,j+\frac{1}{2}} 
			=   \sum
			\limits_{\mu=1}^{Q}
			\omega_{\mu}
			\left[
			\left(
			J\pd{\xi_2}{t}
			\right)
			_{i,j+\frac{1}{2}}^{\mu}
			\mean{\bm{U}}_{i,j+\frac{1}{2}}^{\mu}
			- \frac{\alpha_2}{2}
			\jump{\bm{U}}_{i,j+\frac{1}{2}}^{\mu}
			\right],
			\\
			\left(
			\bm{
				\widetilde{\mathcal{F}}
			}_2
			\right)
			_{i,j-\frac{1}{2}} 
			=   \sum
			\limits_{\mu=1}^{Q}
			\omega_{\mu}
			\left[
			\left(
			J\pd{\xi_2}{t}
			\right)
			_{i,j-\frac{1}{2}}^{\mu}
			\mean{\bm{U}}_{i,j-\frac{1}{2}}^{\mu}
			- \frac{\alpha_2}{2}
			\jump{\bm{U}}_{i,j-\frac{1}{2}}^{\mu}
			\right].
		\end{align*}
		The discrete evolution equations of $J\bm{w}$ and $J\bm{q}$ can be expressed as 
		\begin{align}
			(\overline{J\bm{w}})_{i,j}^{n+1}
			- (\overline{J\bm{w}})_{i,j}^{n}=& 
			-\frac{1}{\Delta \xi_1}
			\left(
			\widetilde{\bm{w}}_{i+\frac{1}{2},j} 
			-\widetilde{\bm{w}}_{i-\frac{1}{2},j}\right)
			-\frac{1}{\Delta \xi_2}
			\left(
			\widetilde{\bm{w}}_{i,j+\frac{1}{2}} 
			-\widetilde{\bm{w}}_{i,j-\frac{1}{2}}\right)
			, 
			\label{eq:2D_dis_proof_w}
			\\
			(\overline{J\bm{q}})_{i,j}^{n+1} - (\overline{J\bm{q}})_{i,j}^{n} =& 
			-\frac{1}{\Delta \xi_1}
			\left(
			\widetilde{\bm{q}}_{i+\frac{1}{2},j} 
			-
			\widetilde{\bm{q}}_{i-\frac{1}{2},j}\right)
			-\frac{1}{\Delta \xi_2}
			\left(
			\widetilde{\bm{q}}_{i,j+\frac{1}{2}} 
			-
			\widetilde{\bm{q}}_{i,j-\frac{1}{2}}
			\right).  \label{eq:2D_dis_proof_q}
		\end{align}
		Combining \eqref{eq:2D_prove_U} with \eqref{eq:2D_dis_proof_w} gives 
		\begin{align}
			\left(
			\overline{J\left(
				\bm{U}+\bm{w}
				\right)}
			\right)_{i,j}^{n+1} 
			-
			\left(
			\overline{J\left(
				\bm{U}+\bm{w}
				\right)}
			\right)_{i,j}^{n}
			=& 
			-\sum
			\limits_{\mu=1}^{Q} \omega_{\mu}\dfrac{\Delta t}{\Delta \xi_1}
			\Bigg(
			\left(
			J\pd{\xi_1}{t}
			\right)
			_{i+\frac{1}{2},j}^{\mu}
			\mean{\bm{U}+\bm{w}}_{i+\frac{1}{2},j}^{\mu}
			-\frac{\alpha_1}{2}
			\jump{\bm{U}+\bm{w}}_{i+\frac{1}{2},j}^{\mu}\nonumber
			\\
			&-\left(
			J\pd{\xi_1}{t}
			\right)
			_{i-\frac{1}{2},j}^{\mu}
			\mean{\bm{U}+\bm{w}}_{i-\frac{1}{2},j}^{\mu}
			+ \frac{\alpha_1}{2}
			\jump{\bm{U}+\bm{w}}_{i-\frac{1}{2},j}^{\mu}
			\Bigg)\nonumber
			\\
			& 
			-\sum
			\limits_{\mu=1}^{Q} \omega_{\mu}\dfrac{\Delta t}{\Delta \xi_2}
			\Bigg(
			\left(
			J\pd{\xi_2}{t}
			\right)
			_{i,j+\frac{1}{2}}^{\mu}
			\mean{\bm{U}+\bm{w}}_{i,j+\frac{1}{2}}^{\mu}
			- \frac{\alpha_2}{2}
			\jump{\bm{U}+\bm{w}}_{i,j+\frac{1}{2}}^{\mu}\nonumber
			\\
			&-\left(
			J\pd{\xi_2}{t}
			\right)
			_{i,j-\frac{1}{2}}^{\mu}
			\mean{\bm{U}+\bm{w}}_{i,j-\frac{1}{2}}^{\mu}
			+ \frac{\alpha_2}{2}
			\jump{\bm{U}+\bm{w}}_{i,j-\frac{1}{2}}^{\mu}
			\Bigg).\label{eq:2D_Prove_JUw}
		\end{align}
		The reconstruction procedure yields ${\bm{U}}_{i_{\mu},j_{\nu}}+{\bm{w}}_{i_{\mu},j_{\nu}} = C{\bm{q}}_{i_{\mu},j_{\nu}}$. This implies
		\begin{equation*}
			\mean{\bm{U}+\bm{w}} = C\mean{\bm{q}},\quad \jump{\bm{U}+\bm{w}} = C\jump{\bm{q}}.
		\end{equation*}
		Plugging the above relations into \eqref{eq:2D_Prove_JUw}, one gets 
		\begin{align*}
			\left(
			\overline{J\left(
				\bm{U}+\bm{w}
				\right)}
			\right)_{i,j}^{n+1} 
			-
			\left(
			\overline{JC\bm{q}}
			\right)_{i,j}^{n}
			=& 
			-\sum
			\limits_{\mu=1}^{Q} \omega_{\mu}\dfrac{\Delta t}{\Delta \xi_1}
			\Bigg(
			\left(
			J\pd{\xi_1}{t}
			\right)
			_{i+\frac{1}{2},j}^{\mu}
			\mean{C\bm{q}}_{i+\frac{1}{2},j}^{\mu}
			- \frac{\alpha_1}{2}
			\jump{C\bm{q}}_{i+\frac{1}{2},j}^{\mu}
			\\
			&-\left(
			J\pd{\xi_1}{t}
			\right)
			_{i-\frac{1}{2},j}^{\mu}
			\mean{C\bm{q}}_{i-\frac{1}{2},j}^{\mu}
			+ \frac{\alpha_1}{2}
			\jump{C\bm{q}}_{i-\frac{1}{2},j}^{\mu}
			\Bigg)
			\\
			& 
			-\sum
			\limits_{\mu=1}^{Q} \omega_{\mu}\dfrac{\Delta t}{\Delta \xi_2}
			\Bigg(
			\left(
			J\pd{\xi_2}{t}
			\right)
			_{i,j+\frac{1}{2}}^{\mu}
			\mean{C\bm{q}}_{i,j+\frac{1}{2}}^{\mu}
			- \frac{\alpha_2}{2}
			\jump{C\bm{q}}_{i,j+\frac{1}{2}}^{\mu}
			\\
			&-\left(
			J\pd{\xi_2}{t}
			\right)
			_{i,j-\frac{1}{2}}^{\mu}
			\mean{C\bm{q}}_{i,j-\frac{1}{2}}^{\mu}
			+ \frac{\alpha_2}{2}
			\jump{C\bm{q}}_{i,j-\frac{1}{2}}^{\mu}
			\Bigg).
		\end{align*}
		Substituting \eqref{eq:2D_dis_proof_q} into the above equation gives 
		\begin{equation*}
			\left(
			\overline{J\left(
				\bm{U}+\bm{w}
				\right)}
			\right)_{i,j}^{n+1} 
			= C\overline{J}_{i,j}^{n+1}\overline{\bm{q}}_{i,j}^{n+1},
		\end{equation*}
		which implies 
		\begin{equation*}
			\overline{
				\bm{U}}_{i,j}^{n+1}+\overline{\bm{w}}
			_{i,j}^{n+1} 
			= C\overline{\bm{q}}_{i,j}^{n+1}.
		\end{equation*}
		The proof is completed.	\end{proof}
	
	\subsection{Analysis of PP property}\label{Sec:2DPP}
	The semi-discrete schemes \eqref{eq:2D_dis_U} and \eqref{eq:2D_dis_J} coupled with the forward Euler time discretization can be represented as
	\begin{equation*}
		\overline{J\bm{U}}_{i,j}^{\Delta t}  = \overline{J\bm{U}}_{i,j}  - \Delta t\bm{L}_{i,j},\quad\overline{J}_{i,j}^{\Delta t}  = \overline{J}_{i,j}  - \Delta t{L}_{i,j},
	\end{equation*}
	where $\bm{L}_{i,j}$ and $L_{i,j}$ denote the right-hand-side terms of  \eqref{eq:2D_dis_U}  and \eqref{eq:2D_dis_J}, respectively. $\overline{J\bm{U}}_{i,j}$ and $\overline{J}_{i,j}$ denote the approximations at the $n$th time level. 
		\begin{theorem}\rm\label{Pro:SWE_PP_J}
			Assume that $\overline{J}_{i,j}>0$ for all $i$ and $j$. If 
			$
			J_{i_{\mu},j_{\nu}} > 0 
			$
			for all $i,j,\mu,\nu,$
			then $
			\overline{J}_{i,j}^{\Delta t}>0
			$
			holds under the  condition
			\begin{equation}\label{eq:CFL_condition_J}
				\Delta t \leqslant \dfrac{1}{2} \min
				\left\{
				\dfrac{
					\omega_1 
					\Delta \xi_1 
				}
				{\left(\beta_1\right)_{J}
				},
				\dfrac{ 
					\omega_1 
					\Delta \xi_2 
				}
				{\left(\beta_2\right)_{J}
				}
				\right\}.
			\end{equation}
		\end{theorem}
		The proof is omitted here, since it is similar to the proof of Theorems \ref{Pro:1D_SWE_PP_J} and \ref{Pro:SWE_PP}. Indeed, one can let $h=1, L_1 = 0, L_2 = 0$ in Theorem \ref{Pro:SWE_PP}, and utilize the similar arguements in Theorems \ref{Pro:1D_SWE_PP_J} and \ref{Pro:SWE_PP} to get the conclusion.
		\begin{remark}
			As discussed in Remark \ref{rm:1D_J_pp}, the mesh should not interleave or overlap during computation, i.e., if the mesh  satisfy $(x_1)_{i-\frac{1}{2},j+\frac{1}{2}}< (x_1)_{i+\frac{1}{2},j+\frac{1}{2}}$ for $i = 0, \dots, N_1-1$ and $(x_2)_{i+\frac{1}{2},j-\frac{1}{2}}< (x_2)_{i+\frac{1}{2},j+\frac{1}{2}}$ for $j = 0, \dots, N_2-1$, then for the new  mesh  $(x_1)_{i-\frac{1}{2},j+\frac{1}{2}}^{\Delta t}< (x_1)_{i+\frac{1}{2},j+\frac{1}{2}}^{\Delta t}$ and $(x_2)_{i+\frac{1}{2},j-\frac{1}{2}}^{\Delta t}< (x_2)_{i+\frac{1}{2},j+\frac{1}{2}}^{\Delta t}$ still holds.  The treatment will be detailed in Section \ref{Sec:MM}.
		\end{remark}
	\begin{theorem}\rm\label{Pro:SWE_PP}
		Assume that $\overline{\bm{U}}_{i,j}\in \mathcal{G}$ and $\overline{J}_{i,j}>0$  for all $i$ and $j$. If $\overline{J}_{i,j}^{\Delta t}>0$, and 
		\begin{equation}\label{eq:Condition_PP}
			J_{i_{\mu},j_{\nu}}>0,~	h_{i_{\mu},j_{\nu}}\geqslant0, \quad \forall i,j,\mu,\nu,
		\end{equation}  
		then the PP property
		\begin{equation*}
			\overline{\bm{U}}_{i,j}^{\Delta t}:= \overline{J\bm{U}}_{i,j}^{\Delta t}/\overline{J}_{i,j}^{\Delta t} \in \mathcal{G},
		\end{equation*}
		holds under the CFL condition
		\begin{equation}\label{eq:CFL_condition}
			\Delta t \leqslant \dfrac{1}{2} \min_{i,j,\mu}
			\left\{
			\dfrac{
				\omega_1 
				J_{i\pm\frac{1}{2},j}^{\mp,\mu}
				\Delta \xi_1 
			}
			{\alpha_1
				+(L_1)
				_{i\pm\frac{1}{2},j}^{\mu}
				\beta
			},
			\dfrac{ 
				\omega_1 
				J_{i,j\pm\frac{1}{2}}^{\mp,\mu}
				\Delta \xi_2 
			}
			{\alpha_2
				+(L_2)
				_{i,j\pm\frac{1}{2}}^{\mu}
				\beta
			}
			\right\}.
		\end{equation}
	\end{theorem}
	\begin{proof}
		Consider the discrete evolution equation for  $Jh$: 
		\begin{equation}\label{eq:2D_PP_h}
			\begin{aligned}
				\overline{Jh}_{i,j}^{\Delta t}  - 
				\overline{Jh}_{i,j}
				=&  
				-
				\frac{\Delta t}{\Delta \xi_{1}}
				\sum
				\limits_{\mu=1}^{Q}\omega_{\mu}\left[
				\left(J \pd{\xi_1}{t}\right)
				_{i+\frac{1}{2},j}^{\mu}
				\mean{h}
				_{i+\frac{1}{2},j}^{\mu}
				-\frac{\alpha_1}{2}
				\jump{h}
				_{i+\frac{1}{2},j}^{\mu}
				-
				\left(J \pd{\xi_1}{t}\right)
				_{i-\frac{1}{2},j}^{\mu}
				\mean{h}
				_{i-\frac{1}{2},j}^{\mu}
				+\frac{\alpha_1}{2}
				\jump{h}
				_{i-\frac{1}{2},j}^{\mu}
				\right]
				\\
				&-\frac{\Delta t}{\Delta \xi_{1}}
				\sum
				\limits_{\mu=1}^{Q}\omega_{\mu}
				\Bigg[
				\left(L_1\right)
				_{i+\frac{1}{2},j}^{\mu}
				\langle
				(\bm{n}_1)
				_{i+\frac{1}{2},j}^{\mu},
				\left(\bm{\widetilde{F}}
				_{i+\frac{1}{2},j}^{\mu}
				\right)_1
				\rangle
				-\left(L_1\right)
				_{i+\frac{1}{2},j}^{\mu}
				\frac{\beta}{2}
				\jump{h}
				_{i+\frac{1}{2},j}^{\mu,*}
				\\
				&\quad \quad \quad \quad \quad \quad -
				\left(L_1\right)
				_{i-\frac{1}{2},j}^{\mu}
				\langle
				(\bm{n}_1)
				_{i-\frac{1}{2},j}^{\mu},
				\left(\bm{\widetilde{F}}
				_{i-\frac{1}{2},j}^{\mu}
				\right)_1
				\rangle
				+
				\left(L_1\right)
				_{i-\frac{1}{2},j}^{\mu}
				\frac{\beta}{2}
				\jump{h}
				_{i-\frac{1}{2},j}^{\mu,*}
				\Bigg]
				\\
				&-
				\frac{\Delta t}{\Delta \xi_2}
				\sum
				\limits_{\mu=1}^{Q}\omega_{\mu}\left[
				\left(J \pd{\xi_2}{t}\right)
				_{i,j+\frac{1}{2}}^{\mu}
				\mean{h}
				_{i,j+\frac{1}{2}}^{\mu}
				-\frac{\alpha_2}{2}
				\jump{h}
				_{i,j+\frac{1}{2}}^{\mu}
				-
				\left(J \pd{\xi_2}{t}\right)
				_{i,j-\frac{1}{2}}^{\mu}
				\mean{h}
				_{i,j-\frac{1}{2}}^{\mu}
				+\frac{\alpha_2}{2}
				\jump{h}
				_{i,j-\frac{1}{2}}^{\mu}
				\right]
				\\
				&-\frac{\Delta t}{\Delta \xi_{2}}
				\sum
				\limits_{\mu=1}^{Q}\omega_{\mu}
				\Bigg[
				\left(L_2\right)
				_{i,j+\frac{1}{2}}^{\mu}
				\langle
				(\bm{n}_2)
				_{i,j+\frac{1}{2}}^{\mu},
				\left(\bm{\widetilde{F}}
				_{i,j+\frac{1}{2}}^{\mu}
				\right)_1
				\rangle
				-\left(L_2\right)
				_{i,j+\frac{1}{2}}^{\mu}
				\frac{\beta}{2}
				\jump{h}
				_{i,j+\frac{1}{2}}^{\mu,*}
				\\
				&\quad \quad \quad \quad\quad\quad -
				\left(L_2\right)
				_{i,j-\frac{1}{2}}^{\mu}
				\langle
				(\bm{n}_1)
				_{i,j-\frac{1}{2}}^{\mu},
				\left(\bm{\widetilde{F}}
				_{i,j-\frac{1}{2}}^{\mu}
				\right)_1
				\rangle
				+\left(L_2\right)
				_{i,j-\frac{1}{2}}^{\mu}
				\frac{\beta}{2}
				\jump{h}
				_{i,j-\frac{1}{2}}^{\mu,*}
				\Bigg].
			\end{aligned}
		\end{equation}
		Based on the relations 
		\begin{align*}
			&\pm
			\langle
			(\bm{n}_1)
			_{i\pm\frac{1}{2},j}^{\mu},
			\left(\bm{\widetilde{F}}
			_{i\pm\frac{1}{2},j}^{\mu}
			\right)_1
			\rangle  \leqslant 
			\beta \mean{h}
			_{i\pm\frac{1}{2},j}^{\mu,*}
			,~
			\pm
			\langle
			(\bm{n}_2)
			_{i,j\pm\frac{1}{2}}^{\mu},
			\left(\bm{\widetilde{F}}
			_{i,j\pm\frac{1}{2}}^{\mu}
			\right)_1
			\rangle  \leqslant 
			\beta \mean{h}
			_{i,j\pm\frac{1}{2}}^{\mu,*},
			\\
			& \frac{\alpha_1}{2}
			\jump{h}
			_{i+\frac{1}{2},j}^{\mu} 
			-\frac{\alpha_1}{2}
			\jump{h}
			_{i-\frac{1}{2},j}^{\mu}
			=
			\alpha_1
			\mean{h}
			_{i+\frac{1}{2},j}^{\mu} 
			+\alpha_1
			\mean{h}
			_{i-\frac{1}{2},j}^{\mu}
			-
			\alpha_1
			{h}
			_{i+\frac{1}{2},j}^{-,\mu} 
			-\alpha_1
			{h}
			_{i-\frac{1}{2},j}^{+,\mu},
			\\
			&   \frac{\alpha_2}{2}
			\jump{h}
			_{i,j+\frac{1}{2}}^{\mu} 
			-\frac{\alpha_2}{2}
			\jump{h}
			_{i,j-\frac{1}{2}}^{\mu}
			=
			\alpha_2
			\mean{h}
			_{i,j+\frac{1}{2}}^{\mu} 
			+\alpha_2
			\mean{h}
			_{i,j-\frac{1}{2}}^{\mu}
			-
			\alpha_2
			{h}
			_{i,j+\frac{1}{2}}^{-,\mu} 
			-\alpha_2
			{h}
			_{i,j-\frac{1}{2}}^{+,\mu},
		\end{align*}
		the lower bound of the right-hand side of \eqref{eq:2D_PP_h} can be estimated as
		\begin{equation}\label{eq:JHDelta}
			\overline{Jh}_{i,j}^{\Delta t}  - 
			\overline{Jh}_{i,j}
			\geqslant I_1+I_2+I_3+I_4
		\end{equation}
		with
		\begin{equation*}
			\begin{aligned}
				I_1 &= \frac{\Delta t}{\Delta \xi_1}
				\sum
				\limits_{\mu=1}^{Q}
				\omega_{\mu}
				\left[
				\left(
				\left(
				\alpha_1
				-
				J \pd{\xi_1}{t}
				\right)
				_{i+\frac{1}{2},j}^{\mu}
				\right)
				\mean{h}
				_{i+\frac{1}{2},j}^{\mu}
				+
				\left(
				\left(
				\alpha_1
				+
				J \pd{\xi_1}{t}
				\right)
				_{i-\frac{1}{2},j}^{\mu}
				\right)
				\mean{h}
				_{i-\frac{1}{2},j}^{\mu}
				\right],\\
				I_2 &= \frac{\Delta t}{\Delta \xi_2}
				\sum
				\limits_{\mu=1}^{Q}
				\omega_{\mu}
				\left[
				\left(
				\left(
				\alpha_2
				-
				J \pd{\xi_2}{t}
				\right)
				_{i,j+\frac{1}{2}}^{\mu}
				\right)
				\mean{h}
				_{i,j+\frac{1}{2}}^{\mu}
				+
				\left(
				\left(
				\alpha_2
				+
				J \pd{\xi_2}{t}
				\right)
				_{i,j-\frac{1}{2}}^{\mu} 
				\right)
				\mean{h}
				_{i,j-\frac{1}{2}}^{\mu}
				\right],
			\end{aligned}
		\end{equation*}
		\begin{equation*}
			\begin{aligned}
				I_3 &= -\frac{\Delta t}{\Delta \xi_1}
				\sum
				\limits_{\mu=1}^{Q}
				\omega_{\mu}
				\left[
				\left(
				L_1
				\right)
				_{i+\frac{1}{2},j}^{\mu}
				\beta{h}
				_{i+\frac{1}{2},j}^{-,\mu,*}
				+
				\left(
				L_1
				\right)
				_{i-\frac{1}{2},j}^{\mu}
				\beta{h}
				_{i-\frac{1}{2},j}^{+,\mu,*}
				\right]\\
				&\quad -\frac{\Delta t}{\Delta \xi_2}
				\sum
				\limits_{\mu=1}^{Q}
				\omega_{\mu}
				\left[
				\left(
				L_2
				\right)
				_{i,j+\frac{1}{2}}^{\mu}
				\beta{h}
				_{i,j+\frac{1}{2}}^{-,\mu,*}
				+
				\left(
				L_2
				\right)
				_{i,j-\frac{1}{2}}^{\mu}
				\beta{h}
				_{i,j-\frac{1}{2}}^{+,\mu,*}
				\right],\\
				I_4 &= -\frac{\Delta t}{\Delta \xi_1}
				\sum
				\limits_{\mu=1}^{Q}
				\omega_{\mu}
				\left[
				\alpha_1
				{h}
				_{i+\frac{1}{2},j}^{-,\mu}
				+
				\alpha_1
				{h}
				_{i-\frac{1}{2},j}^{+,\mu}
				\right] -\frac{\Delta t}{\Delta \xi_2}
				\sum
				\limits_{\mu=1}^{Q}
				\omega_{\mu}
				\left[
				\alpha_2
				{h}
				_{i,j+\frac{1}{2}}^{-,\mu}
				+
				\alpha_2
				{h}
				_{i,j-\frac{1}{2}}^{+,\mu}
				\right].\\
			\end{aligned}
		\end{equation*}
		The definitions of $\alpha_1$ and $\alpha_2$ in \eqref{eq:alpha_def} ensure $I_1\geqslant0$ and $I_2\geqslant0$. The remaining task is to estimate $I_3$ and $I_4$.
		Utilizing $h_{i,j\pm\frac{1}{2}}^{\mp,\mu,*}\leqslant h_{i,j\pm\frac{1}{2}}^{\mp,\mu}$ and $h_{i\pm\frac{1}{2},j}^{\mp,\mu,*}\leqslant h_{i\pm\frac{1}{2},j}^{\mp,\mu}$ gives 
		\begin{align}
			I_3+I_4 \geqslant
			&
			-\frac{\Delta t}{\Delta \xi_1}
			\sum
			\limits_{\mu=1}^{Q}
			\omega_{\mu}
			\left[
			\left(
			L_1
			\right)
			_{i+\frac{1}{2},j}^{\mu}
			\beta{h}
			_{i+\frac{1}{2},j}^{-,\mu}
			+
			\left(
			L_1
			\right)
			_{i-\frac{1}{2},j}^{\mu}
			\beta{h}
			_{i-\frac{1}{2},j}^{+,\mu}
			\right]\nonumber
			\\
			&-\frac{\Delta t}{\Delta \xi_2}
			\sum
			\limits_{\mu=1}^{Q}
			\omega_{\mu}
			\left[
			\left(
			L_2
			\right)
			_{i,j+\frac{1}{2}}^{\mu}
			\beta{h}
			_{i,j+\frac{1}{2}}^{-,\mu}
			+
			\left(
			L_2
			\right)
			_{i,j-\frac{1}{2}}^{\mu}
			\beta{h}
			_{i,j-\frac{1}{2}}^{+,\mu}
			\right]\nonumber
			\\
			&-\frac{\Delta t}{\Delta \xi_1}
			\sum
			\limits_{\mu=1}^{Q}
			\omega_{\mu}
			\left[
			\alpha_1
			{h}
			_{i+\frac{1}{2},j}^{-,\mu}
			+
			\alpha_1
			{h}
			_{i-\frac{1}{2},j}^{+,\mu}
			\right] -\frac{\Delta t}{\Delta \xi_2}
			\sum
			\limits_{\mu=1}^{Q}
			\omega_{\mu}
			\left[
			\alpha_2
			{h}
			_{i,j+\frac{1}{2}}^{-,\mu}
			+
			\alpha_2
			{h}
			_{i,j-\frac{1}{2}}^{+,\mu}
			\right]\nonumber
			\\
			=&
			-\frac{\Delta t}{\Delta \xi_1}
			\sum
			\limits_{\mu=1}^{Q}
			\omega_{\mu}
			\left[
			\left(
			\alpha_1 
			+ 
			\left(
			L_1
			\right)
			_{i+\frac{1}{2},j}^{\mu}
			\beta
			\right)
			{h}
			_{i+\frac{1}{2},j}^{-,\mu}
			+
			\left(
			\alpha_1 
			+ 
			\left(
			L_1
			\right)
			_{i-\frac{1}{2},j}^{\mu}
			\beta
			\right)
			{h}
			_{i-\frac{1}{2},j}^{+,\mu}
			\right]\nonumber
			\\
			&-\frac{\Delta t}{\Delta \xi_2}
			\sum
			\limits_{\mu=1}^{Q}
			\omega_{\mu}
			\left[
			\left(
			\alpha_2 
			+ 
			\left(
			L_2
			\right)
			_{i,j+\frac{1}{2}}^{\mu}
			\beta
			\right)
			{h}
			_{i,j+\frac{1}{2}}^{-,\mu}
			+
			\left(
			\alpha_2
			+ 
			\left(
			L_2
			\right)
			_{i,j-\frac{1}{2}}^{\mu}
			\beta
			\right)
			{h}
			_{i,j-\frac{1}{2}}^{+,\mu}
			\right].\label{eq:I3I4}
		\end{align}
		Plugging $I_1\geqslant0$,  $I_2\geqslant0$, and \eqref{eq:I3I4} into equation \eqref{eq:JHDelta}, one obtains
		\begin{align}\label{eq:h_pp_delta_t}
			\overline{h}_{i,j}^{\Delta t}  - \dfrac{\overline{Jh}_{i,j}}{\overline{J}_{i,j}^{\Delta t}} \geqslant
			&
			-\frac{\Delta t}{\overline{J}_{i,j}^{\Delta t}\Delta \xi_1}
			\sum
			\limits_{\mu=1}^{Q}
			\omega_{\mu}
			\left[
			\left(
			\alpha_1 
			+ 
			\left(
			L_1
			\right)
			_{i+\frac{1}{2},j}^{\mu}
			\beta
			\right)
			{h}
			_{i+\frac{1}{2},j}^{-,\mu}
			+
			\left(
			\alpha_1 
			+ 
			\left(
			L_1
			\right)
			_{i-\frac{1}{2},j}^{\mu}
			\beta
			\right)
			{h}
			_{i-\frac{1}{2},j}^{+,\mu}
			\right]\nonumber
			\\
			&-\frac{\Delta t}{\overline{J}_{i,j}^{\Delta t} \Delta \xi_2}
			\sum
			\limits_{\mu=1}^{Q}
			\omega_{\mu}
			\left[
			\left(
			\alpha_2 
			+ 
			\left(
			L_2
			\right)
			_{i,j+\frac{1}{2}}^{\mu}
			\beta
			\right)
			{h}
			_{i,j+\frac{1}{2}}^{-,\mu}
			+
			\left(
			\alpha_2
			+ 
			\left(
			L_2
			\right)
			_{i,j-\frac{1}{2}}^{\mu}
			\beta
			\right)
			{h}
			_{i,j-\frac{1}{2}}^{+,\mu}
			\right].
		\end{align}
		Furthermore, by using the fact
		\begin{align*}
			\overline{Jh}_{i,j}
			=&\frac{1}{2} \sum
			\limits_{\mu=1}^{Q}\sum_{\nu = 2}^{Q-1} \omega_{\mu} \omega_{\nu}
			{J}_{i_{\mu},j_{\nu}}{h}_{i_{\mu},j_{\nu}} 
			+ \frac{1}{2} \sum_{\mu=2}^{Q-1}\sum_{\nu=1}^{Q} \omega_{\mu} \omega_{\nu}
			{J}_{i_{\mu},j_{\nu}}{h}_{i_{\mu},j_{\nu}} 
			\\&+ \sum
			\limits_{\mu=1}^{Q}
			\frac{\omega_{\mu}\omega_1}{2}
			\left[
			J
			_{i+\frac{1}{2},j}^{-,\mu}
			{h}
			_{i+\frac{1}{2},j}^{-,\mu}
			+
			J
			_{i-\frac{1}{2},j}^{+,\mu}
			{h}
			_{i-\frac{1}{2},j}^{+,\mu}
			+
			J
			_{i,j+\frac{1}{2}}^{-,\mu}
			{h}
			_{i,j+\frac{1}{2}}^{-,\mu}
			+
			J
			_{i,j-\frac{1}{2}}^{+,\mu}
			{h}
			_{i,j-\frac{1}{2}}^{+,\mu}
			\right],
		\end{align*}
		the inequality \eqref{eq:h_pp_delta_t} can be reformulated as
		\begin{small}
			\begin{equation*}
				\begin{aligned}
					\overline{h}_{i,j}^{\Delta t}   \geqslant
					&
					\frac{1}{\overline{J}_{i,j}^{\Delta t}}
					\sum
					\limits_{\mu=1}^{Q}
					\omega_{\mu}
					\left[
					\left(
					\frac{1}{2}
					\omega_1
					{J}
					_{i+\frac{1}{2},j}^{-,\mu}
					-\frac{\Delta t}{\Delta \xi_1}
					\left(
					\alpha_1 
					+ 
					\left(
					L_1
					\right)
					_{i+\frac{1}{2},j}^{\mu}
					\beta
					\right)
					\right)
					{h}
					_{i+\frac{1}{2},j}^{-,\mu}
					+
					\left(
					\frac{1}{2}
					\omega_1
					{J}
					_{i-\frac{1}{2},j}^{+,\mu}
					-\frac{\Delta t}{\Delta \xi_1}
					\left(
					\alpha_1 
					+ 
					\left(
					L_1
					\right)
					_{i-\frac{1}{2},j}^{\mu}
					\beta
					\right)
					\right)
					{h}
					_{i-\frac{1}{2},j}^{+,\mu}
					\right]
					\\
					&+\frac{1}{\overline{J}_{i,j}^{\Delta t}}
					\sum
					\limits_{\mu=1}^{Q}
					\omega_{\mu}
					\left[
					\left(
					\frac{1}{2}
					\omega_1 {J}
					_{i,j+\frac{1}{2}}^{-,\mu}
					-\frac{\Delta t}{\Delta \xi_2}
					\left(
					\alpha_2
					+ 
					\left(
					L_2
					\right)
					_{i,j+\frac{1}{2}}^{\mu}
					\beta
					\right)
					\right)
					{h}
					_{i,j+\frac{1}{2}}^{-,\mu}
					+
					\left(
					\frac{1}{2}
					\omega_1 {J}
					_{i,j-\frac{1}{2}}^{+,\mu}
					-\frac{\Delta t}{\Delta \xi_2}
					\left(
					\alpha_2	
					+ 
					\left(
					L_2
					\right)
					_{i,j-\frac{1}{2}}^{\mu}
					\beta
					\right)
					\right)
					{h}
					_{i,j-\frac{1}{2}}^{+,\mu}
					\right]
					\\ &+\frac{1}{2\overline{J}_{i,j}^{\Delta t}} \sum
					\limits_{\mu=1}^{Q}\sum_{\nu = 2}^{Q-1} \omega_{\mu} \omega_{\nu}
					{J}_{i_{\mu},j_{\nu}}{h}_{i_{\mu},j_{\nu}} 
					+ \frac{1}{2\overline{J}_{i,j}^{\Delta t}} \sum_{\mu=2}^{Q}\sum_{\nu=1}^{Q} \omega_{\mu} \omega_{\nu}
					{J}_{i_{\mu},j_{\nu}}{h}_{i_{\mu},j_{\nu}} ,
				\end{aligned}
			\end{equation*}
		\end{small}
		\hspace{-0.25cm}
		where we have used the facts ${h}
		_{i-\frac{1}{2},j}^{+,\mu} = h_{i_{1},j_{\mu}}$, ${h}
		_{i+\frac{1}{2},j}^{-,\mu} = h_{i_{Q},j_{\mu}}$, ${h}
		_{i,j-\frac{1}{2}}^{+,\mu} = h_{i_{\mu},j_{1}}$, and ${h}
		_{i,j+\frac{1}{2}}^{-,\mu} = h_{i_{\mu},j_{Q}}$.
		Under the condition \eqref{eq:Condition_PP} and the CFL condition \eqref{eq:CFL_condition}, one obtains
		\begin{equation*}
			\overline{h}_{i,j}^{\Delta t}  \geqslant 0,
		\end{equation*}
		which means $\overline{\bm{U}}_{i,j}^{\Delta t} \in \mathcal{G}$.
		The proof is completed.
	\end{proof}
	
	\section{Adaptive moving mesh}\label{Sec:MM}
	This section presents the adaptive moving mesh strategy that determines the mesh velocities involved in the temporal mesh metrics;
	see \cite{DUAN2021109949,Duan2022High,Tang2003Adaptive} for more details. We will take the 2D case as an example.
	Define the following mesh adaptation functional
	\begin{equation*}
		\widetilde{E}(\bm{x})=\frac{1}{2} \sum_{\ell=1}^2 \int_{\Omega_c}\left(\nabla_{\bm{\xi}} x_\ell\right)^{\top} \bm{G}_\ell\left(\nabla_{\bm{\xi}} x_\ell\right) \mathrm{d} \bm{\xi},
	\end{equation*}
	whose Euler-Lagrange equations are  
	\begin{equation}\label{eq:Mesh_Control}
		\nabla_{\bm{\xi}} \cdot\left(\bm{G}_\ell \nabla_{\bm{\xi}} x_\ell\right)=0, ~\bm{\xi} \in \Omega_c, ~\ell=1,2,
	\end{equation}
	where $\nabla_{\bm{\xi}}$ denotes the spatial gradient in the computational domain.
	Then the new redistributed mesh $\bm{x}$ in the physical domain $\Omega_p$ can be obtained by solving \eqref{eq:Mesh_Control} for a given symmetric positive definite matrix $\bm{G}_{\ell}$,
	which controls the mesh concentration.
	The choice of $\bm{G}_{\ell}$ is one of the most important parts of the adaptive moving mesh method,
	and usually depends on the solutions of the underlying governing equations or their derivatives.
	Referring to the variable diffusion method in \cite{Winslow1967Numerical}, the simplest choice is the isotropic one as follows
	\begin{equation*}
		\bm{G}_\ell=\omega I_2, ~\ell=1,2,
	\end{equation*}
	where the positive scalar $\omega$ is called the monitor function. It can be taken as
	\begin{align}\label{eq:monitor}
		\omega=\Big({1+\sum_{k=1}^K \theta_k\Big(
			\dfrac{\abs{\nabla_{\bm{\xi}}\sigma_k}}
			{\max\abs{\nabla_{\bm{\xi}}\sigma_k}}\Big)^2}\Big)^{1/2}.
	\end{align}
	Here $\{\sigma_k\}_{k=1}^K$ are the chosen physical variables, and $\theta_k$ is a positive parameter. 
	The monitor function \eqref{eq:monitor} is approximated by using the second-order central difference for the derivatives in the computational domain, 
	and it is further smoothed by performing $5\sim10$ times via the following low-pass filter
	\begin{equation*}
		\omega_{i+\frac{1}{2},j+\frac{1}{2}} = \sum_{i_1,j_1 = 0,\pm 1} \left(\frac{1}{2}\right)^{|i_1|+|j_1|+2}\omega_{i+\frac{1}{2}+i_1,j+\frac{1}{2}+j_1}.
	\end{equation*}
	Utilizing the second-order central difference scheme to discretize the mesh governing equations \eqref{eq:Mesh_Control} and coupling it with the Jacobi iteration method give 
	\begin{equation*}
		\begin{aligned}
			&\left[\left(\omega_{i+\frac{1}{2},j+\frac{1}{2}}+\omega_{i+\frac{3}{2},j+\frac{1}{2}}\right)\left(\bm{x}_{i+\frac{3}{2},j+\frac{1}{2}}^{[\nu]}-\bm{x}_{i+\frac{1}{2},j+\frac{1}{2}}^{[\nu+1]}\right)-\left(\omega_{i+\frac{1}{2},j+\frac{1}{2}}+\omega_{i-\frac{1}{2},j+\frac{1}{2}}\right)\left(\bm{x}_{i+\frac{1}{2},j+\frac{1}{2}}^{[\nu+1]}-\bm{x}_{i-\frac{1}{2},j+\frac{1}{2}}^{[\nu]}\right)\right]/\Delta \xi_1^{2}\\
			&\left[\left(\omega_{i+\frac{1}{2},j+\frac{1}{2}}+\omega_{i+\frac{1}{2},j+\frac{3}{2}}\right)\left(\bm{x}_{i+\frac{1}{2},j+\frac{3}{2}}^{[\nu]}-\bm{x}_{i+\frac{1}{2},j+\frac{1}{2}}^{[\nu+1]}\right)-\left(\omega_{i+\frac{1}{2},j+\frac{1}{2}}+\omega_{i+\frac{1}{2},j-\frac{1}{2}}\right)\left(\bm{x}_{i+\frac{1}{2},j+\frac{1}{2}}^{[\nu+1]}-\bm{x}_{i+\frac{1}{2},j-\frac{1}{2}}^{[\nu]}\right)\right]/\Delta \xi_2^{2}\\
			&=~0, ~\nu=0,1, \cdots, \mu-1,
		\end{aligned}
	\end{equation*}
	where the initial mesh in the iterations $\bm{x}_{i+\frac{1}{2},j+\frac{1}{2}}^{[0]}:=\bm{x}_{i+\frac{1}{2},j+\frac{1}{2}}^n
	$ is chosen as the mesh at $t^n$, 
	and the monitor function $\omega$ is computed based on the solution at $t^n$.
	Once the grid $\left\{\bm{x}_{i+\frac{1}{2},j+\frac{1}{2}}^{[\mu]}\right\}$ is obtained, the final adaptive mesh is given by
	\begin{equation*}
		\bm{x}_{i+\frac{1}{2},j+\frac{1}{2}}^{n+1}:=\bm{x}_{i+\frac{1}{2},j+\frac{1}{2}}^n+\Delta_\tau\left(\delta_\tau \bm{ x}\right)_{i+\frac{1}{2},j+\frac{1}{2}}^n,~\left(\delta_\tau \bm{x}\right)_{i+\frac{1}{2},j+\frac{1}{2}}^n:=\bm{x}_{i+\frac{1}{2},j+\frac{1}{2}}^{[\mu]}- \bm{x}_{i+\frac{1}{2},j+\frac{1}{2}}^n
	\end{equation*}
	with the parameter $\Delta_{\tau}$ to control the mesh movement
	\begin{equation*}
		\Delta_\tau \leqslant \begin{cases}-\frac{1}{2\left(\delta_\tau x_1\right)_{i+\frac{1}{2},j+\frac{1}{2}}}\left[\left(x_1\right)_{i+\frac{1}{2},j+\frac{1}{2}}^n-\left(x_1\right)_{i-\frac{1}{2},j+\frac{1}{2}}^n\right], & \left(\delta_\tau x_1\right)_{i+\frac{1}{2},j+\frac{1}{2}}<0, \\ +\frac{1}{2\left(\delta_\tau x_1\right)_{i+\frac{1}{2},j+\frac{1}{2}}}\left[\left(x_1\right)_{i+\frac{3}{2},j+\frac{1}{2}}^n-\left(x_1\right)_{i+\frac{1}{2},j+\frac{1}{2}}^n\right], & \left(\delta_\tau x_1\right)_{i+\frac{1}{2},j+\frac{1}{2}}>0 ,
			\\
			-\frac{1}{2\left(\delta_\tau x_2\right)_{i+\frac{1}{2},j+\frac{1}{2}}}\left[\left(x_2\right)_{i+\frac{1}{2},j+\frac{1}{2}}^n-\left(x_2\right)_{i+\frac{1}{2},j-\frac{1}{2}}^n\right], & \left(\delta_\tau x_2\right)_{i+\frac{1}{2},j+\frac{1}{2}}<0, 
			\\ 
			+\frac{1}{2\left(\delta_\tau x_2\right)_{i+\frac{1}{2},j+\frac{1}{2}}}\left[\left(x_2\right)_{i+\frac{1}{2},j+\frac{3}{2}}^n-\left(x_2\right)_{i+\frac{1}{2},j+\frac{1}{2}}^n\right], & \left(\delta_\tau x_2\right)_{i+\frac{1}{2},j+\frac{1}{2}}>0 .
		\end{cases}
	\end{equation*}
	Finally, the mesh velocities for the next time step are defined as 
	\begin{equation*}
		\dot{\bm{x}}_{i+\frac{1}{2},j+\frac{1}{2}}^n:=\Delta_\tau\left(\delta_\tau \bm{x}\right)_{i+\frac{1}{2},j+\frac{1}{2}}^n / \Delta t^n,
	\end{equation*}
	where $\Delta t^{n}$ is the time stepsize. One can get the values $\dot{\bx}_{i_{\mu},j_{\nu}}$ by using the similar dimension-by-dimension approach illustrated in Subsection \ref{Sec:Mesh_coeff}, so we omit it here. 
	
	To get the fully-discrete schemes, we use the explicit SSP RK3 time discretization 
	\begin{equation*}
		\begin{aligned}
			&\left(J\bm{{U}}\right)^{*} = \ \left(J\bm{{U}}\right)^{n}+\Delta t^n \bm{{L}}\left(\bU^{n}, \bx^n\right),~
			J^{*} = J^{n}+\Delta t^n {{L}}\left(
			\bx^n\right),\\
			&\bx^{*} = \bx^{n}+\Delta t^n \dot{\bx}^n,
			\\
			&\left(J\bm{{U}}\right)^{**} = \ \frac{3}{4}\left(J\bm{{U}}\right)^{n}+\frac{1}{4}\left(\left(J\bm{{U}}\right)^{*}+\Delta t^n \bm{{L}}\left(\bU^{*},
			{\bx}^*\right)\right),~
			J^{**} = \frac{3}{4}J^{n}+\frac{1}{4}\left(J^{*}+\Delta t^n {{L}}\left(
			{\bx}^*\right)\right),\\
			&\bx^{**} = \frac{3}{4}\bx^{n}+\frac{1}{4}\left(\bx^{*} + \Delta t^n\dot{\bx}^n\right),
			\\
			&\left(J\bm{{U}}\right)^{n+1} = \ \frac{1}{3}\left(J\bm{{U}}\right)^{n}+\frac{2}{3}\left(\left(J\bm{{U}}\right)^{**}+\Delta t^n \bm{{L}}\left(\bU^{**},
			{\bx}^{**}\right)\right),~
			J^{n+1} = \frac{1}{3}J^{n}+\frac{2}{3}\left(J^{**}+\Delta t^n {{L}}\left(
			{\bx}^{**}\right)\right),\\
			&\bx^{n+1} = \frac{1}{3}\bx^{n}+\frac{2}{3}\left(\bx^{**} + \Delta t^n{\bx}^n\right),
		\end{aligned}
	\end{equation*}
	where $\dot{\bx}^n$ is the mesh velocity determined by using the adaptive moving mesh strategy in
	Section \ref{Sec:MM}, 
	$\bm{{L}}$ is the right-hand side of \eqref{eq:1D_dis_U} or \eqref{eq:2D_dis_U}, 
	and ${{L}}$ is the right-hand side of the semi-discrete VCL \eqref{eq:1D_dis_J} or \eqref{eq:2D_dis_J}.

	\section{Numerical tests}\label{sec:Numerical_Test}
	
	This section conducts several numerical experiments to 
	validate the high-order accuracy, high resolution, and WB and PP properties of the proposed moving mesh scheme. 
	For the adaptive mesh redistribution,
	the total number of iterations for solving the adaptive mesh equation 
	(see \cite{Duan2022High,Li2022High})
	is set as $10$, 
	and the monitor function $\omega$ will be given in each example.
	The fifth-order 
	WENO-Z method \cite{Borges2008An} is employed for spatial reconstruction.  
	The time stepsize $\Delta t^{n}$ is chosen according to Theorems \ref{Pro:1D_SWE_PP_J}, \ref{Pro:1D_SWE_PP}, \ref{Pro:SWE_PP_J} and \ref{Pro:SWE_PP}, except for 
	the accuracy tests in Examples \ref{ex:1DSmooth} and \ref{ex:2D_Smooth_PP}, where it is taken as $ \min \{ \Delta t^n,  (\Delta \xi)^{5/3} \}$ and $ \min \{ \Delta t^n,  (\min_{\ell}\{ \Delta \xi_\ell \})^{5/3} \}$ , respectively, to align the temporal and spatial convergence rates. Our structure-preserving schemes on fixed uniform and moving meshes are denoted as ``\texttt{UM-SP}" and ``\texttt{MM-SP}", respectively.  To address the effect of round-off error for very shallow water depths, the velocities are computed using the following formula proposed in \cite{Alexander2007Asecond}: 
	\begin{equation*}
		(v_\ell) = \begin{cases}
			(hv_{\ell})/(h), & \text{if}\quad h>10^{-4},
			\\
			\dfrac{\sqrt{2}h(hv_\ell)}{\sqrt{h^4+\max\{h^4,\mathcal{T}\}}}, &\text{otherwise}, 
		\end{cases}
	\end{equation*}
	where $\mathcal{T} = (\Delta \xi)^{4}$ in 1D cases and $T = (\Delta \xi_1 \Delta \xi_2)^{2}$ in 2D cases, respectively.

	\subsection{1D tests}
	\begin{example}[Accuracy test with manufactured solution]\label{ex:1DSmooth}\rm
		This example \cite{Zhang2023High} is used to validate the convergence rate of the proposed moving mesh scheme by solving the following equations with an additional source term:
		\begin{equation*}
			\pd{\bU}{t} + \pd{\bF}{x} = -gh\pd{\bm{B}}{x} + \bm{S},
		\end{equation*}
		where the gravitational acceleration $g=1$. 
		The physical domain is set as $[0,2]$ with periodic boundary conditions. The exact solution is given by
		\begin{equation*}
			h(x,t) = 4+\cos(\pi x)\cos(\pi t),~v_1(x,t) = \frac{\sin(\pi x)\sin(\pi t)}{h},~b(x) = 1.5+\sin(\pi x),
		\end{equation*}
		and the additional source term is defined as
		\begin{align*}
			\bm{S}(x,t) = (0, ~&4\pi \cos(\pi x) + \pi\cos(\pi t)\cos^2(\pi x) - 3\pi\cos(\pi t)\sin(\pi x) - \pi\cos^2(\pi t)\cos(\pi x)\sin(\pi x) \\&+ (\pi\cos(\pi t)\sin^2(\pi t) \sin^3(\pi x))/(\cos(\pi t)\cos(\pi x) + 4)^2 \\&+ (2\pi\cos(\pi x)\sin^2(\pi t)\sin(\pi x))/(\cos(\pi t)\cos(\pi x) + 4)
			,~0)^{\top}.
		\end{align*}
		The numerical simulations are carried out up to $t = 0.1$. The monitor function is chosen as
		\begin{equation*}
			\omega=\left(1+\theta\left(\frac{ |\nabla_{{\xi}} (h+b)|}{\max |\nabla_{{\xi}} (h+b)|}\right)^2\right)^{1/2}
		\end{equation*}
		with $\theta = 3$.
	\end{example}
	Figure \ref{fig:1D_Smooth_Accuracy} presents the errors and the convergence behavior in the water depth $h$ computed by the \texttt{MM-SP} scheme at $t = 0.1$. The results clearly verify the designed fifth-order accuracy of the proposed scheme. 
	\begin{figure}[hbt!]
		\centering
		\begin{subfigure}[b]{0.5\textwidth}
			\centering
			\includegraphics[width=1.0\linewidth]{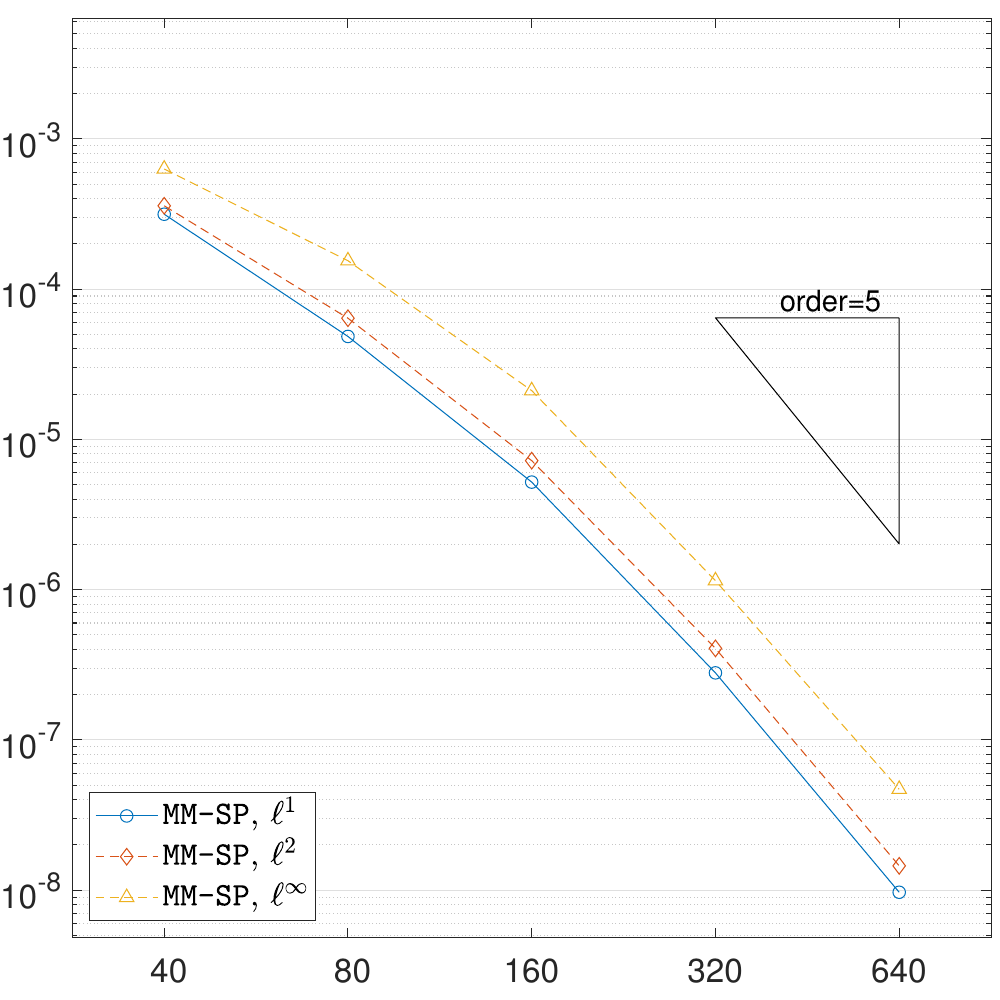}
		\end{subfigure}
		\caption{Example \ref{ex:1DSmooth}. The $\ell^{1}$, $\ell^{2}$, and $\ell^{\infty}$ errors and convergence rate in $h$ obtained by the \texttt{MM-SP} scheme at $t=0.1$.}\label{fig:1D_Smooth_Accuracy}
	\end{figure}

	\begin{example}[1D WB test]\label{ex:1D_WB_Test}\rm
		To demonstrate the WB property of the proposed scheme for the 1D SWEs, two different bottom topographies are considered within the physical domain $[0, 10]$ with outflow boundary conditions. The first topography is a smooth Gaussian profile:  
		\begin{equation}\label{eq:b_Smooth}
			b(x) = 5e^{-\frac{2}{5}(x-5)^2},
		\end{equation}
		and the second is a discontinuous square step:
		\begin{equation}\label{eq:b_dis}
			b(x)= \begin{cases}4,~&\text{if}\quad x \in[4,8], \\ 0,~&\text{otherwise}.\end{cases}
		\end{equation}
		The initial condition is given by
		$h = 10 - b$ and zero velocity. The gravitational acceleration constant is taken as $g = 1$.
		The output time is $t = 1$. 
		The monitor function is
		\begin{equation*}
			\omega=\left({1+100\left(\frac{|\nabla_{{\xi}} h|}{\max |\nabla_{{\xi}} h|}\right)^2}\right)^{1/2}.
		\end{equation*}
	\end{example}
	Table \ref{tab:1D Well_Balance} displays the $\ell^{1}$ and $\ell^{\infty}$ errors in water surface level $h+b$ and velocity $v_1$ computed with 200 mesh points. The errors are at the level of rounding error in double precision, demonstrating that both the \texttt{UM-SP} and \texttt{MM-SP} schemes are indeed WB. 
	Figure \ref{fig:1D_Well_balance} illustrates the water surface level $h+b$, bottom topography $b$, and the adaptive mesh for the \texttt{MM-SP} scheme. These visualizations show that our moving mesh scheme effectively preserves the lake-at-rest condition and that mesh points are densely concentrated around the sharp transitions in the bottom topography.

	\begin{table}[!htb]
		\centering
		\caption{Example \ref{ex:1D_WB_Test}.
			Errors in $h+b$ and $v_1$ obtained by using our schemes with $200$ mesh points at $t=1$, for two different bottom topographies \eqref{eq:b_Smooth} and \eqref{eq:b_dis}, respectively.}\label{tab:1D Well_Balance}
		\begin{tabular}{cccccc}
			\hline\hline
			\multicolumn{2}{c}{\multirow{2}{*}{}} & \multicolumn{2}{c}{\texttt{UM-SP}} & \multicolumn{2}{c}{\texttt{MM-SP}} \\ 
			\cmidrule(lr){3-4}\cmidrule(lr){5-6}
			\multicolumn{2}{c}{} & \multicolumn{1}{c}{$\ell^{1}$~error}  & \multicolumn{1}{c}{$\ell^{\infty}$~error} &  \multicolumn{1}{c}{$\ell^{1}$~error}  & \multicolumn{1}{c}{$\ell^{\infty}$~error}  \\
			\hline
			\multirow{2}{*}{$b$ in \eqref{eq:b_Smooth}} &
			$h+b$ & 1.24e-14 	 & 4.27e-14 	 & 2.13e-14 	 & 4.92e-14  \\ 
			& $v_1$ & 3.32e-15 	 & 1.21e-14 	 & 7.04e-15 	 & 1.63e-14  \\ 
			\hline
			\multirow{2}{*}{$b$ in \eqref{eq:b_dis}} &
			$h+b$ & 1.24e-14 	 & 2.79e-14 	 & 1.24e-14 	 & 3.76e-14  \\ 
			&	$v_1$ & 3.15e-15 	 & 1.23e-14 	 & 4.39e-15 	 & 1.46e-14  \\ 
			\hline\hline
		\end{tabular}
	\end{table}
	
	\begin{figure}[!htb]
		\centering
		\begin{subfigure}[b]{0.35\textwidth}
			\centering
			\includegraphics[width=1.0\textwidth]{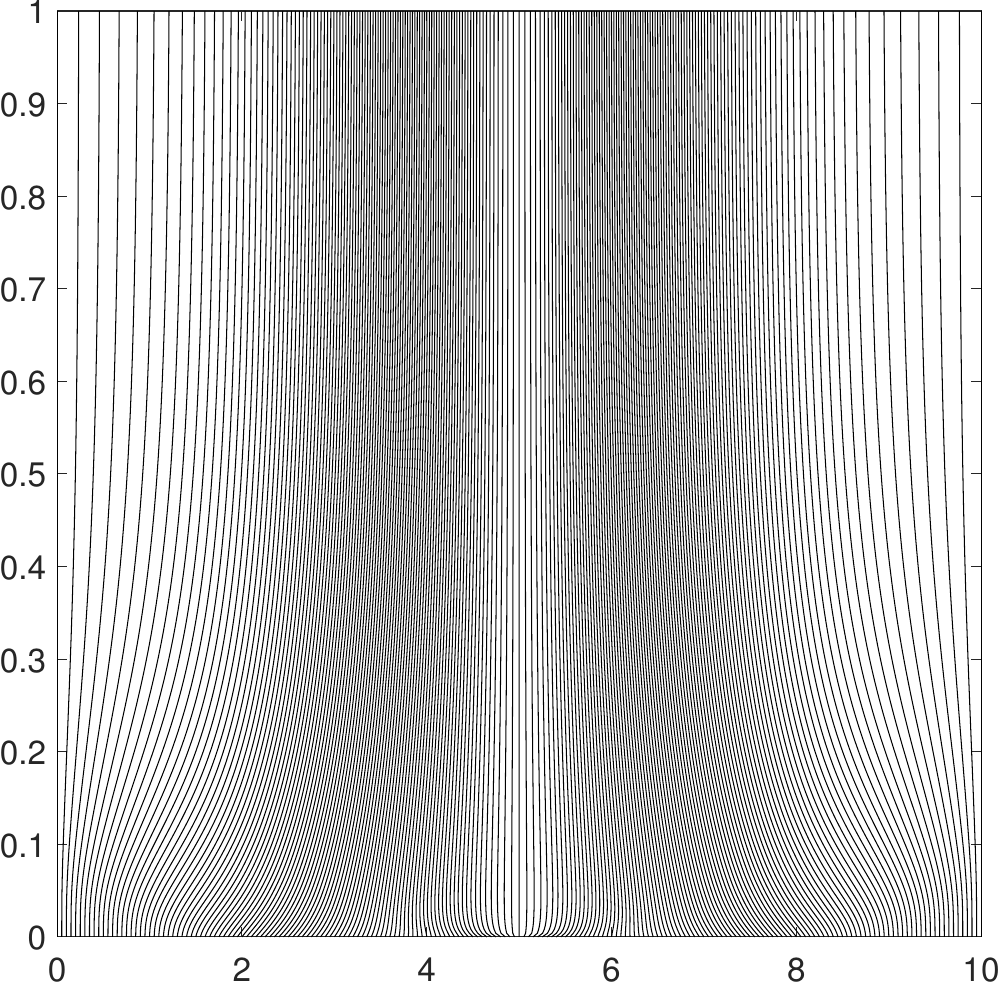}
		\end{subfigure}
		\begin{subfigure}[b]{0.35\textwidth}
			\centering
			\includegraphics[width=1.0\textwidth,  clip]{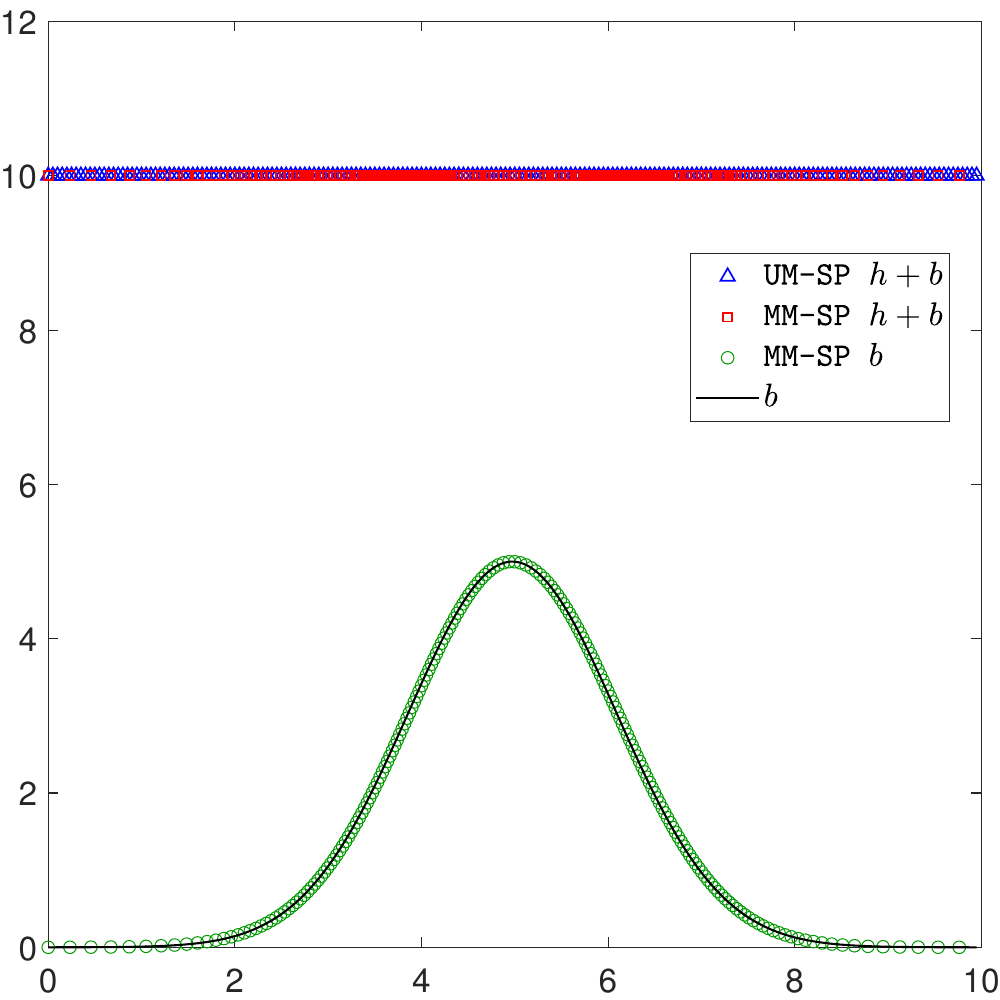}
		\end{subfigure}
		\\
		\begin{subfigure}[b]{0.35\textwidth}
			\centering
			\includegraphics[width=1.0\textwidth]{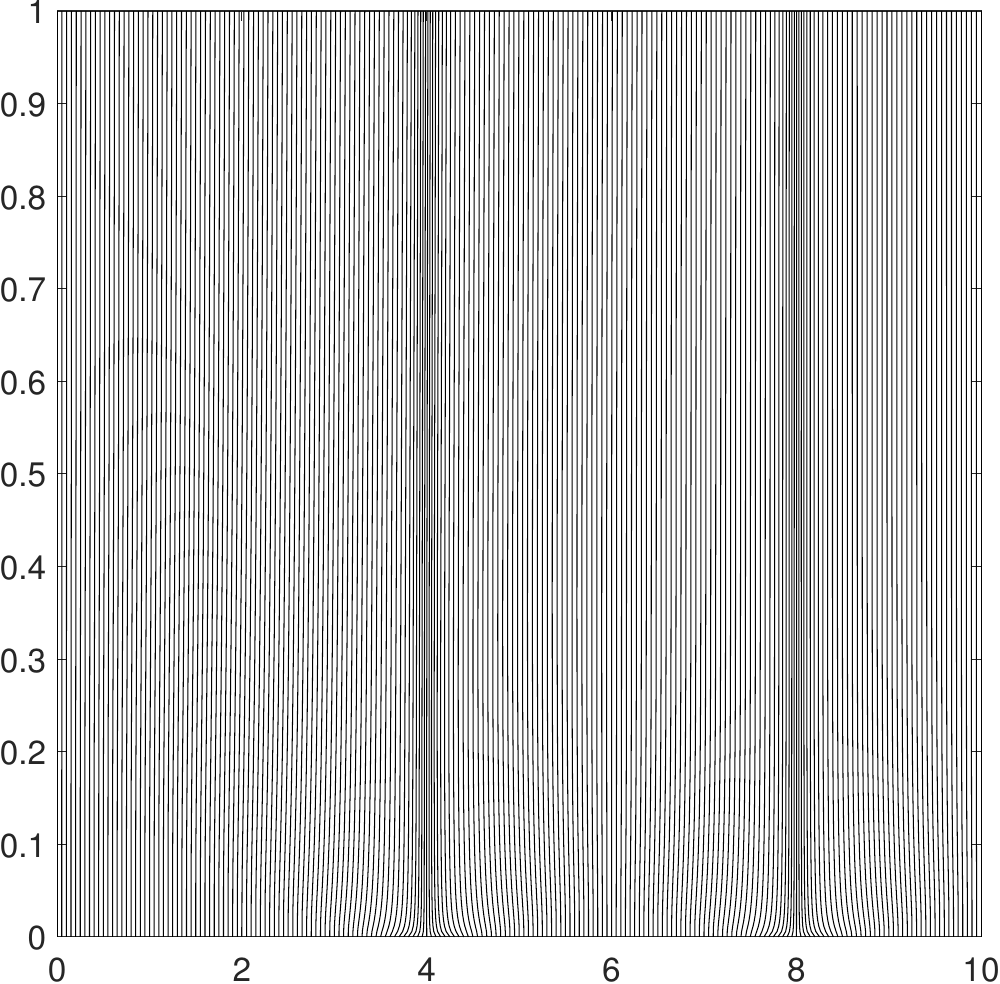}
		\end{subfigure}
		\begin{subfigure}[b]{0.35\textwidth}
			\centering
			\includegraphics[width=1.0\textwidth,  clip]{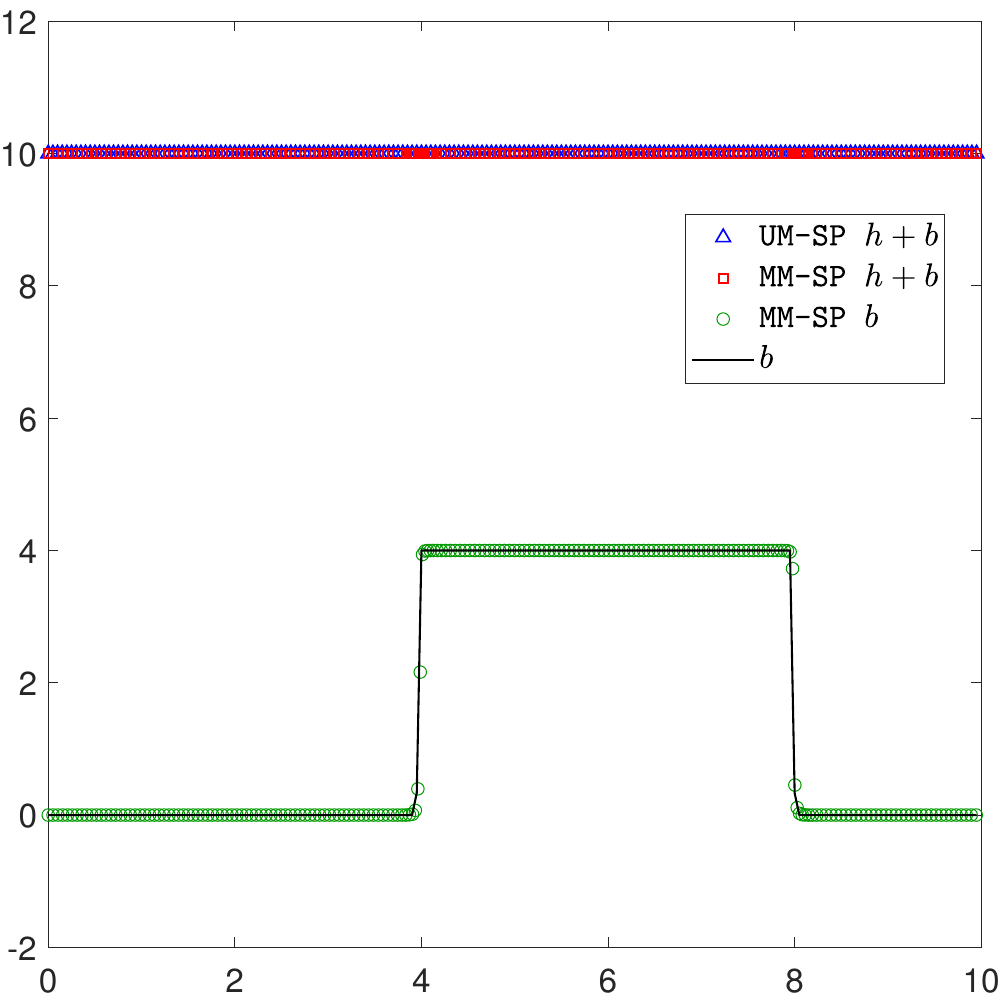}
		\end{subfigure}
		\caption{Example \ref{ex:1D_WB_Test}.
			Left: the mesh trajectory obtained by the \texttt{MM-SP} scheme,
			right: the bottom topography $b$ and water surface level $h+b$.
			Top: with the bottom topography \eqref{eq:b_Smooth},
			bottom: with the bottom topography \eqref{eq:b_dis}.
			The results are obtained with $200$ mesh points at $t=1$.
		}
		\label{fig:1D_Well_balance}
	\end{figure}

	\begin{example}[1D PP test]\rm\label{ex:1D_PP_Mose}
		This example \cite{holden2015front} is used to test the PP property of our schemes for the 1D SWEs with a flat bottom. The physical domain is $[-300,300]$, and the initial conditions are given by 
		\begin{equation*}
			h(x,0) = \begin{cases}
				10, &\text{if}\quad  x<  -70,\\
				0, &\text{if}\quad  -70 \leqslant x\leqslant 70,\\
				10, &\text{if}\quad x>70,
			\end{cases}
		\end{equation*}
		with zero velocity. The monitor function is the same as that in Example \ref{ex:1D_WB_Test}. 
	\end{example}
	
	Figure \ref{fig:1D_PP_Mose} presents the numerical solution at $t=1$, showing the
	comparison between the \texttt{UM-SP} scheme with $200$ and $600$ cells and the \texttt{MM-SP} scheme with $200$
	cells. 
	The mesh points for the \texttt{MM-SP} scheme are observed to concentrate around the regions where the water depth $h$ exhibits a large gradient. The numerical experiments confirm that both the \texttt{UM-SP} and \texttt{MM-SP} schemes effectively preserve the non-negativity of $h$. Notably, without the PP limiter, the code fails at the first time step.

	\begin{figure}[hbt!]
		\centering
		\begin{subfigure}[b]{0.3\textwidth}
			\centering
			\includegraphics[width=1.0\linewidth]{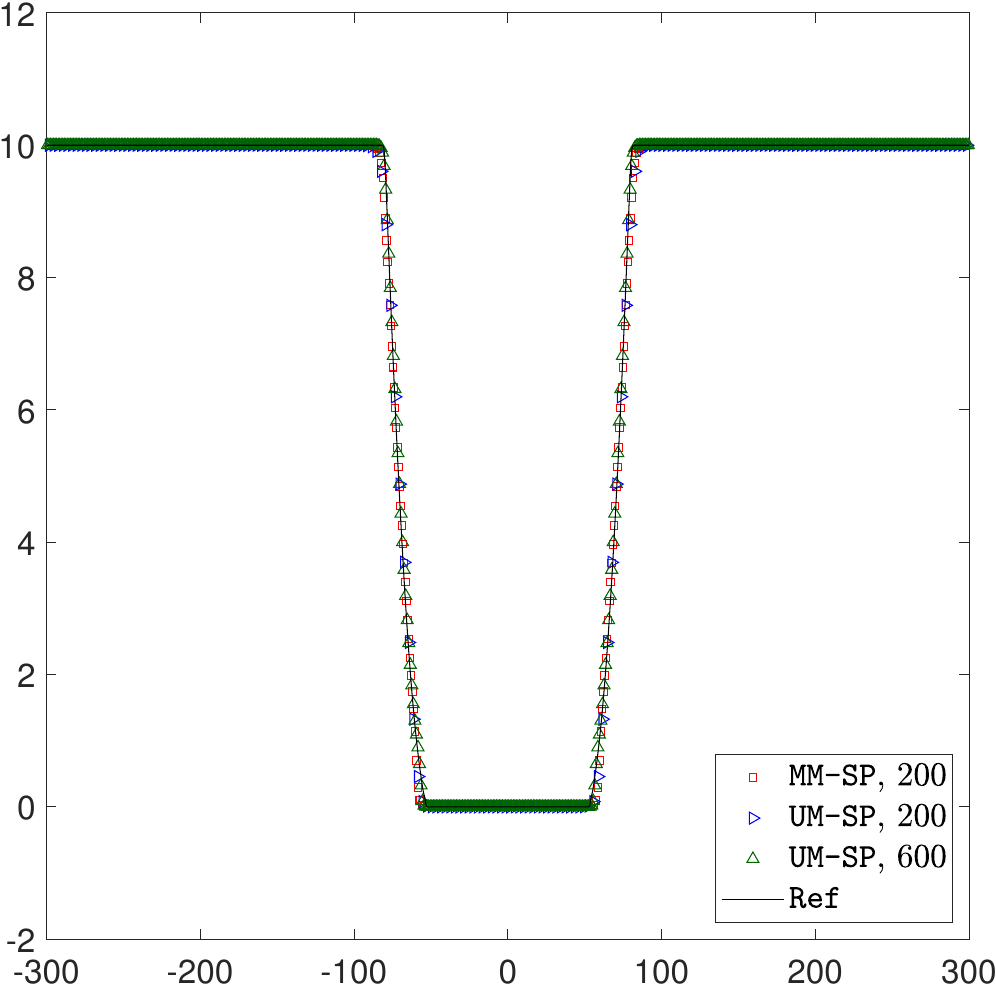}
			\caption{$h$}
		\end{subfigure}
		\begin{subfigure}[b]{0.3\textwidth}
			\centering
			\includegraphics[width=1.0\linewidth]{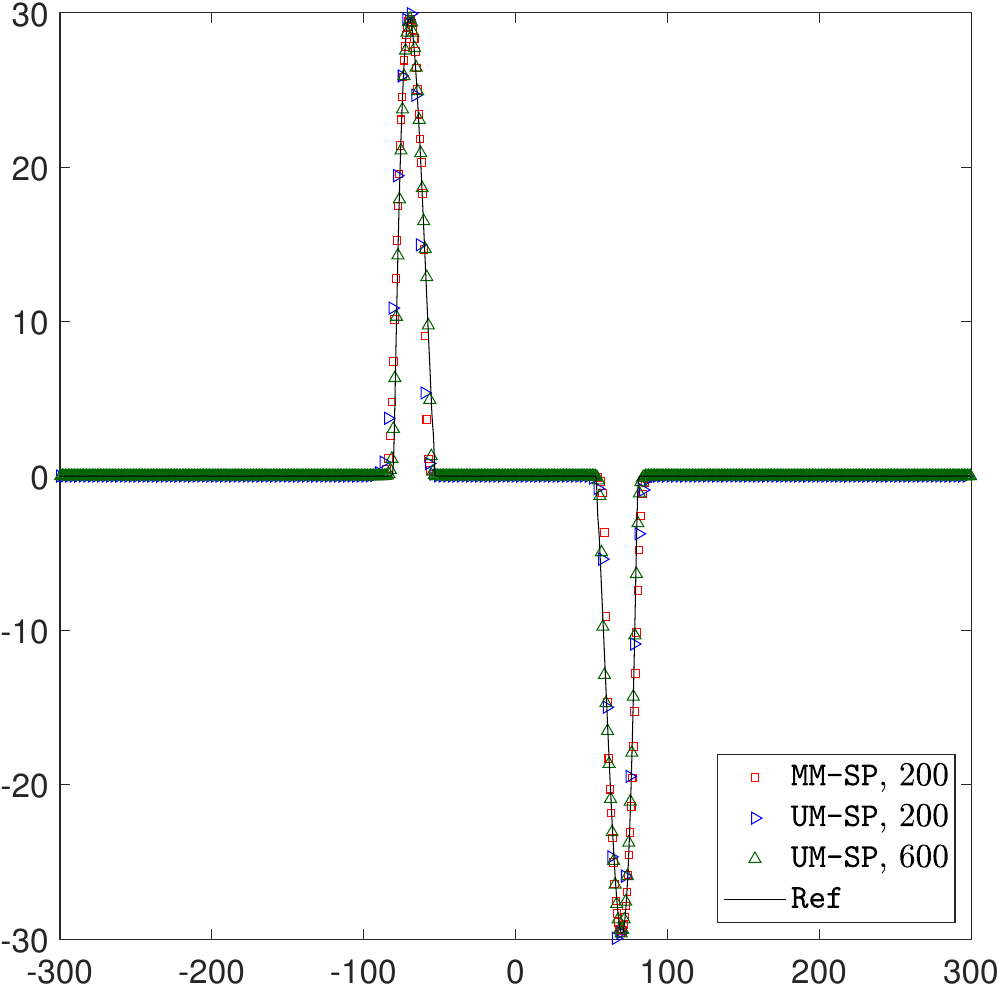}
			\caption{$hv_1$}
		\end{subfigure}
		\quad
		\begin{subfigure}[b]{0.3\textwidth}
			\centering
			\includegraphics[width=1.0\linewidth]{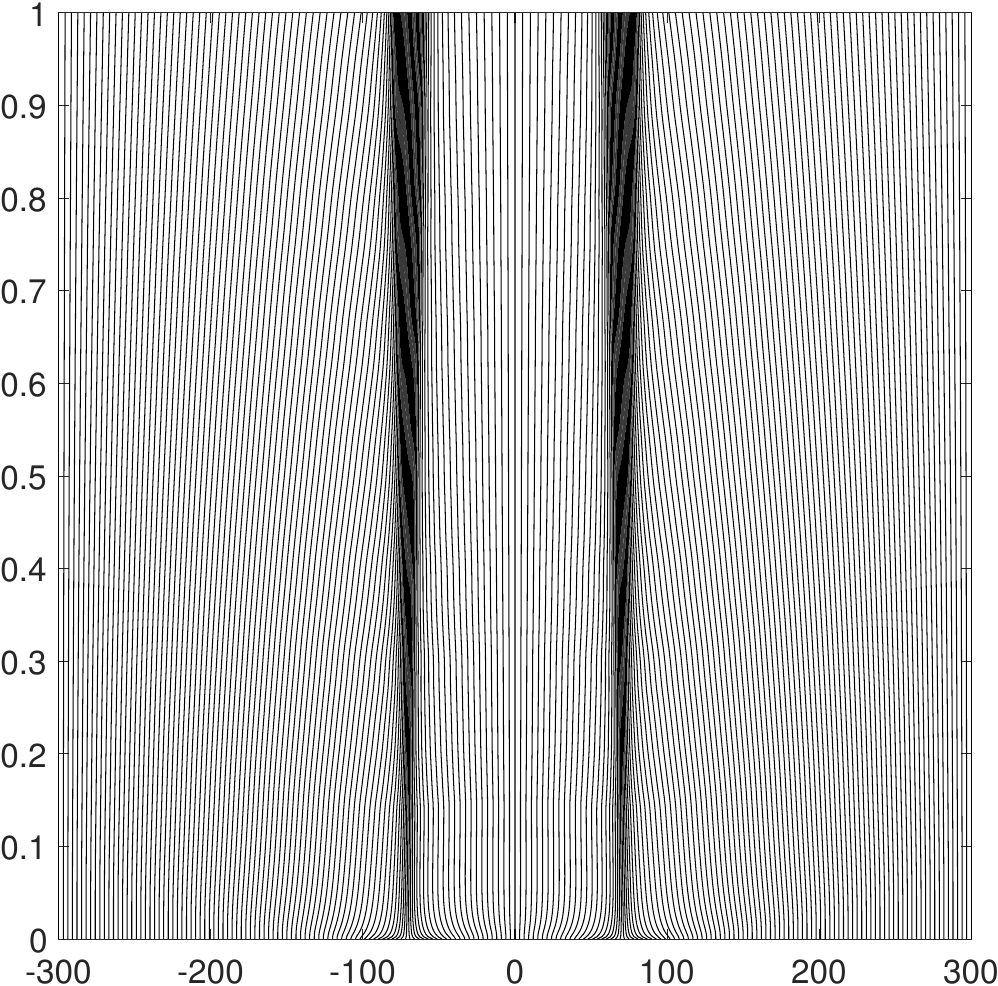}
			\caption{mesh trajectory}
		\end{subfigure}
		\caption{Example \ref{ex:1D_PP_Mose}. The results obtained by using the \texttt{UM-SP} and \texttt{MM-SP} schemes.}\label{fig:1D_PP_Mose}
	\end{figure}

	\begin{example}[Small perturbation test]\label{ex:1D_Pertubation_Test}\rm
		To demonstrate the capability of the \texttt{MM-SP} scheme in capturing small perturbations of steady-state flow, we consider a bottom topography modeled as a ``hump" \cite{Leveque1998Balancing}:
		\begin{equation*}
			b(x) = \begin{cases}
				0.25(\cos(10\pi(x-1.5))+1), &\text{if}\quad 1.4 \leqslant x \leqslant 1.6,\\
				0,   &\text{otherwise}. 
			\end{cases}
		\end{equation*}
		The initial conditions for water depth and velocities are defined as
		\begin{equation*}
			h = \begin{cases}
				1-b+\epsilon, &\text{if}\quad 1.1\leqslant x\leqslant 1.2,\\
				1-b, &\text{otherwise},
			\end{cases},~v_1 = v_2 = 0.0,
		\end{equation*}
		in the physical domain $[0, 2]$ with outflow boundary conditions. The gravitational acceleration constant is set to $g = 9.812$. We test two small perturbations, $\epsilon = 0.2$ and $0.001$, respectively. The monitor function is the same as in Example \ref{ex:1DSmooth}, with $\theta = 100$.  
		Figure \ref{fig:1D_Pertubation_WB} displays the numerical results at $t=0.2$ from both the \texttt{UM-SP} and \texttt{MM-SP} schemes using 200 mesh points, along with a reference solution obtained using the \texttt{UM-SP} scheme with 1000 mesh points. The results illustrate that the wave structures are well captured without notable spurious oscillations, and the \texttt{MM-SP} scheme outperforms the \texttt{UM-SP} scheme with an equivalent number of mesh points. Figure \ref{fig:1D_Perturbation_Compare} shows a comparison of using  different reconstructions. It is observed that merely using Step 1 in Section 3.2 (denoted as \texttt{REC-1}) leads to overshoot and undershoot around the sharp wave structures. This observation has motivated our modifications to perform the WENO reconstruction within the computational domain.

		To further demonstrate the WB and PP properties of the proposed schemes, we consider the following bottom topography 
		\begin{equation}\label{b:1D_WB_PP}
			b(x) = \begin{cases}
				0.5(\cos(10\pi(x-1.5))+1), &\text{if}\quad 1.4 \leqslant x \leqslant 1.6,\\
				0,   &\text{otherwise},
			\end{cases}
		\end{equation}
		which includes a dry region where the water depth $h = 0$ at $x = 1.5$. We introduce a small perturbation with $\epsilon = 0.001$ to evaluate the capability of our schemes in capturing such subtle variations.
		The monitor function is the same as in Example \ref{ex:1DSmooth}, with $\theta = 100$.  
		The numerical results are shown 
		Figure \ref{fig:1D_Pertubation_WB_PP}, validating that both the \texttt{UM-SP} and \texttt{MM-SP} schemes are able to handle the small perturbation cases within a dry region.
		The wave structures captured by the \texttt{MM-SP} scheme with $200$ cells closely align with the reference solutions and surpass the performance of the \texttt{UM-SP} scheme with the same cell count. Notably, the numerical simulation fails at $t \approx 0.1$ if the PP limiter is turned off. 
	\end{example}

	\begin{figure}[!htb]
		\centering
		\begin{subfigure}[b]{0.35\textwidth}
			\centering
			\includegraphics[width=1.0\textwidth]{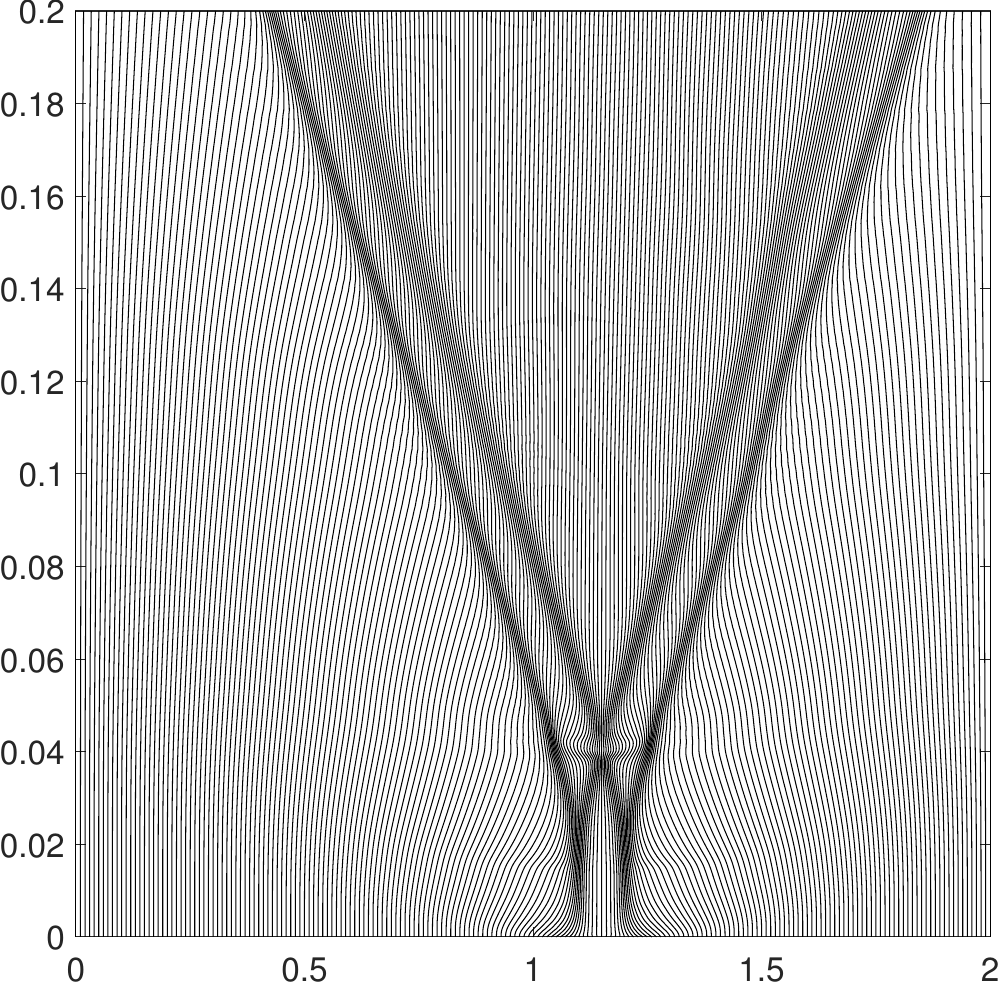}
		\end{subfigure}
		\begin{subfigure}[b]{0.35\textwidth}
			\centering
			\includegraphics[width=1.0\textwidth,  clip]{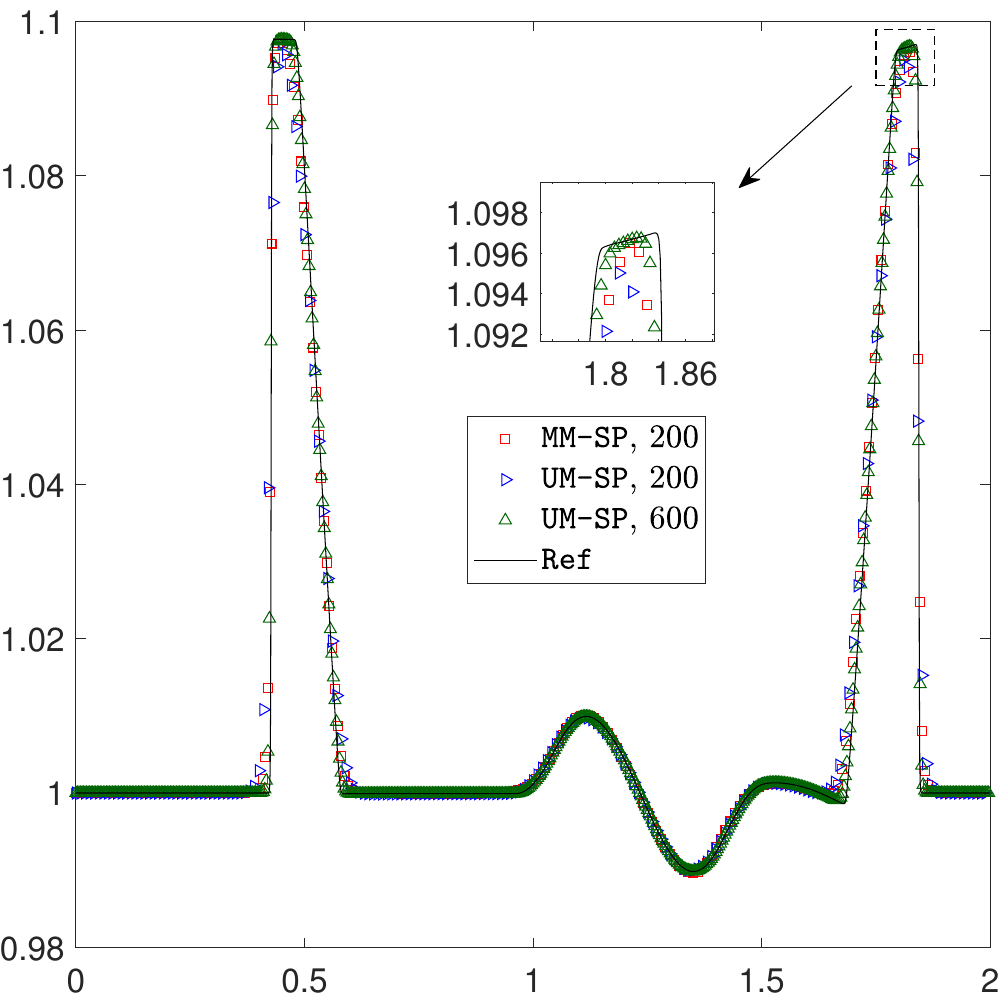}
		\end{subfigure}
		\\
		\begin{subfigure}[b]{0.35\textwidth}
			\centering
			\includegraphics[width=1.0\textwidth]{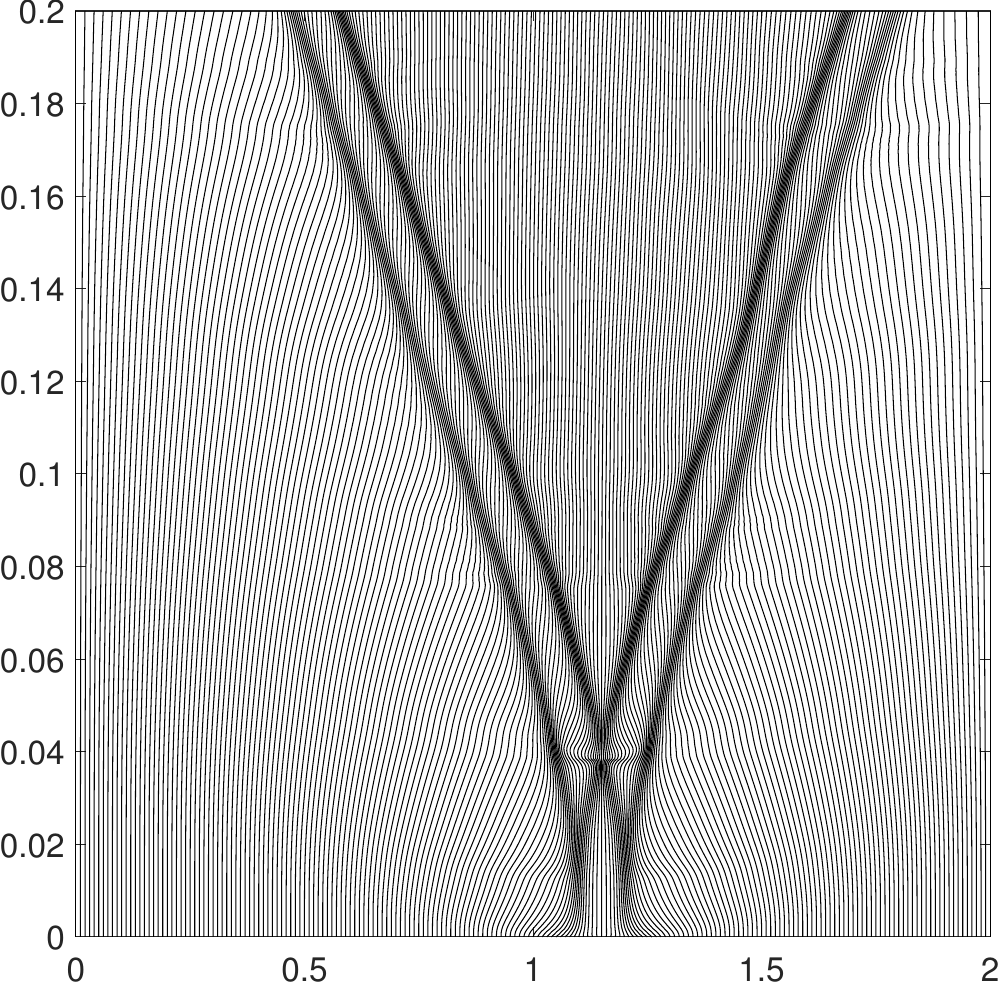}
		\end{subfigure}
		\begin{subfigure}[b]{0.35\textwidth}
			\centering
			\includegraphics[width=1.0\textwidth,  clip]{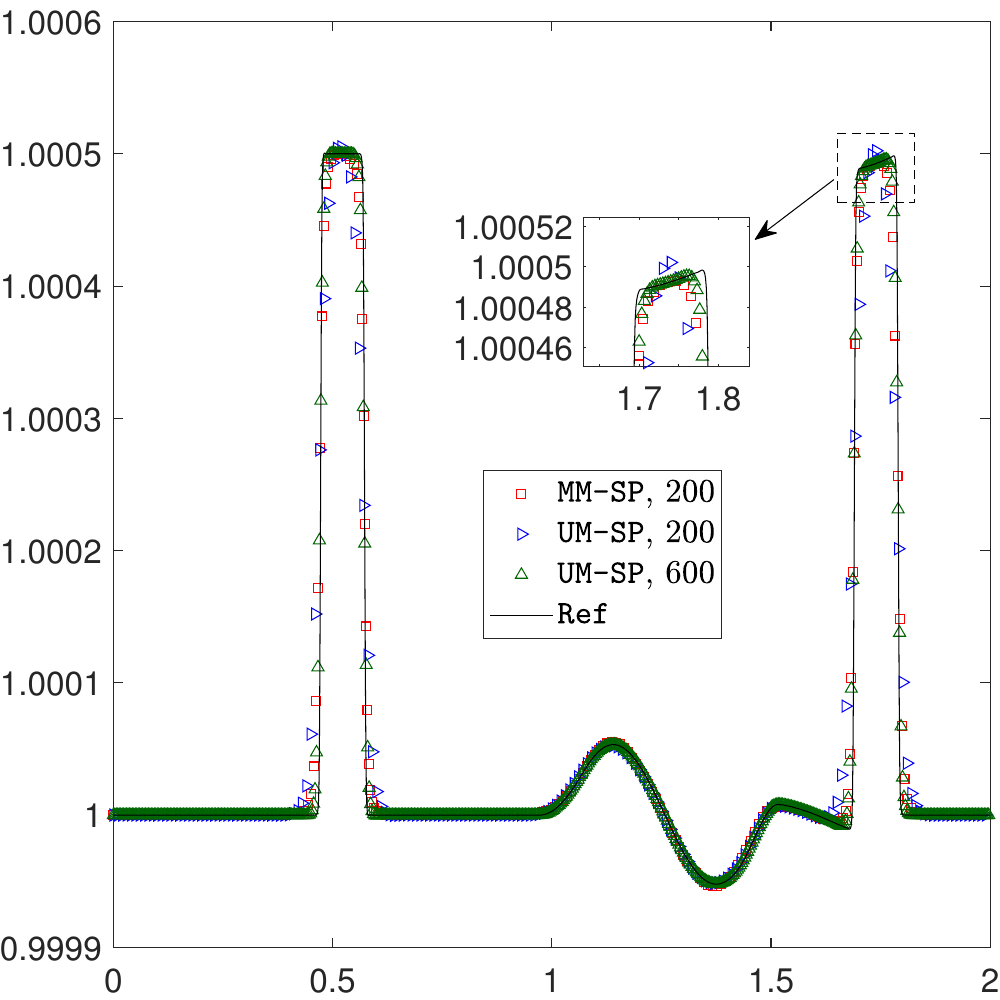}
		\end{subfigure}
		\caption{Example \ref{ex:1D_Pertubation_Test}.
			The numerical solutions obtained by using the \texttt{UM-SP} and \texttt{MM-SP} schemes at $t=0.2$.
			The reference solution is obtained by using the \texttt{UM-SP} scheme with $1000$ mesh points. Left: the adaptive mesh obtained by the \texttt{MM-SP} scheme,
			right: the water surface level $h+b$.
			Top: the case of $\epsilon = 0.2$,
			bottom: the case of $\epsilon = 0.001$.
		}
		\label{fig:1D_Pertubation_WB}
	\end{figure}

	\begin{figure}[hbt!]
		\centering
		\begin{subfigure}[b]{0.3\textwidth}
			\centering
			\includegraphics[width=1.0\linewidth]{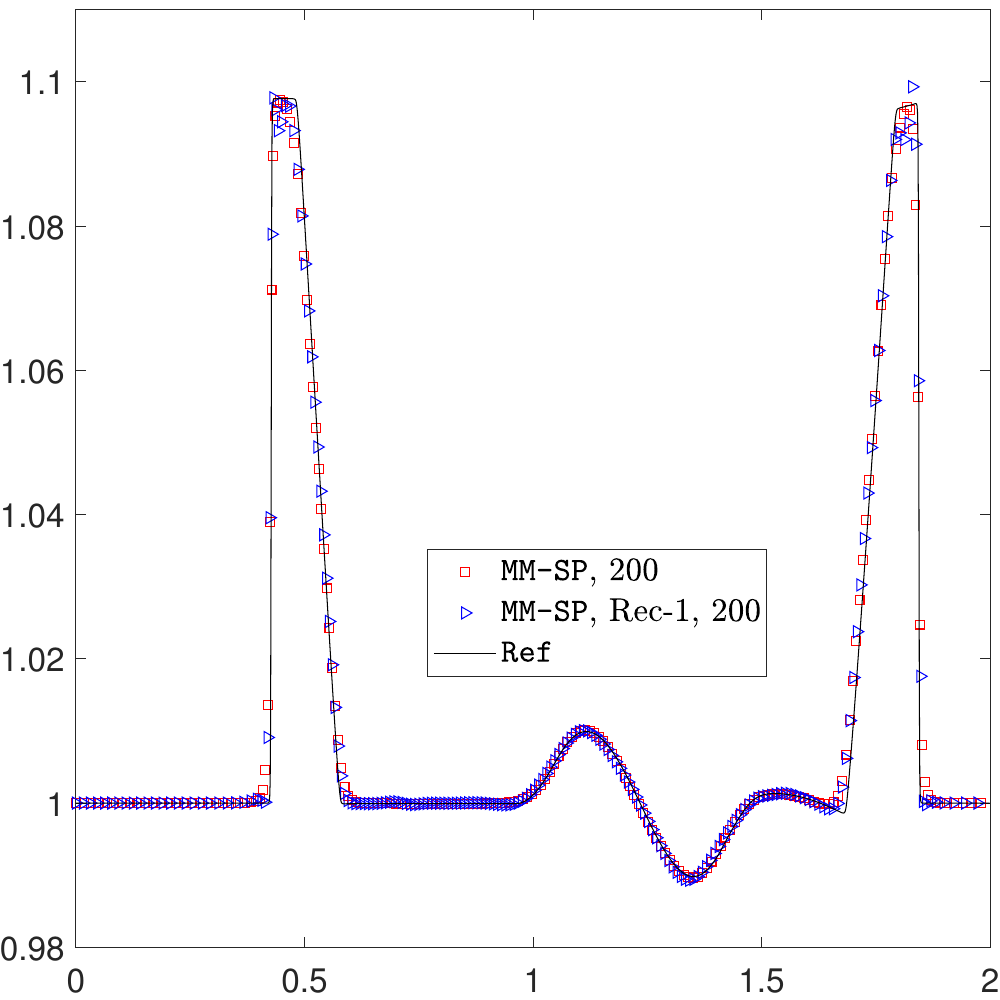}
		\end{subfigure}
		\begin{subfigure}[b]{0.15\textwidth}
			\centering
			\includegraphics[height=2.0\linewidth]{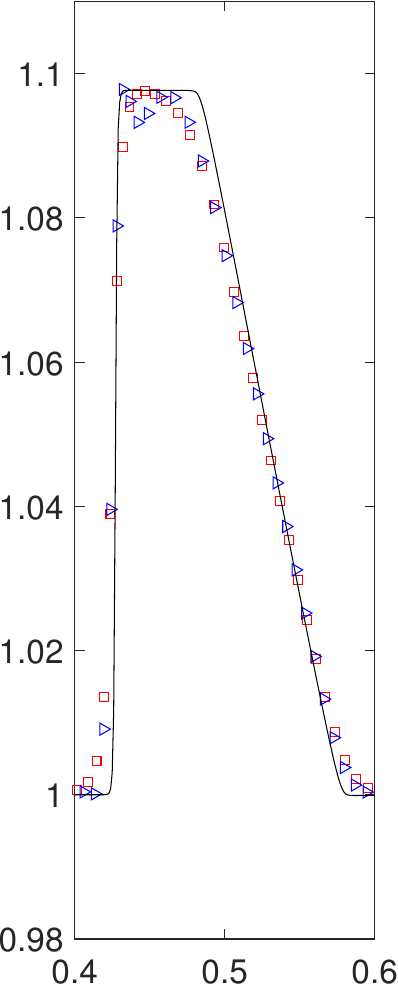}
		\end{subfigure}
		\quad
		\begin{subfigure}[b]{0.15\textwidth}
			\centering
			\includegraphics[height=2.0\linewidth]{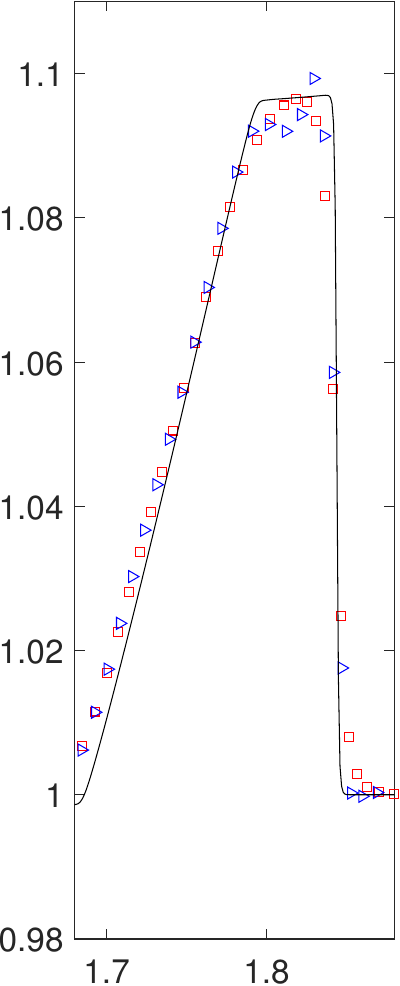}
		\end{subfigure}
		\\
		\begin{subfigure}[b]{0.3\textwidth}
			\centering
			\includegraphics[width=1.0\linewidth]{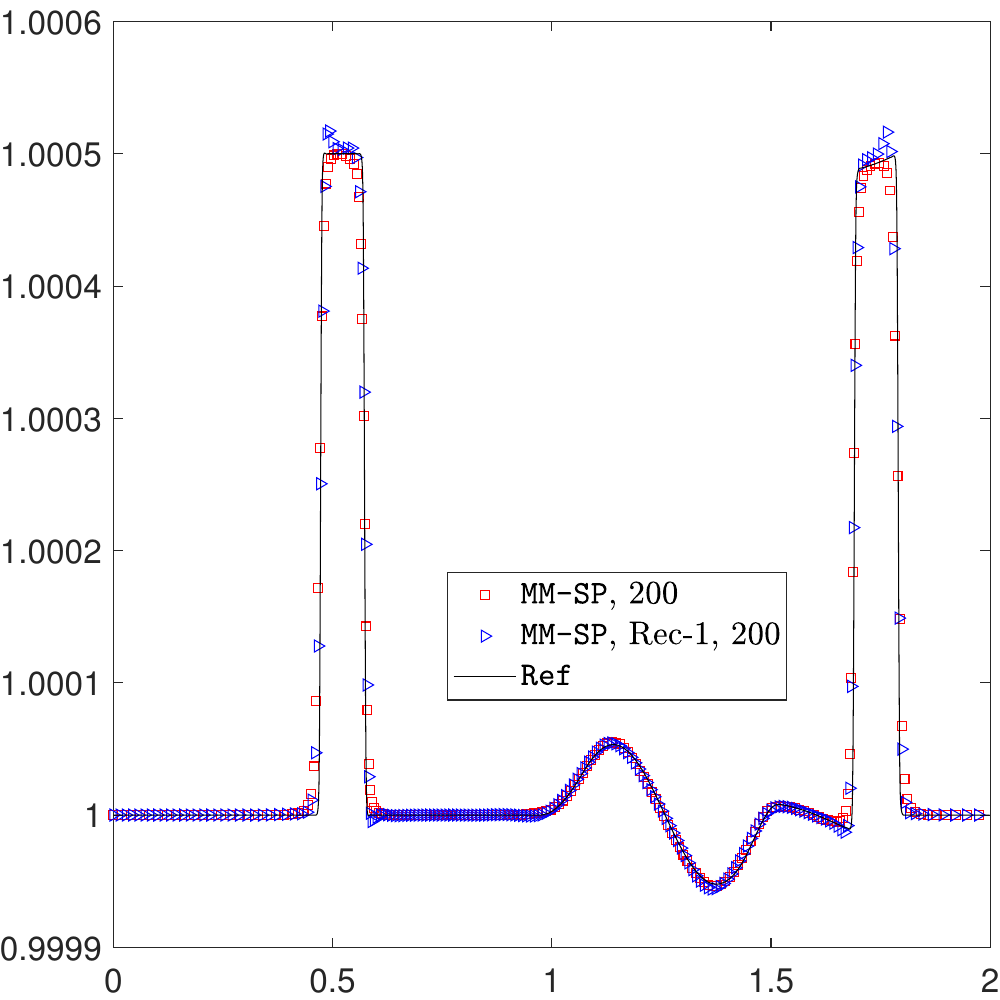}
		\end{subfigure}
		\begin{subfigure}[b]{0.15\textwidth}
			\centering
			\includegraphics[height=2.0\linewidth]{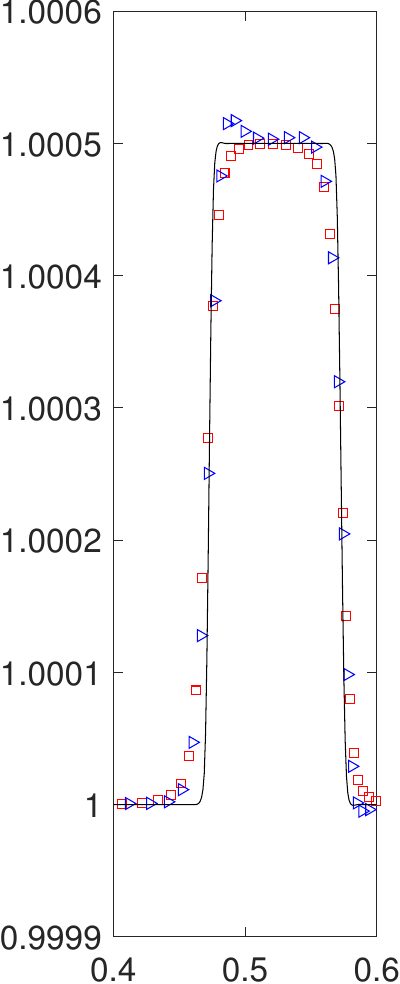}
		\end{subfigure}
		\quad
		\begin{subfigure}[b]{0.15\textwidth}
			\centering
			\includegraphics[height=2.0\linewidth]{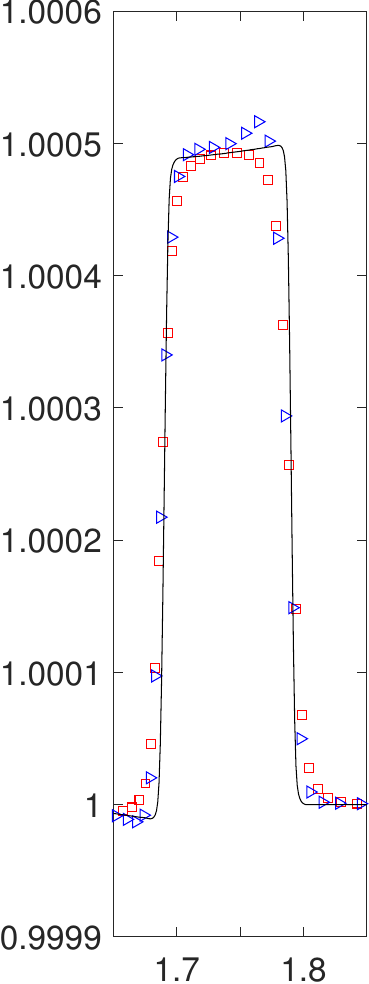}
		\end{subfigure}
		\caption{Example \ref{ex:1D_Pertubation_Test}. The results obtained by using the  \texttt{MM-SP} scheme and \texttt{MM-SP} scheme with {\tt Rec-1}. Left: the water surface level $h+b$,
			middle: the close view around $x = 0.5$, right: the close view around $x = 1.75$. }\label{fig:1D_Perturbation_Compare}
	\end{figure}

	\begin{figure}[!htb]
		\centering
		\begin{subfigure}[b]{0.32\textwidth}
			\centering
			\includegraphics[width=1.0\textwidth]{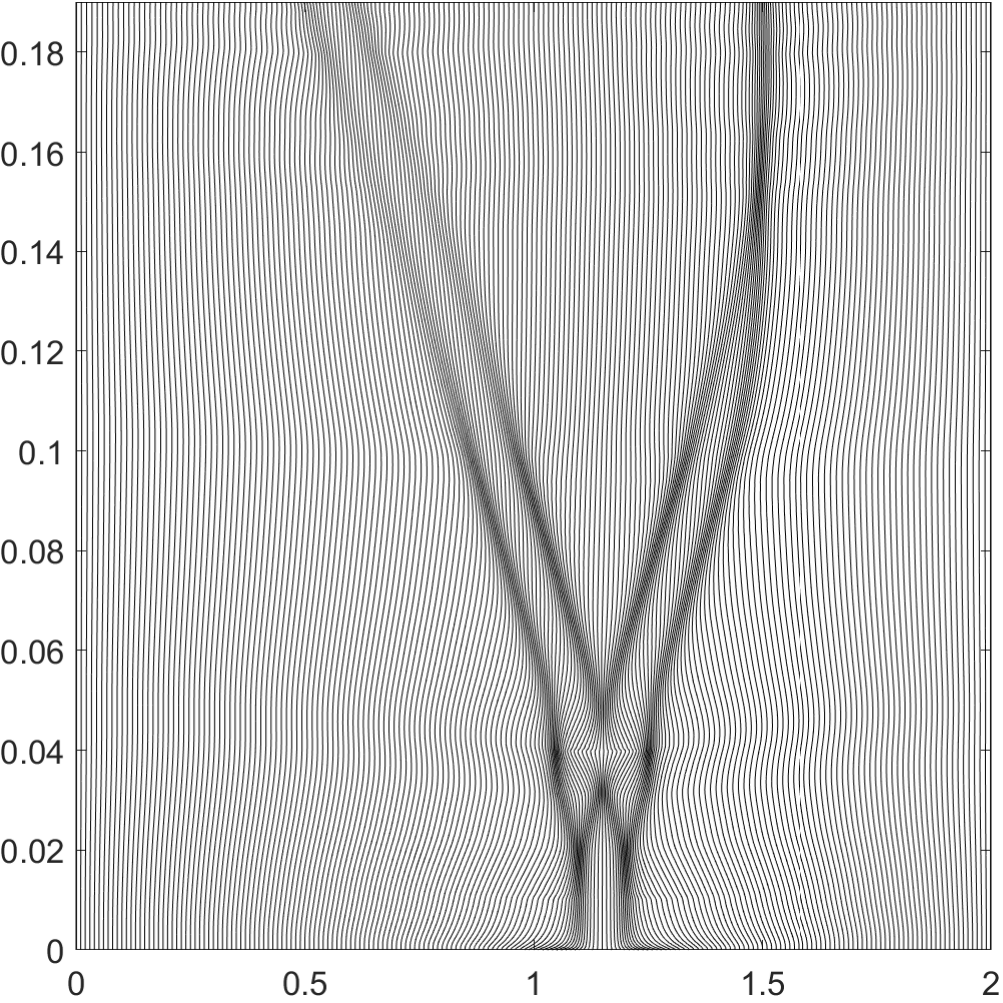}
		\end{subfigure}
		\begin{subfigure}[b]{0.32\textwidth}
			\centering
			\includegraphics[width=1.0\textwidth]{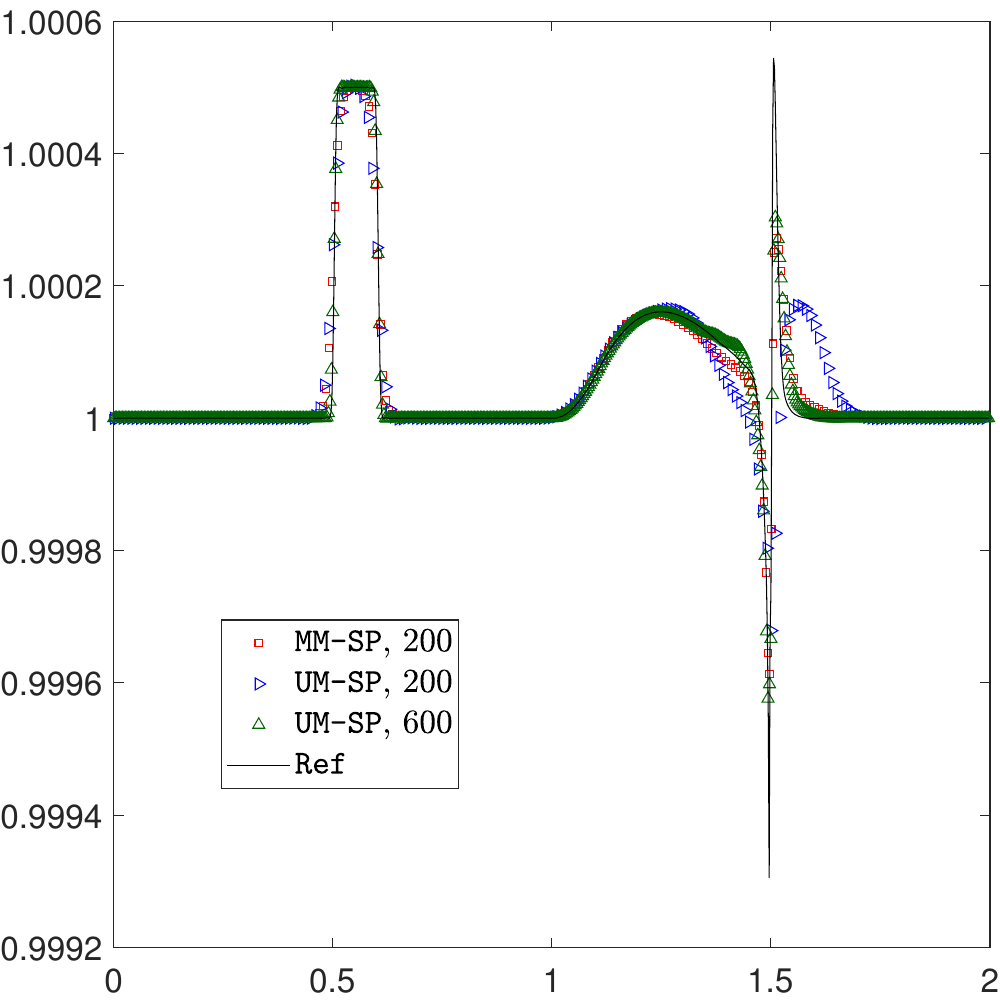}
		\end{subfigure}
		\begin{subfigure}[b]{0.32\textwidth}
			\centering
			\includegraphics[width=1.0\textwidth]{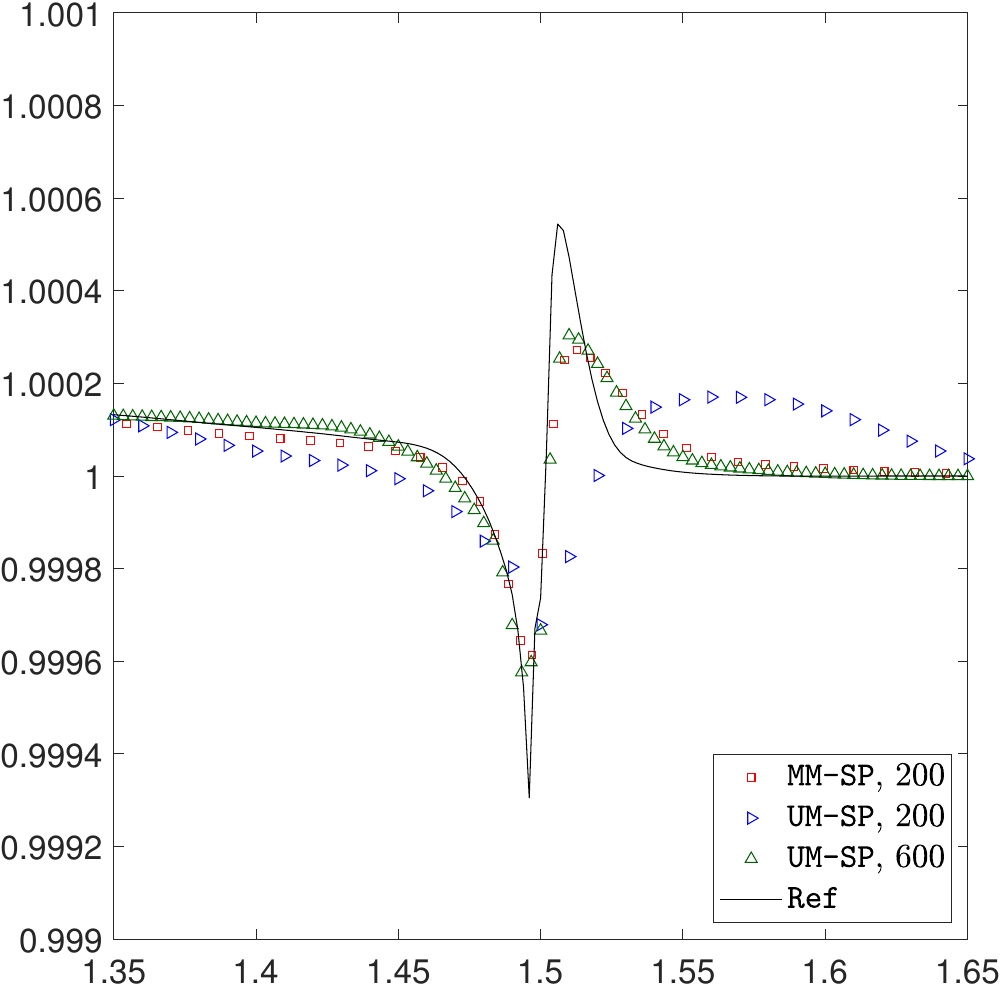}
		\end{subfigure}
		\caption{Example \ref{ex:1D_Pertubation_Test}.
			The numerical solutions obtained by using the \texttt{UM-SP} and \texttt{MM-SP} schemes at $t=0.19$.
			The reference solution is obtained by using the \texttt{UM-SP} scheme with $1000$ mesh points. Left: the mesh trajectory obtained by the \texttt{MM-SP} scheme, middle: the water surface level $h+b$,
			right: the close view around $x\in [1.35,1.65]$.
		}
		\label{fig:1D_Pertubation_WB_PP}
	\end{figure}

	\subsection{2D tests}
	
	\begin{example}[Accuracy test with moving vortex]\label{ex:2D_Smooth_PP}\rm
		The 2D vortex problem proposed in \cite{Zhang2023High} is used to verify the accuracy of our schemes with the PP limiter. A steady-state vortex is described by
		\begin{equation*}
			\begin{aligned}
				&h^{\prime}=h_{\max}-v_{\max }^2 e^{1-r^2} /(2 g), \\
				&\left(v_1^{\prime}, v_2^{\prime}\right)=v_{\max} e^{0.5\left(1-r^2\right)}(-x_2, x_1), \\
			\end{aligned}
		\end{equation*}
		where $h_{\max} = 10^{-6} + \frac{v_{\max}^2 e}{2g}, v_{\max} = 0.2, r = \sqrt{x_1^2 + x_2^2}, g = 1$, ensuring the minimum water depth is $10^{-6}$ so that the PP limiter is necessary for this test. 
		A moving vortex with a constant velocity of $(1,1)$ is then obtained by using the Galilean transformation
		\begin{equation*}
			\begin{aligned}
				& h(x_1, x_2, t)=h^{\prime}(x_1-t, x_2-t, t),~\left(v_1, v_2\right)(x_1, x_2, t)=(1,1)+\left(v_1^{\prime}, v_2^{\prime}\right)(x_1-t, x_2-t, t). \\
			\end{aligned}
		\end{equation*}
		The simulation runs up to $t = 0.1$ in the physical domain $[-10, 10] \times [-10, 10]$ with periodic boundary conditions. The monitor function used is
		\begin{equation*}
			\omega=\left({1+5 \left(\frac{|\nabla_{\bm{\xi}} (h+b)|}{\max |\nabla_{\bm{\xi}} (h+b)|}\right)^2+5\left(\frac{ |\Delta_{\bm{\xi}} (h+b)|}{\max |\Delta_{\bm{\xi}} (h+b)|}\right)^2}\right)^{1/2}.
		\end{equation*}
	\end{example}
	Figure \ref{fig:2D_Smooth_PP_Accuracy} displays the adaptive mesh and ten equally spaced contours of the water surface level $h+b$, computed using the \texttt{MM-SP} scheme with $80\times80$ cells. The figure also reports the $\ell^{1}$, $\ell^{2}$, and $\ell^{\infty}$ errors, along with the corresponding convergence behavior of the \texttt{MM-SP} scheme for the water depth $h$ at $t=0.1$. It can be observed that the mesh points concentrate around the vortex. The \texttt{MM-SP} scheme achieves fifth-order accuracy as expected, and the PP limiter does not compromise this accuracy.
	\begin{figure}[hbt!]
		\centering
		\begin{subfigure}[b]{0.4\textwidth}
			\centering
			\includegraphics[width=1.0\linewidth]{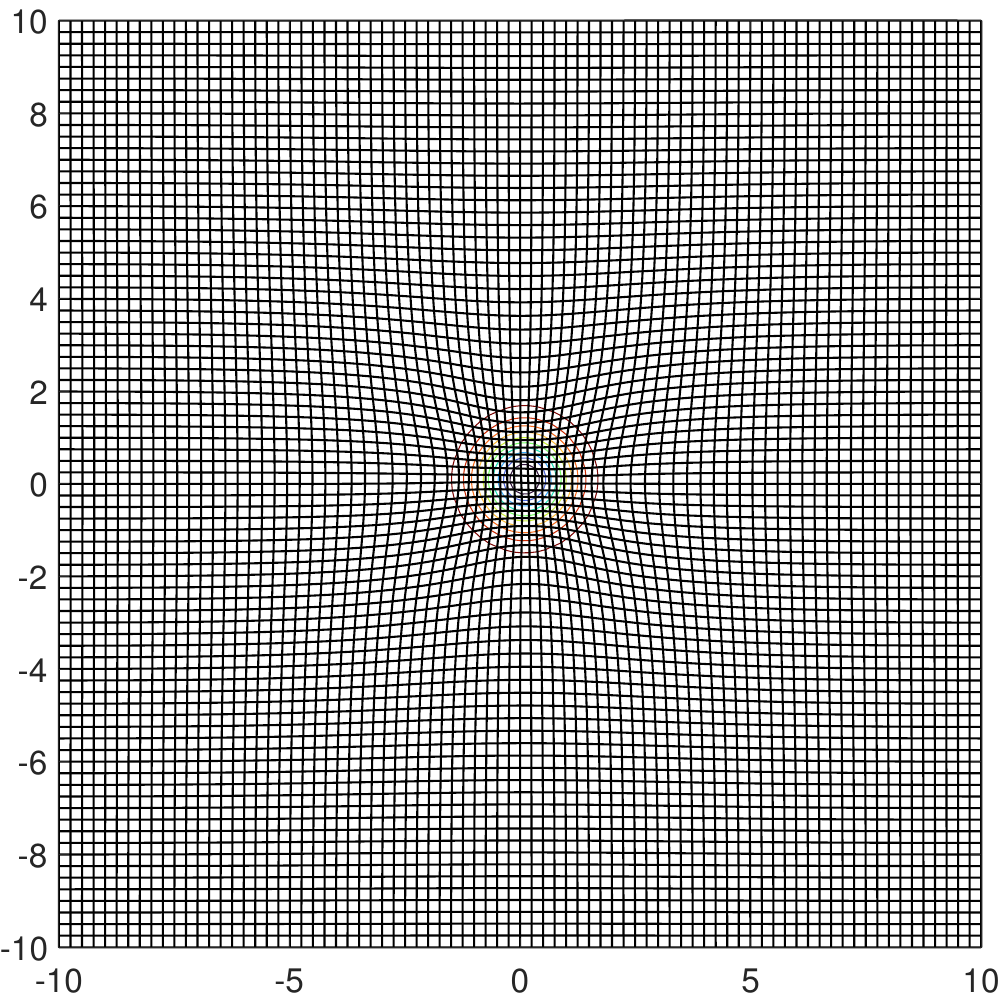}
		\end{subfigure}
		\begin{subfigure}[b]{0.4\textwidth}
			\centering
			\includegraphics[width=1.0\linewidth]{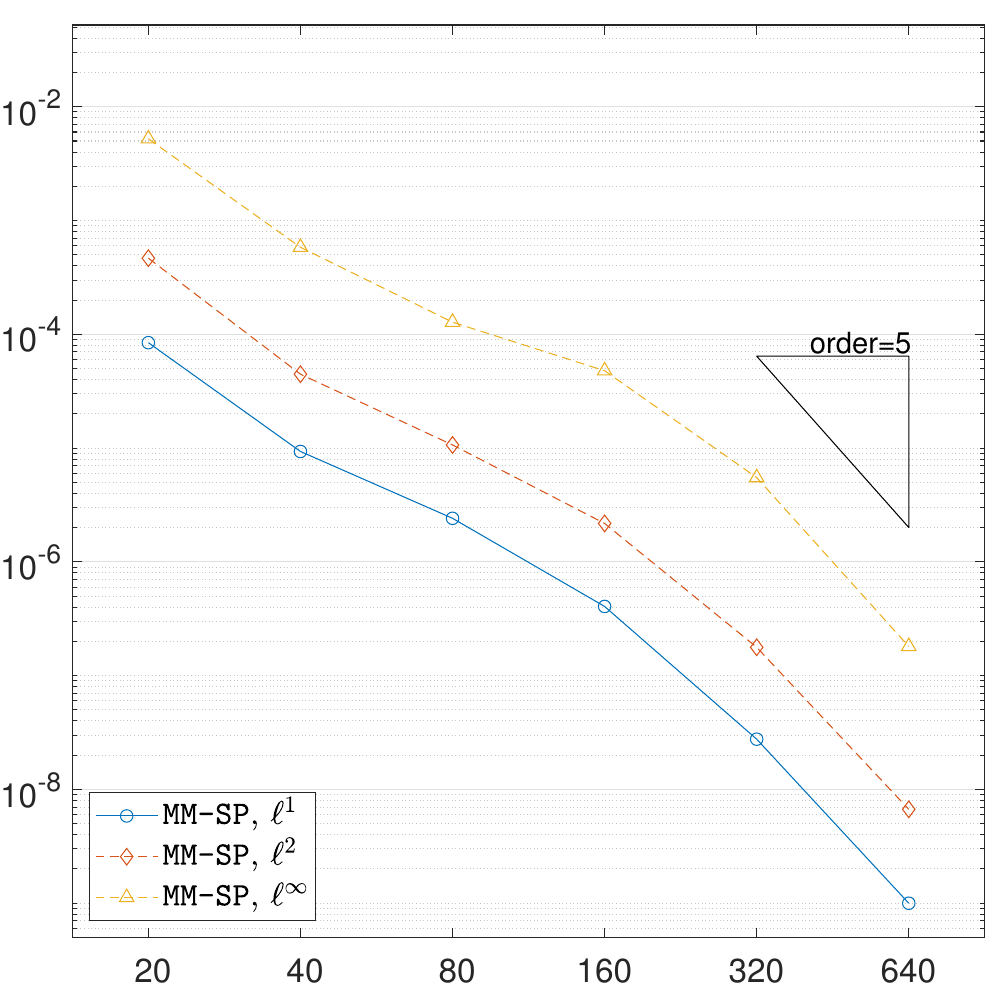}
		\end{subfigure}
		\caption{Example \ref{ex:2D_Smooth_PP}. The results computed by the \texttt{MM-SP} scheme at $t=0.1$. Left: the adaptive mesh and $10$ equally spaced contours of $h+b$ obtained with $80\times80$ cells. Right: the $\ell^{1}$, $\ell^{2}$, and $\ell^{\infty}$ errors and convergence behavior for $h$.}\label{fig:2D_Smooth_PP_Accuracy}
	\end{figure}

	\begin{example}[2D PP test]\label{ex:2D_PP_Flat}\rm
		This numerical experiment is performed to verify the PP property of the proposed schemes on fixed and moving meshes \cite{Xing2013Positivity}.
		The physical domain is $[-100,100]\times[-100,100]$ with outflow boundary conditions.  
		The initial conditions for water depth and velocities are given by
		\begin{equation*}
			\begin{aligned}
				& h = \begin{cases}
					10, & \text{if}~\sqrt{(x_1)^2+(x_2)^2} \leqslant 60, \\
					0, & \text {otherwise},
				\end{cases} \\
				& v_1=v_2=0.
			\end{aligned}
		\end{equation*}
		The simulation runs until an output time of $t=0.35$. The gravitational acceleration constant is set to $g = 9.812$. 
		The monitor function is chosen as
		\begin{equation*}
			\omega = \left(1+\theta\left(\frac{\left|\nabla_{\bm{\xi}}\sigma\right|}{\max\left|\nabla_{\bm{\xi}}\sigma\right|}\right)^2\right)^{\frac{1}{2}},
		\end{equation*}
		where $\theta = 50$ and $\sigma = h+b$.
		
	\end{example}
	
	Figure \ref{fig:2D_PP_Flat} presents the numerical results obtained by using the \texttt{UM-SP} scheme with $100\times 100$ and $300\times 300$ cells, and the \texttt{MM-SP} scheme with $100\times 100$ cells. These results demonstrate that our schemes, applied to either fixed or moving meshes, work effectively for challenging scenarios where the water depth includes a dry region. Cut lines along $x_2=1$ reveal that the numerical results from the \texttt{MM-SP} scheme are superior to those obtained by the \texttt{UM-SP} scheme with the same number of mesh points. Notably, the numerical simulation fails rapidly without the PP limiter, given that the minimum value of water depth $h$ reaches $0$.

	\begin{figure}[hbt!]
		\centering
		\begin{subfigure}[b]{0.3\textwidth}
			\centering
			\includegraphics[width=1.0\linewidth]{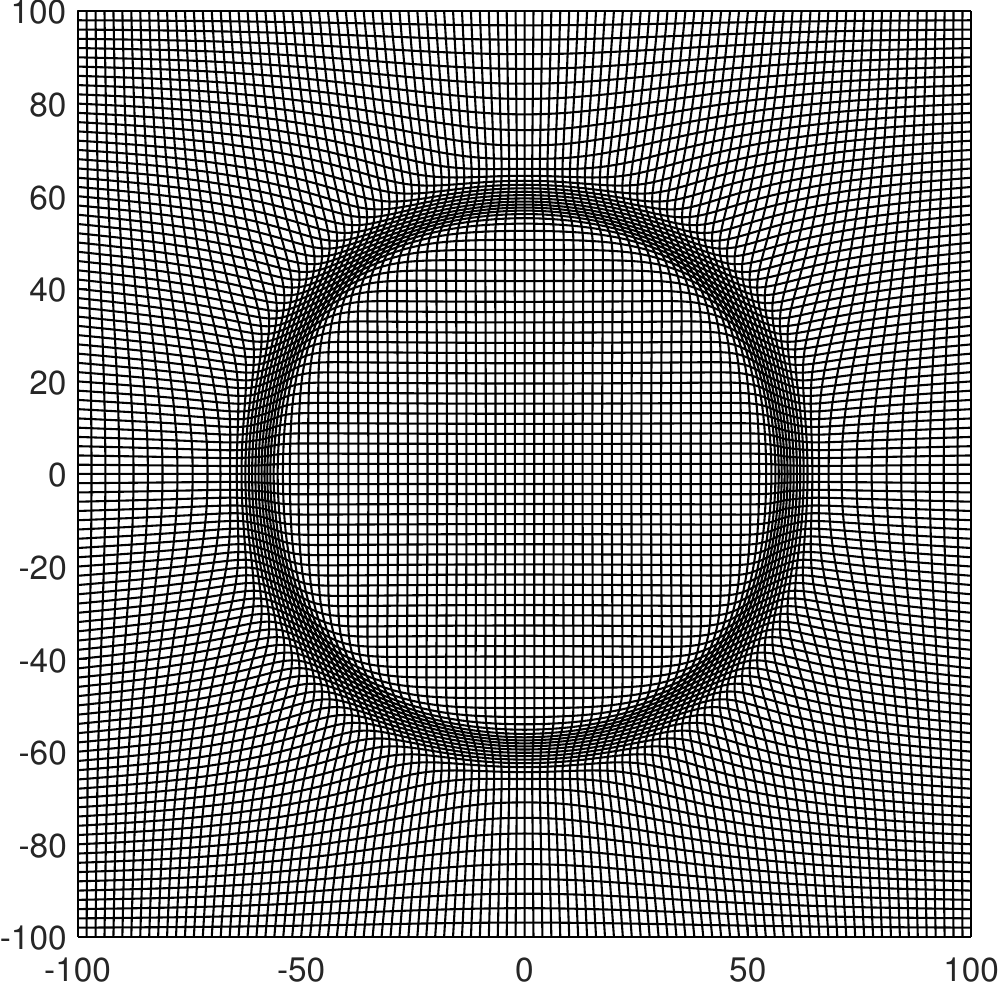}
			\caption{$200\times200$ adaptive mesh, \texttt{MM-SP}}
		\end{subfigure}
		\begin{subfigure}[b]{0.3\textwidth}
			\centering
			\includegraphics[width=1.0\linewidth]{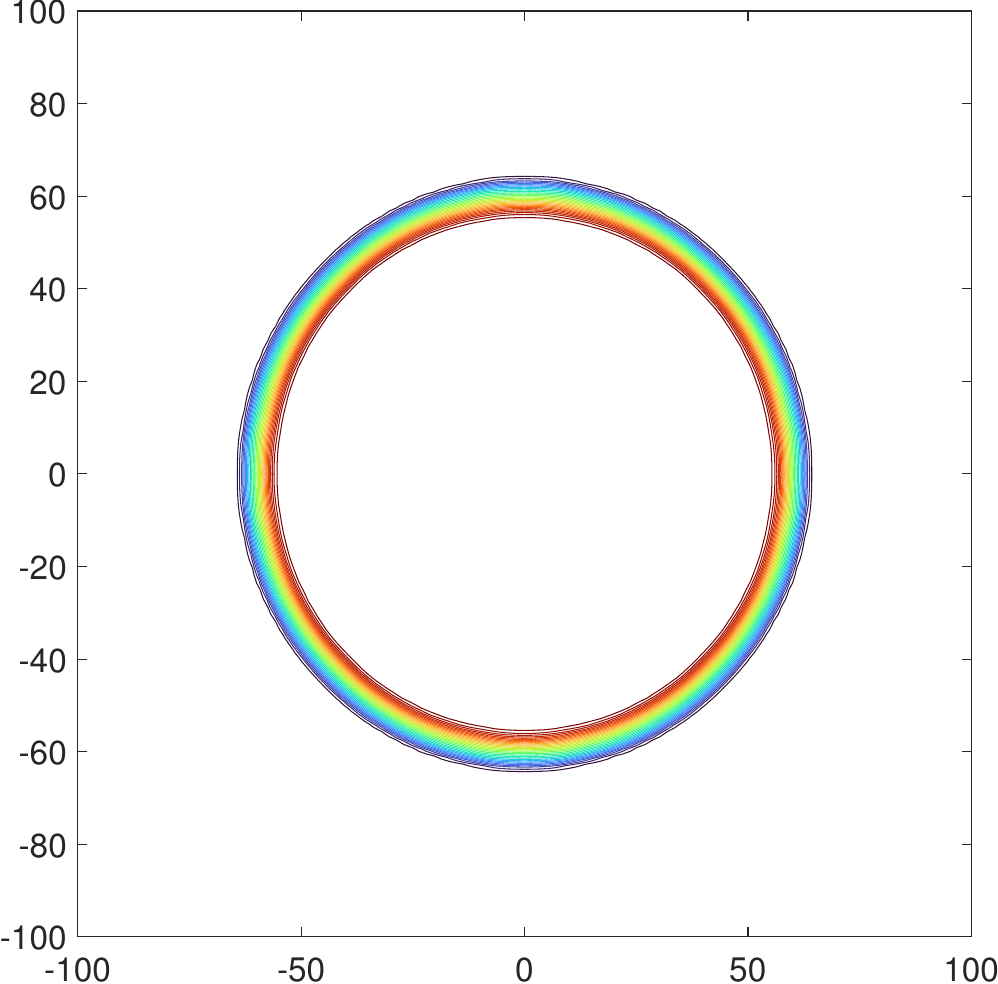}
			\caption{$h+b$, \texttt{MM-SP}}
		\end{subfigure}
		\quad
		\begin{subfigure}[b]{0.3\textwidth}
			\centering
			\includegraphics[width=1.0\linewidth]{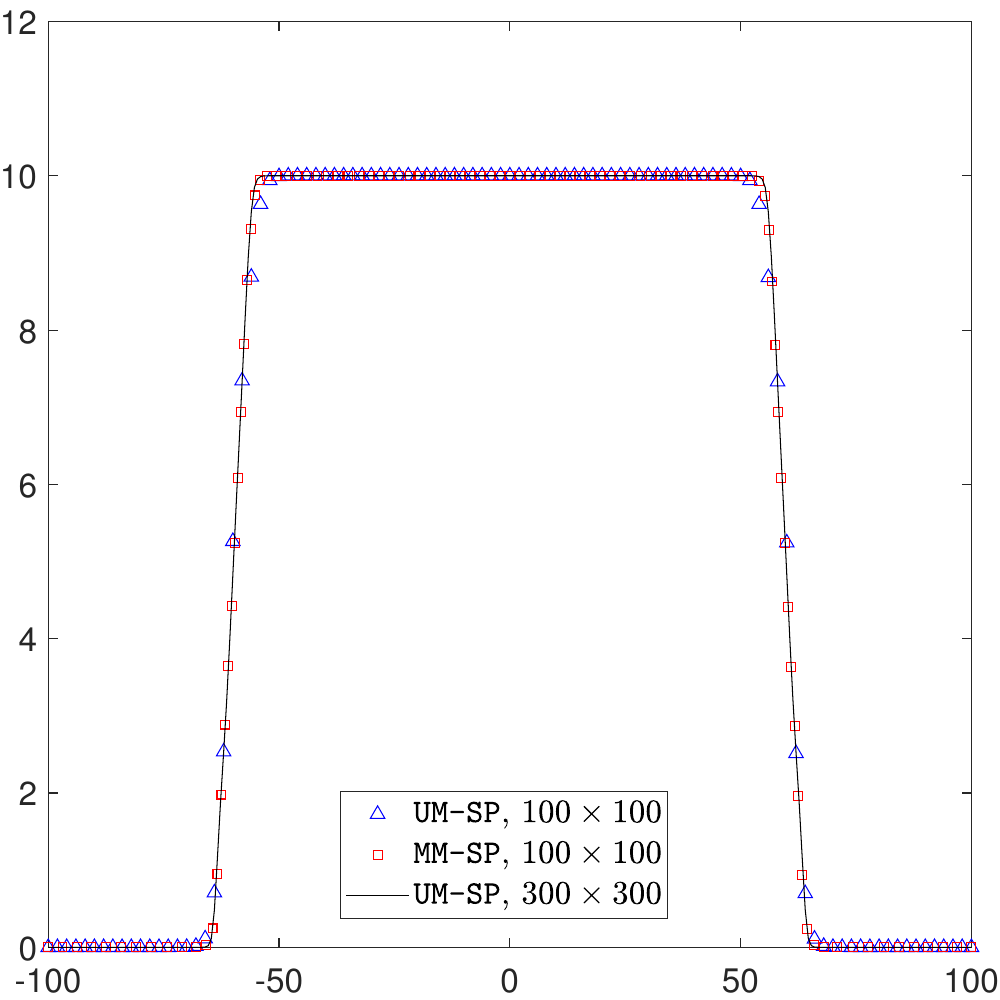}
			\caption{cut lines along $x_2=1$}
		\end{subfigure}
		\caption{Example \ref{ex:2D_PP_Flat}. The results obtained by using the \texttt{UM-SP} and \texttt{MM-SP} schemes.}\label{fig:2D_PP_Flat}
	\end{figure}

	\begin{example}[2D PP test for the bowl]\label{ex:2D_PP_Bowl}\rm
		This numerical experiment proposed in \cite{vater2019limiter} is conducted to further validate the PP property of our schemes with non-flat bottom topography on both fixed and moving meshes.
		The physical domain is $[-4000,4000]\times[-4000,4000]$ with outflow boundary conditions. The initial bottom topography is taken as 
		\begin{equation*}
			b(x_1,x_2) = \dfrac{h_0}{a^2}(x_1^2+x_2^2). 
		\end{equation*}
		The initial velocity are zero, and 
		the initial water depth is given by 
		\begin{equation*}
			h = \max\left\{0, h_0\left(\dfrac{\sqrt{1-A^2}}{1-A} - \dfrac{x^2+y^2}{a^2}\dfrac{1-A^2}{\left(1-A\right)^2}\right)\right\},
		\end{equation*}
		where $A = \frac{a^4-r_0^4}{a^4+r_0^4}$, $h_0 = 1$, $r_0 = 2000$, and $a = 2500$. This example has the analytic solution as follows
		\begin{align*}
			&h(x_1,x_2,t) = \max\left\{0, h_0\left(\dfrac{\sqrt{1-A^2}}{1-A\cos{\kappa t}} - \dfrac{x^2+y^2}{a^2}\dfrac{1-A^2}{\left(1-A\cos{\kappa t}\right)^2}\right)\right\},\\
			&v_1(x_1,x_2,t) = \dfrac{\kappa A \sin{\kappa t}}{2(1-A\cos{\kappa t})}x_1
			,~
			v_2(x_1,x_2,t) = \dfrac{\kappa A \sin{\kappa t}}{2(1-A\cos{\kappa t})}x_2.
		\end{align*}
		The simulations are performed up to two output times $t=925$ and $t=2775$. 
		The gravitational acceleration constant
		is set as $g = 9.812$.
		The monitor function is the same as in Example \ref{ex:2D_PP_Flat}.
	\end{example}
	
	Figure \ref{fig:2D_PP_Bowl} presents the numerical results using both the \texttt{UM-SP} and \texttt{MM-SP} schemes, each with $100 \times 100$ cells. It includes plots of the adaptive mesh, the water surface level $h+b$, and cross-sectional cut lines along $x_2 = 0$. These results demonstrate that our schemes effectively  resolve the shallow water dynamics on non-flat bottom topography and closely match the exact solution. The minimum cell-average value of water depth $h$ reaches $0$. Without the PP limiter, the numerical simulation would fail at the first time step.

	\begin{figure}[hbt!]
		\centering
		\begin{subfigure}[b]{0.3\textwidth}
			\centering
			\includegraphics[width=1.0\linewidth]{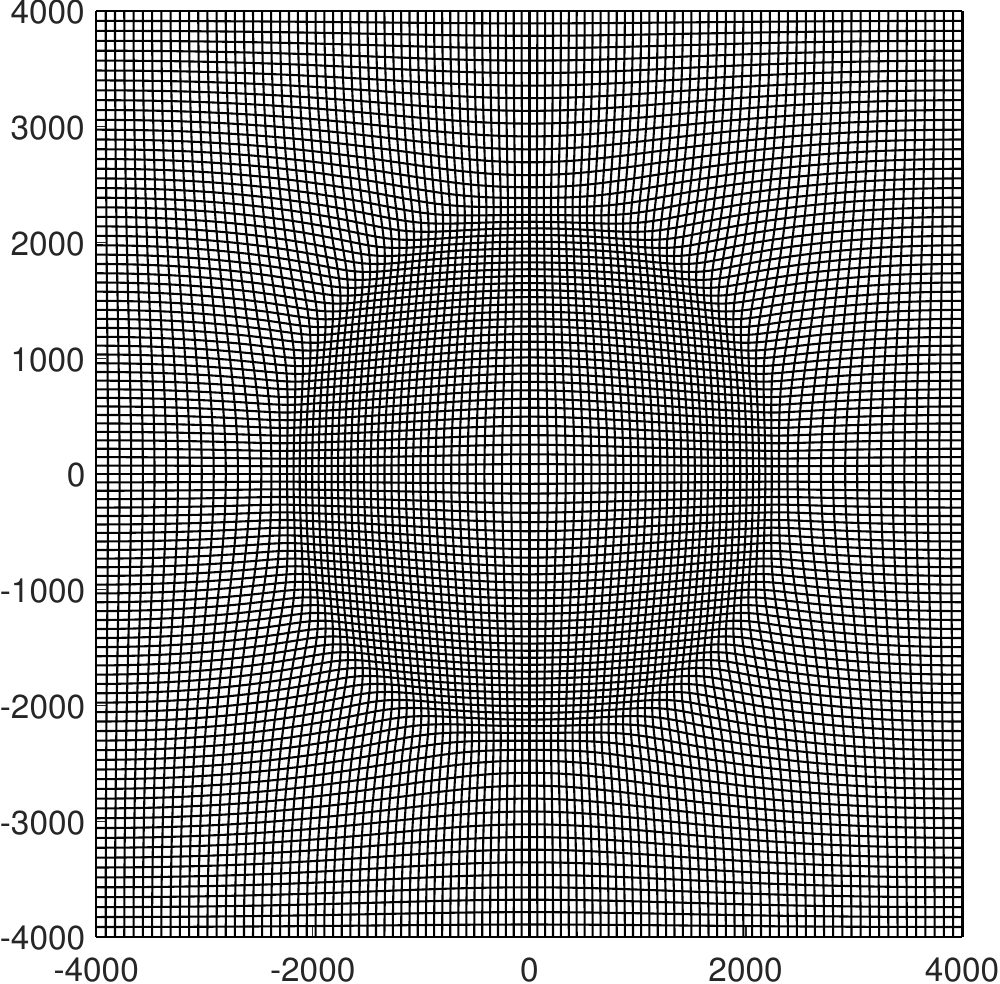}
			\caption{$100\times100$ adaptive mesh, \texttt{MM-SP}}
		\end{subfigure}
		\begin{subfigure}[b]{0.33\textwidth}
			\centering
			\includegraphics[width=1.0\linewidth]{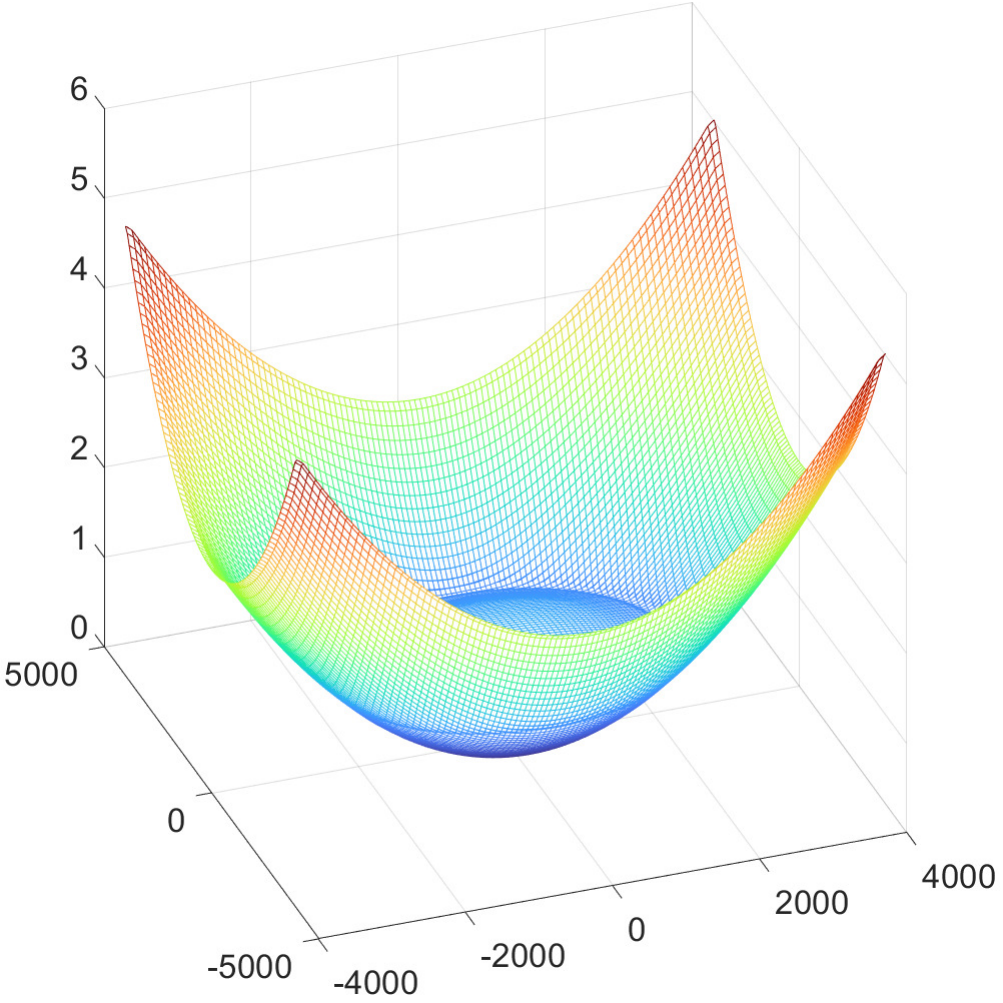}
			\caption{$h$, \texttt{MM-SP}}
		\end{subfigure}
		\quad
		\begin{subfigure}[b]{0.3\textwidth}
			\centering
			\includegraphics[width=1.0\linewidth]{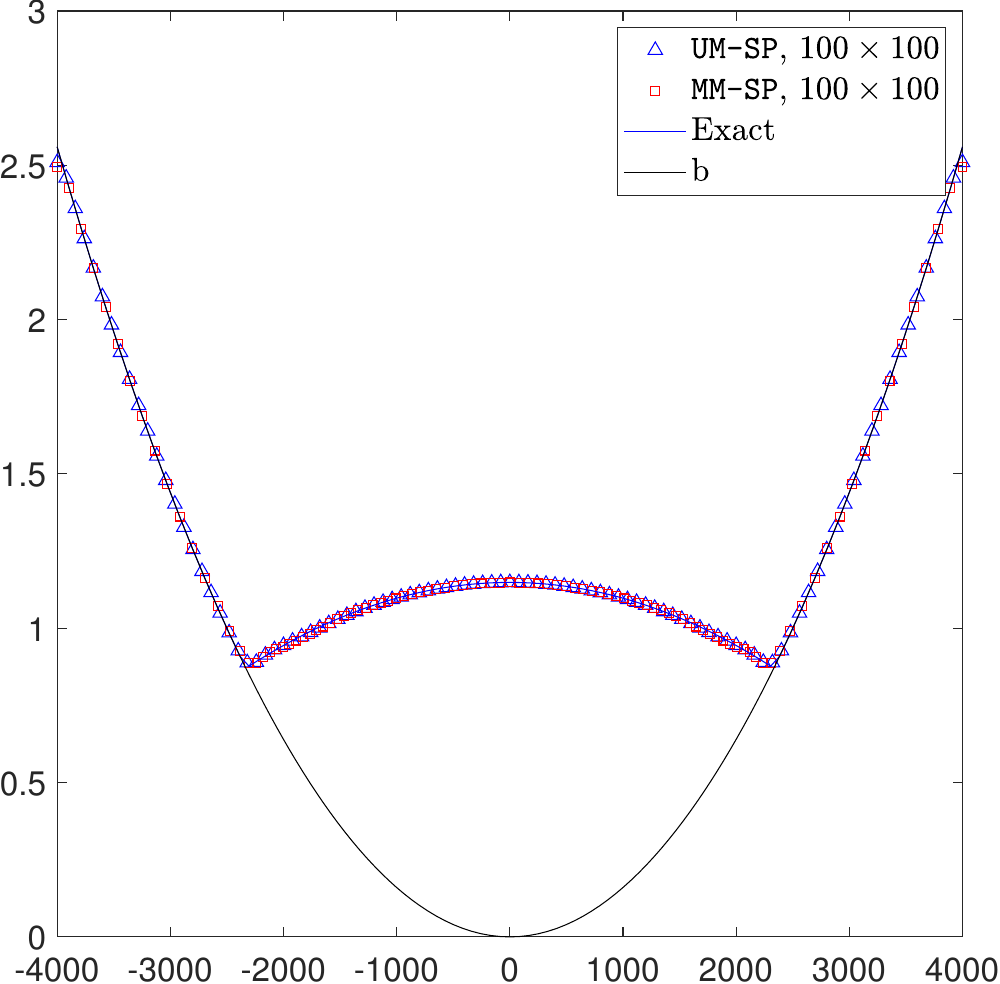}
			\caption{cut lines along $x_2=0$}
		\end{subfigure}
		\begin{subfigure}[b]{0.3\textwidth}
			\centering
			\includegraphics[width=1.0\linewidth]{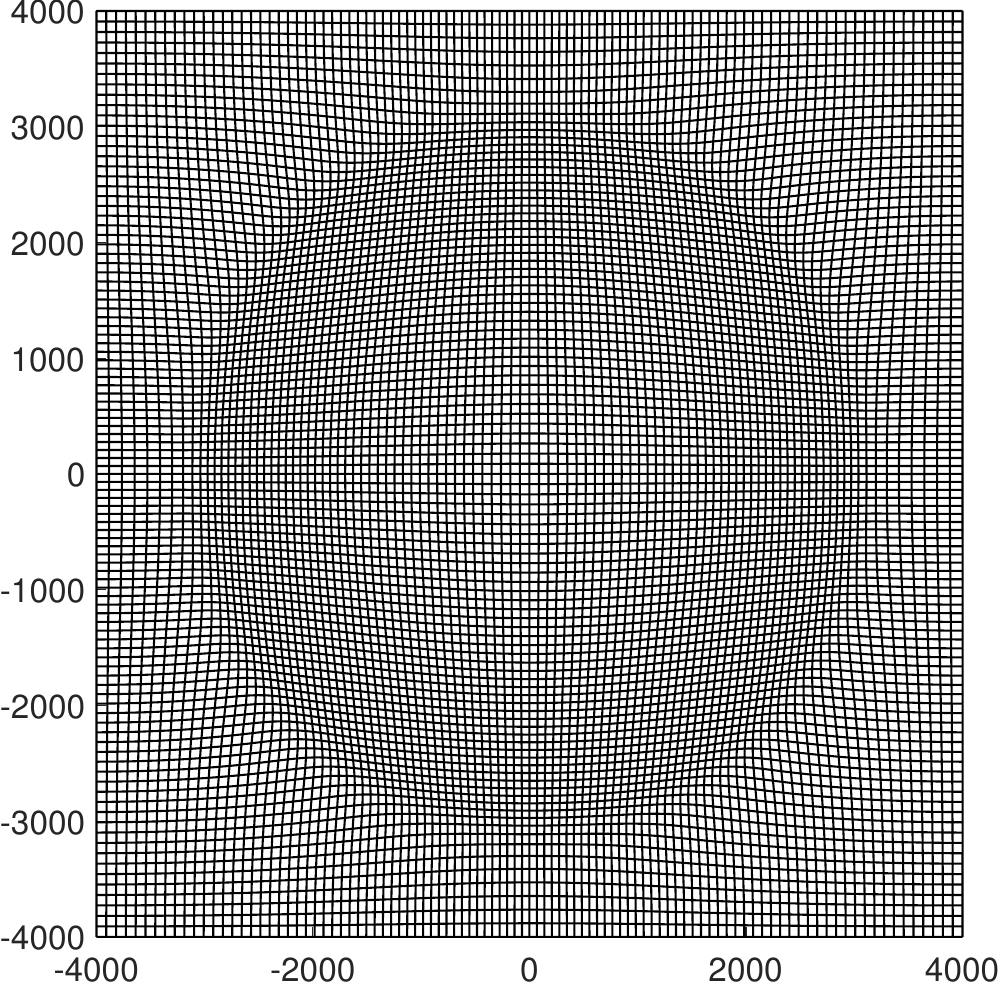}
			\caption{$100\times100$ adaptive mesh, \texttt{MM-SP}}
		\end{subfigure}
		\begin{subfigure}[b]{0.33\textwidth}
			\centering
			\includegraphics[width=1.0\linewidth]{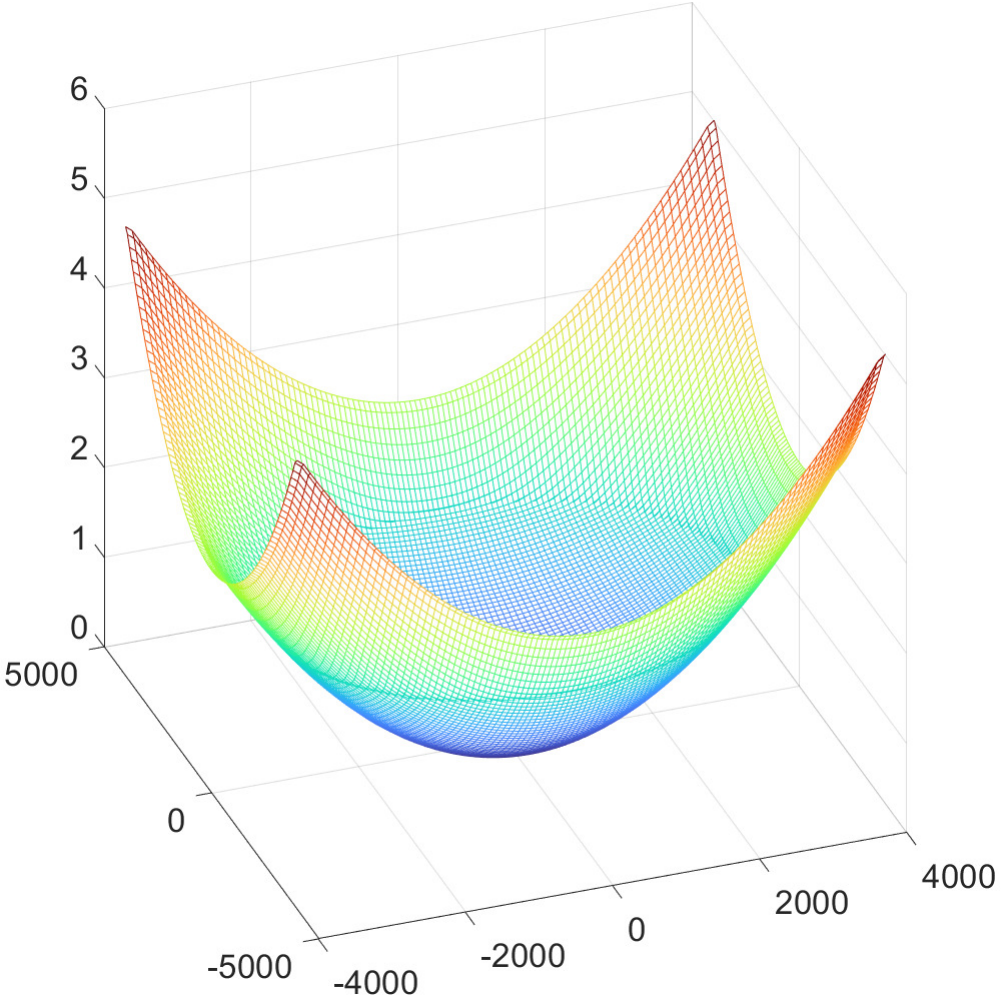}
			\caption{$h$, \texttt{MM-SP}}
		\end{subfigure}
		\quad
		\begin{subfigure}[b]{0.3\textwidth}
			\centering
			\includegraphics[width=1.0\linewidth]{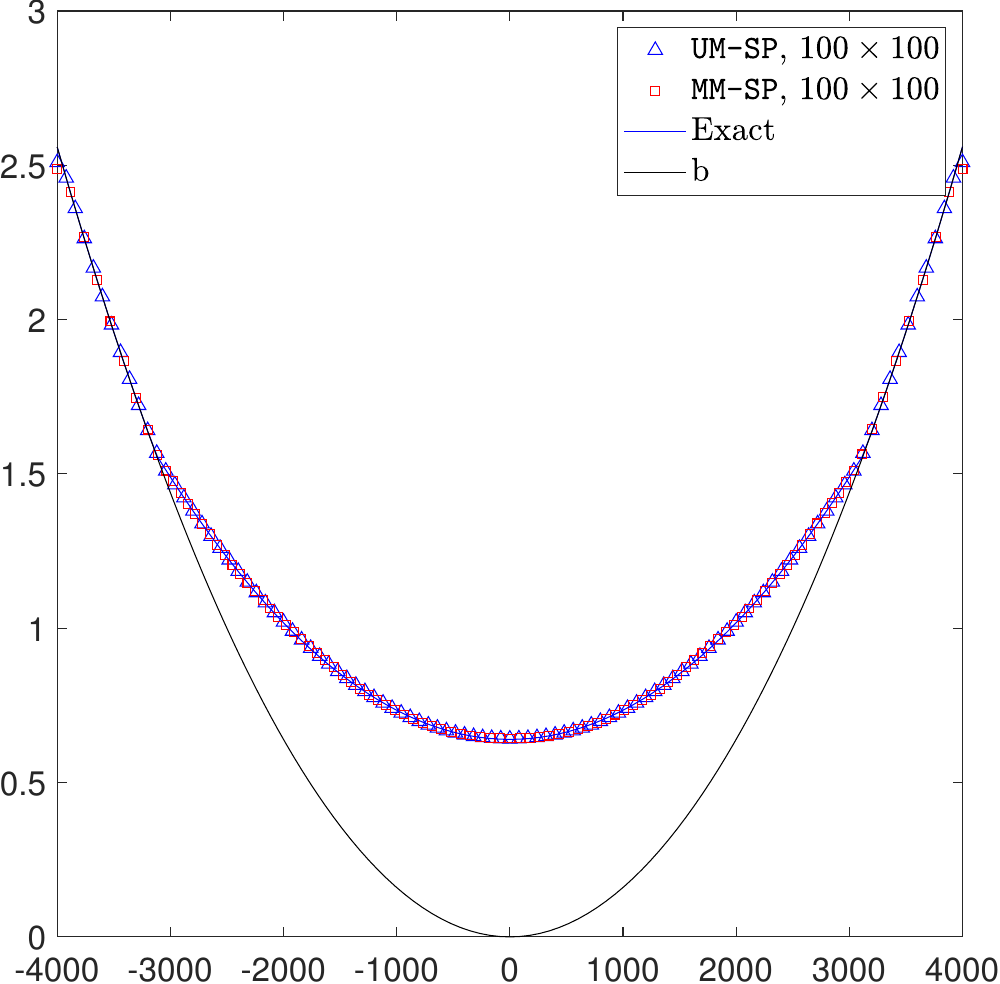}
			\caption{cut lines along $x_2=0$}
		\end{subfigure}
		\caption{Example \ref{ex:2D_PP_Bowl}. The results obtained by using the \texttt{UM-SP} and \texttt{MM-SP} schemes. Top: $t=925$, bottom: $t=2775$.}\label{fig:2D_PP_Bowl}
	\end{figure}
	
	\begin{example}[2D WB test]\label{ex:2D_WB_Test}\rm
		This example is designed to test the WB properties of our 2D schemes. The setup is similar to that used in Example \ref{ex:1D_WB_Test}, with two types of bottom topography:  
		\begin{equation}\label{eq:2D_b_Smooth}
			b(x_1, x_2) = 0.8 \exp\left(-50\left((x_1-0.5)^2+(x_2-0.5)^2\right)\right),~ (x_1,x_2) \in [0,1]\times[0,1],
		\end{equation}
		for the smooth topography, and
		\begin{equation}\label{eq:2D_b_dis}
			b(x_1, x_2) = \begin{cases}
				0.5 , &\text{if}\quad  (x_1,x_2) \in [0.5,1.0]\times[0.5,1.0],\\
				0, &\text{otherwise},
			\end{cases}
		\end{equation}
		for the discontinuous topography. The initial water depth and velocities are given by 
		\begin{equation*}
			h = 1-b,~v_1=v_2 = 0.
		\end{equation*}
		The physical domain is $[0,1] \times [0,1]$ with outflow boundary conditions, and the gravitational acceleration constant is $g = 1$. To focus the mesh points on the areas with sharp transitions in $b$, the monitor function is chosen as
		\begin{equation*}
			\omega=\sqrt{1+\theta\left(\frac{\left|\nabla_{\bm{\xi}} \sigma\right|}{ \max \left|\nabla_{\bm{\xi}} \sigma\right|}\right)^2},
		\end{equation*}
		where $\sigma$ is the water depth $h$ and $\theta = 100$.
	\end{example}

	Table \ref{tab:2D_Well_Balance} lists the 
	$\ell^1$ and $\ell^{\infty}$ errors in the water surface levels and velocities at $t=0.1$ computed by using our schemes with $100\times100$ mesh.  
	These results confirm that the lake-at-rest condition is maintained within the rounding errors of double precision, 
	thus verifying that our schemes are WB on both fixed and adaptive moving meshes. 
	Figure \ref{fig:2D_Well_Balance} displays the results of the water surface levels $h+b$ and the corresponding bottom topography $b$, obtained using the \texttt{MM-SP} scheme with a $100\times100$ mesh. The plots clearly show a concentration of mesh points around areas where $b$ exhibits large gradients. 
	
	\begin{table}[hbt!]
		\centering
		\caption{Example \ref{ex:2D_WB_Test}. Errors in $h+b$, $v_1$, and $v_2$ by using our schemes with $100\times100$ meshes at $t=0.1$,
			for the bottom topographies \eqref{eq:2D_b_Smooth} and \eqref{eq:2D_b_dis}.}\label{tab:2D_Well_Balance}
		\begin{tabular}{cccccc}
			\hline\hline
			\multicolumn{2}{c}{\multirow{2}{*}{}} & \multicolumn{2}{c}{\texttt{UM-SP}} & \multicolumn{2}{c}{\texttt{MM-SP}}  \\ 
			\cmidrule(lr){3-4}\cmidrule(lr){5-6}
			\multicolumn{2}{c}{} & \multicolumn{1}{c}{$\ell^{1}$~error}  & \multicolumn{1}{c}{$\ell^{\infty}$~error} &  \multicolumn{1}{c}{$\ell^{1}$~error}  & \multicolumn{1}{c}{$\ell^{\infty}$~error}   \\
			\hline
			\multirow{3}{*}{$b$ in \eqref{eq:2D_b_Smooth}} &
			$h+b$ & 6.45e-21 	 & 1.12e-15 	 & 6.85e-21 	 & 2.01e-15  \\ 
			&  $v_1$ & 3.24e-15 	 & 2.43e-14 	 &6.84e-16 	 & 1.38e-14   \\ 
			&$v_2$ & 3.38e-15 	 &  2.58e-14	 & 6.92e-16 	 &  1.56e-14   \\ 
			\hline
			\multirow{3}{*}{$b$ in \eqref{eq:2D_b_dis}} &
			$h+b$ & 5.55e-21 	 & 8.33e-16 	 & 9.25e-21 	 &1.83e-15  \\ 
			&  $v_1$ &3.19e-15 	 &2.51e-14 	 & 6.90e-16 	 & 1.64e-14   \\ 
			&$v_2$ & 3.35e-15 	 &  2.60e-14	 & 7.13e-16	 &   1.86e-14   \\ 
			\hline\hline
		\end{tabular}
	\end{table}
	\begin{figure}[hbt!]
		\centering
		\begin{subfigure}[b]{0.39\textwidth}
			\centering
			\includegraphics[width=1.0\linewidth]{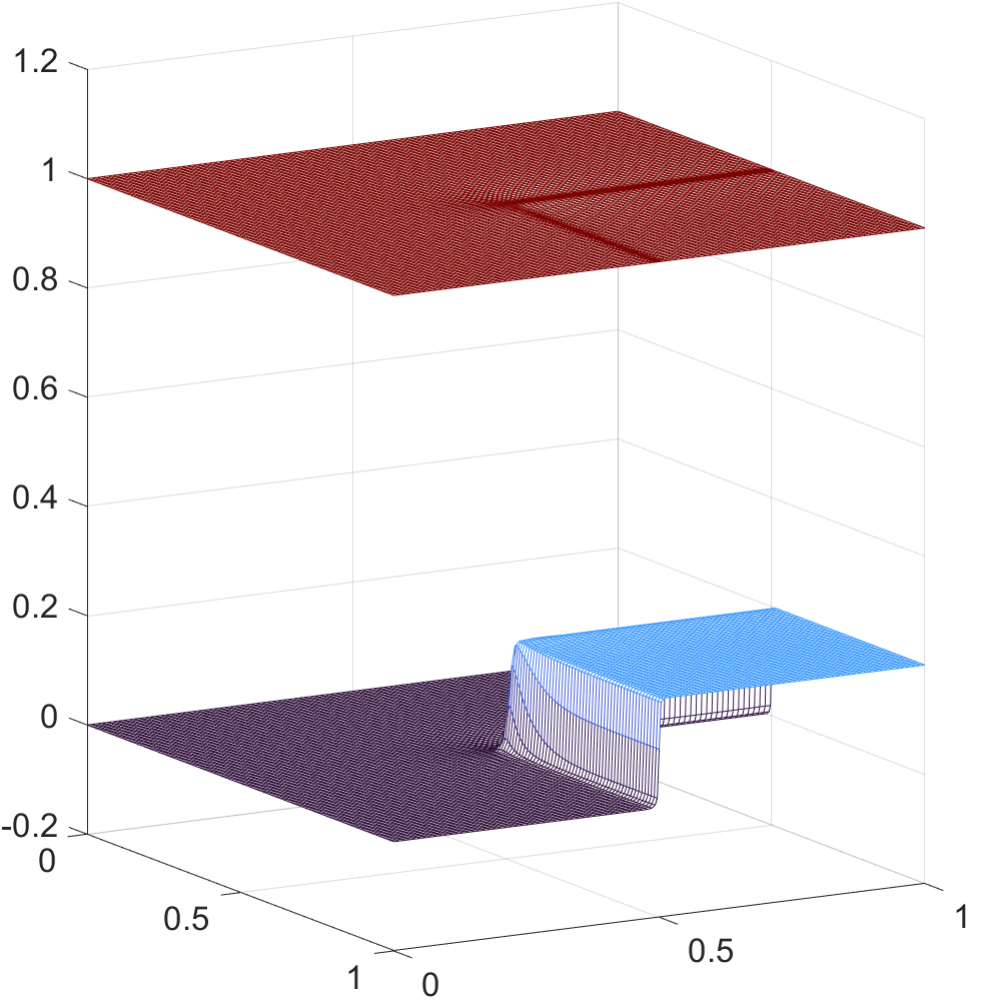}
		\end{subfigure}
		\begin{subfigure}[b]{0.35\textwidth}
			\centering
			\includegraphics[width=1.0\linewidth]{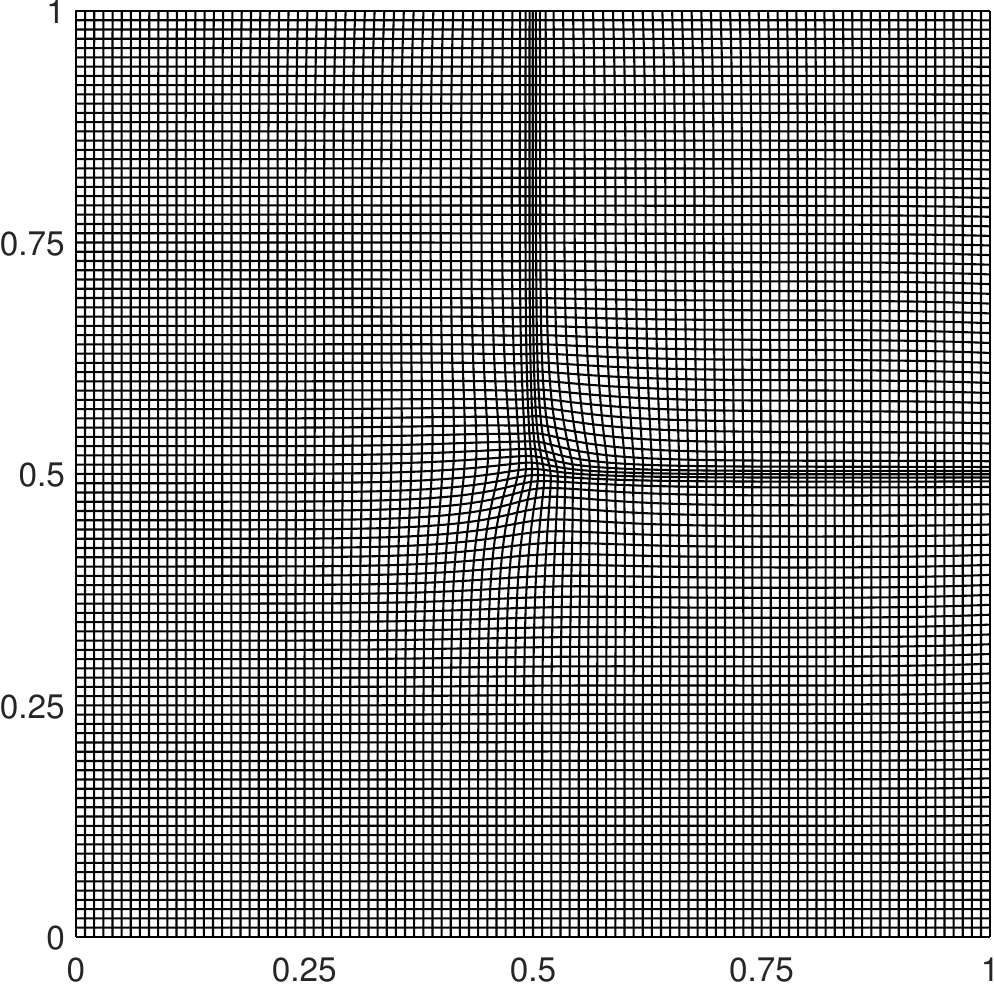}
		\end{subfigure}
		\\
		\begin{subfigure}[b]{0.39\textwidth}
			\centering
			\includegraphics[width=1.0\linewidth]{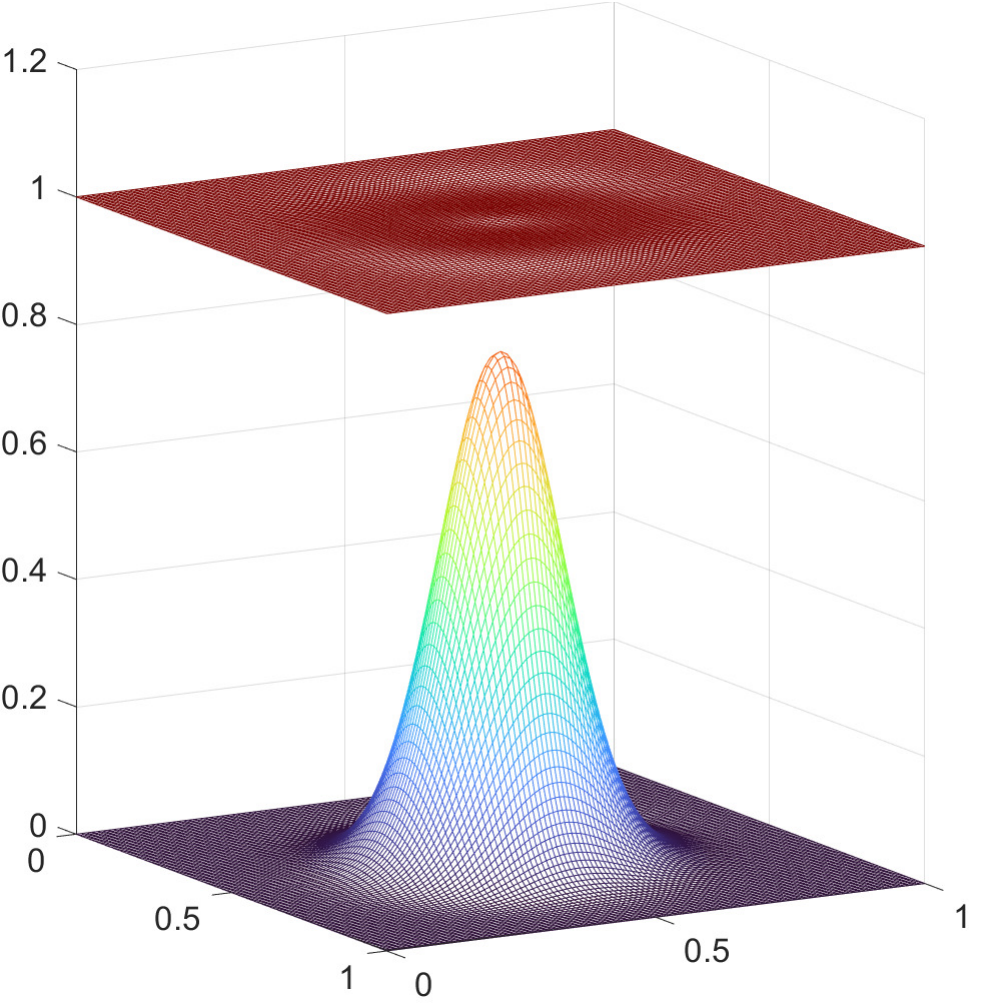}
		\end{subfigure}
		\begin{subfigure}[b]{0.35\textwidth}
			\centering
			\includegraphics[width=1.0\linewidth]{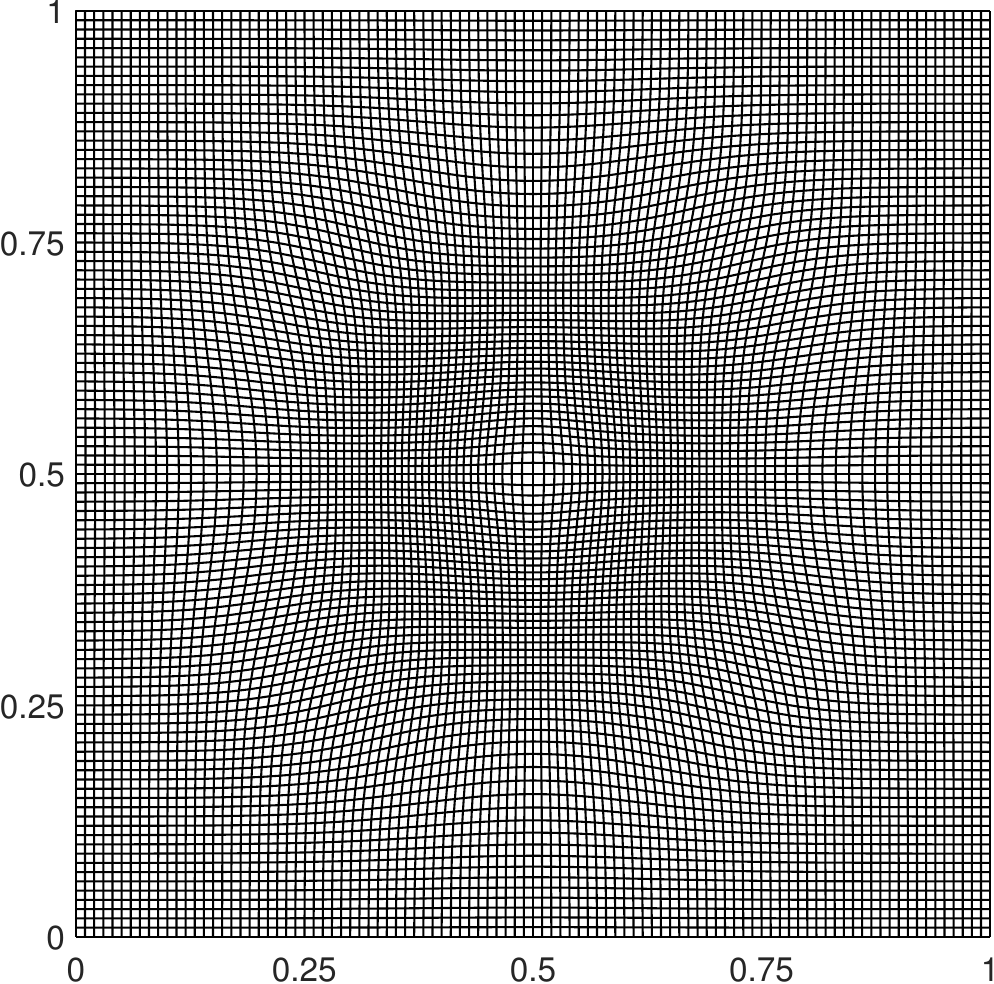}
		\end{subfigure}
		\caption{Example \ref{ex:2D_WB_Test}. The water surface level $h+b$, bottom topography $b$, and adaptive meshes obtained by using the \texttt{MM-SP} scheme with $100 \times 100$ mesh at $t=0.1$.
			The figures in the first and second rows correspond to the cases with the bottom topographies \eqref{eq:2D_b_Smooth} and \eqref{eq:2D_b_dis}, respectively.}
		\label{fig:2D_Well_Balance}
	\end{figure}

	\begin{example}[The perturbed flow in lake at rest]\label{ex:2D_Pertubation}\rm
		This example tests the capability of our schemes in capturing small perturbations over a lake at rest within the domain $[0,2] \times [0,1]$ with outflow boundary conditions. The bottom topography, described as an ``oval hump," is referenced in \cite{Xing2017Numerical,Xing2005High}. The initial conditions are as follows:
		\begin{align*}
			&h= \begin{cases}1.01-b,~&\text{if}\quad x_1 \in[0.05, 0.15], \\
				1-b,~&\text{otherwise},
			\end{cases} \\
			& v_1=v_2=0,\\
			&b=0.8 \exp(-5(x_1-0.9)^2-50(x_2-0.5)^2).
		\end{align*}
		The initial disturbance splits into two waves propagating at speeds of $\pm \sqrt{gh}$, with $g=9.812$. The monitor function used is identical to that in Example \ref{ex:2D_PP_Flat}, but here $\sigma = \log(h+b)$ and $\theta=250$.
	\end{example}
	Figure \ref{fig:2D_Perturbation_Test} displays the contour plot of the water surface level $h+b$ and the adaptive meshes, with cut lines along $x_2 = 0.5$ at $t=0.24$ and $t=0.36$. The results confirm that our schemes accurately capture small features, focusing adaptive mesh points on these areas to enhance resolution. When compared to results obtained by the \texttt{UM-SP} scheme with a finer mesh, the \texttt{MM-SP} scheme achieves comparable outcomes with reduced CPU time, as detailed in Table \ref{tab:Time_Compare}.
	
	\begin{figure}[htb!]
		\centering
		\begin{subfigure}[b]{0.4\textwidth}
			\centering
			\includegraphics[width=1.0\textwidth]{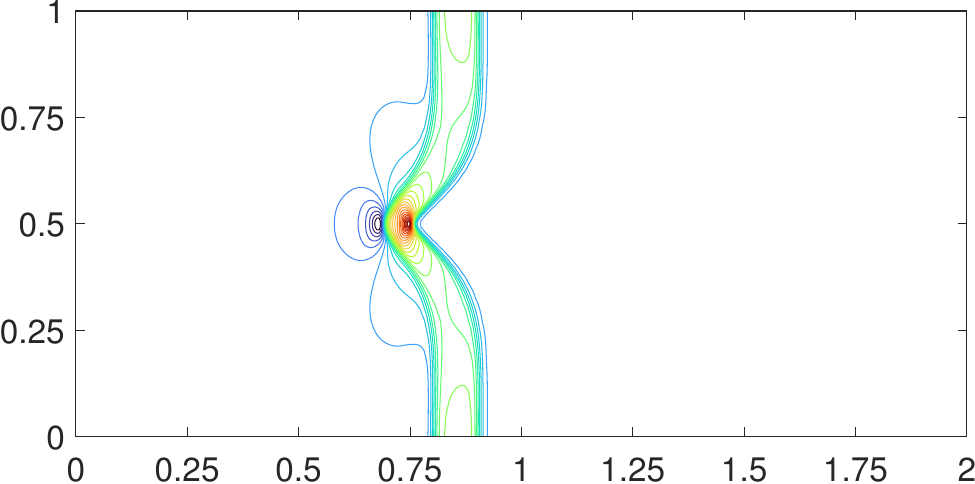}
			\caption{$t=0.24$}
		\end{subfigure}
		\begin{subfigure}[b]{0.4\textwidth}
			\centering
			\includegraphics[width=1.0\textwidth]{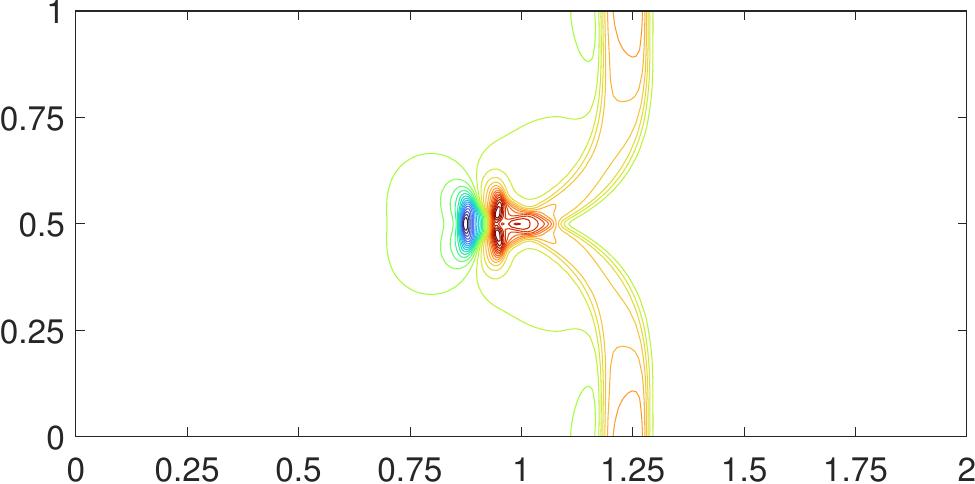}
			\caption{$t=0.36$}
		\end{subfigure}
		\begin{subfigure}[b]{0.4\textwidth}
			\centering
			\includegraphics[width=1.0\textwidth]{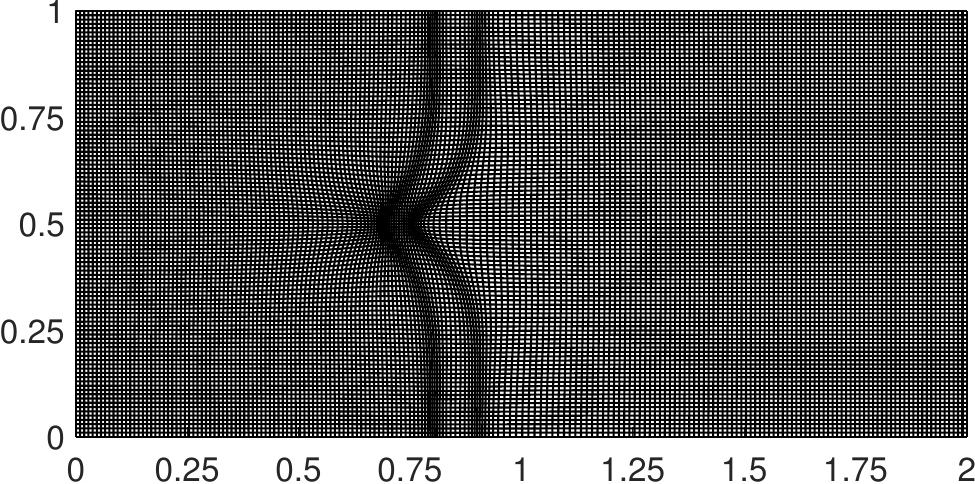}
			\caption{adaptive mesh}
		\end{subfigure}
		\begin{subfigure}[b]{0.4\textwidth}
			\centering
			\includegraphics[width=1.0\textwidth]{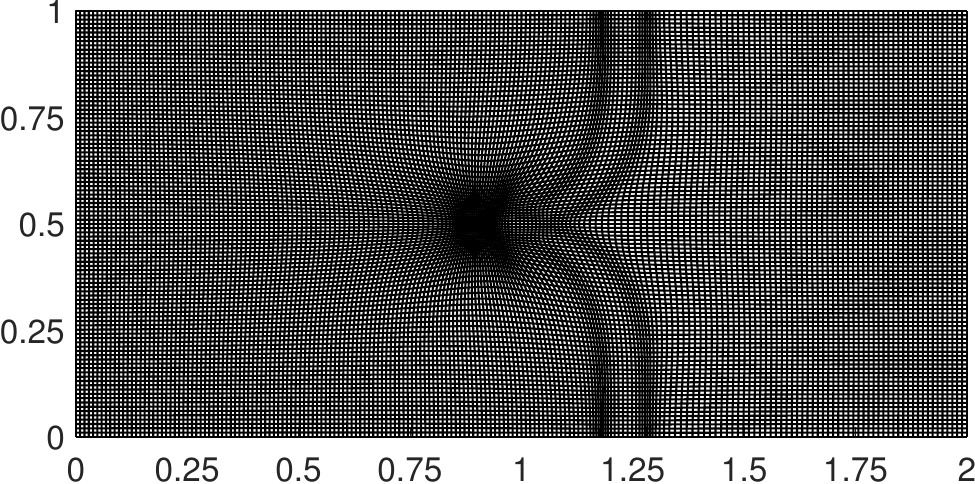}
			\caption{adaptive mesh}
		\end{subfigure}
		
		\begin{subfigure}[b]{0.4\textwidth}
			\centering
			\includegraphics[width=1.0\textwidth]{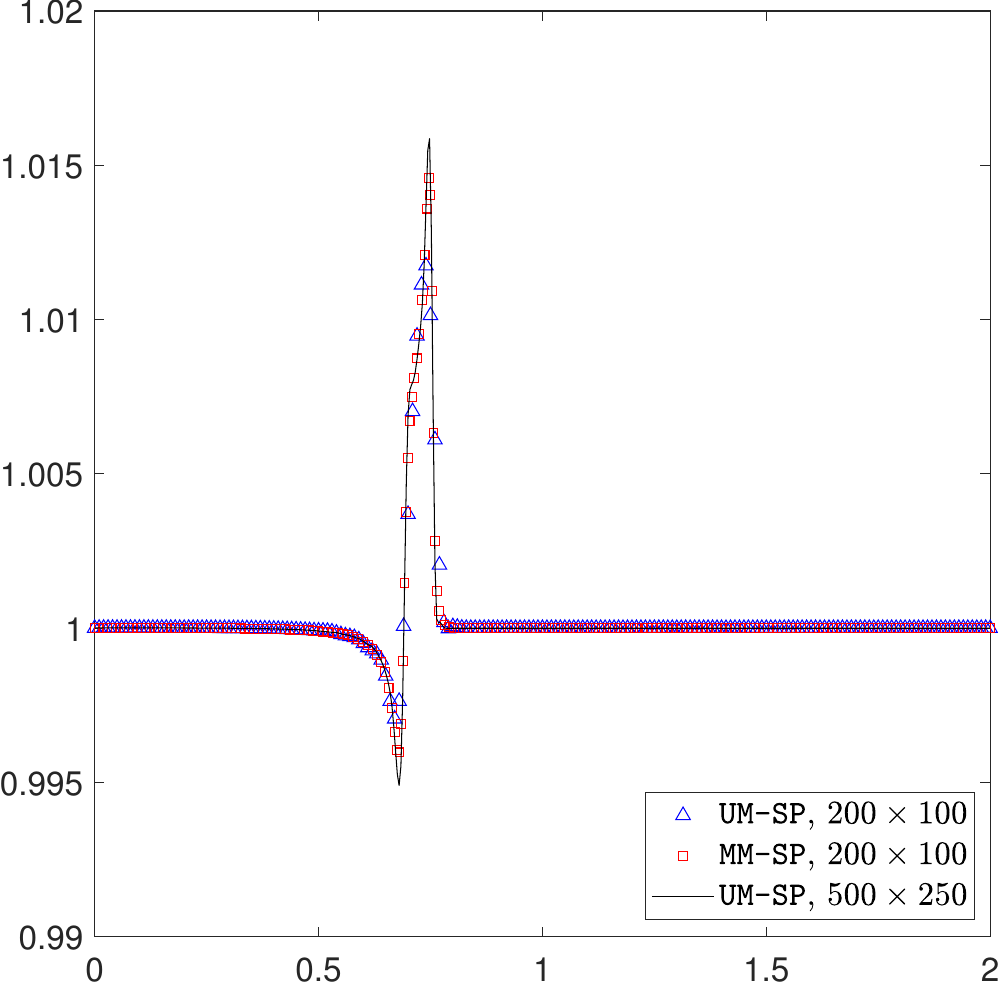}
			\caption{cut line along $x_2=0.5$}
		\end{subfigure}
		\begin{subfigure}[b]{0.4\textwidth}
			\centering
			\includegraphics[width=1.0\textwidth]{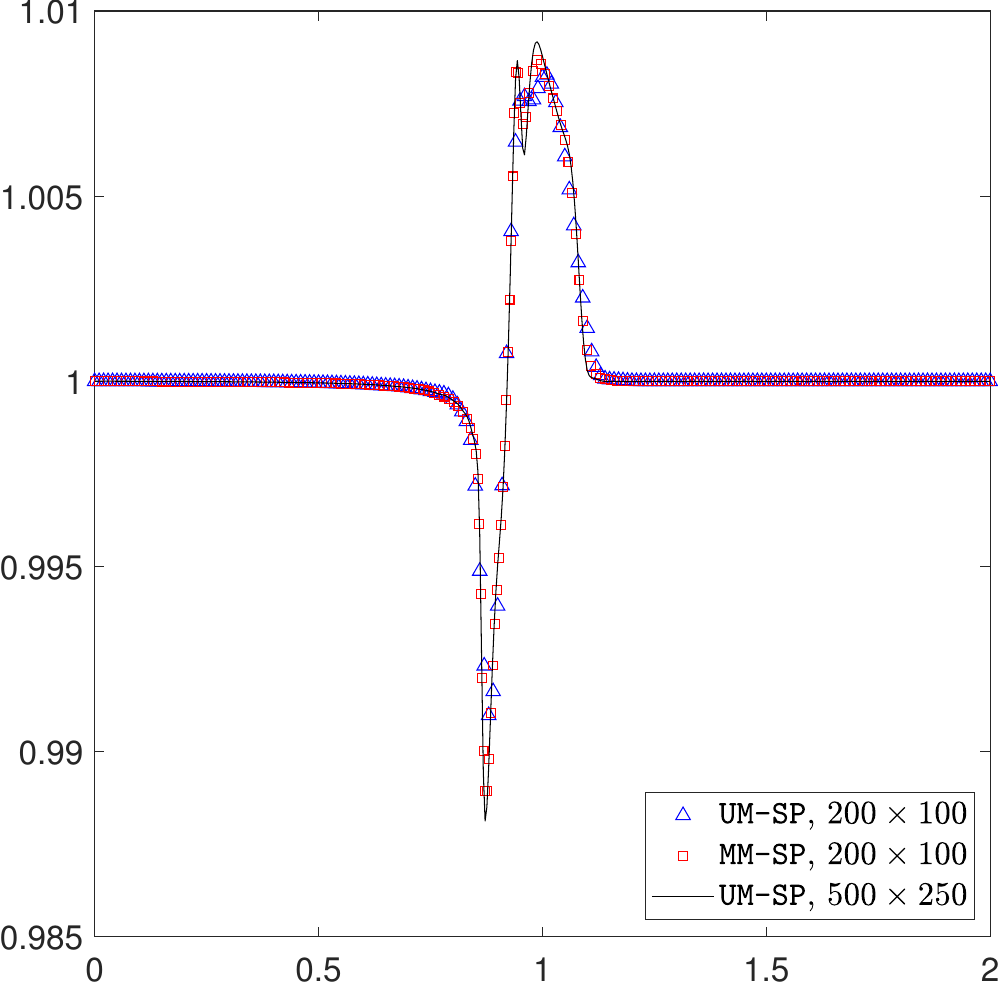}
			\caption{cut line along $x_2=0.5$}
		\end{subfigure}
		\caption{Example \ref{ex:2D_Pertubation}. The numerical results obtained by the \texttt{UM-SP} and \texttt{MM-SP}  scheme. Top: $30$ equally spaced contour lines of $h+b$, middle: $200\times100$ adaptive mesh, bottom: cut line along $x_2 = 0.5$.
		}
		\label{fig:2D_Perturbation_Test}
	\end{figure}
	
	\begin{example}[Water drop problem]\label{ex:2D_Water_Drop}\rm
		This test demonstrates the effectiveness of our moving mesh schemes in capturing complex wave structures by solving the water drop problem over a flat bottom. The physical domain is $[0,2]\times [0,2]$ with reflective boundary conditions. The Gaussian-shaped peak initial condition is considered:
		\begin{equation*}
			h = 1+0.1 e^{-1000(x_1-1)^2+(x_2-1)^2}, 
		\end{equation*}
		accompanied by zero velocities and a gravitational constant $g = 9.812$. The output times are set as $t = 0.35$ and $t = 0.6$. The monitor function is the same as in Example \ref{ex:2D_WB_Test}, but with $\theta = 50$. 
	\end{example}
	Figure \ref{fig:2D_Water_Drop} displays the adaptive meshes, water surface level contours, and a cut line across $x_2 = 1$ using our schemes. The adaptive moving mesh scheme accurately captures complex structures and effectively concentrates mesh points in those areas to enhance resolution. Additionally, our schemes maintain excellent symmetry in this numerical experiment. Table \ref{tab:Time_Compare} illustrates that the CPU time required for the \texttt{MM-SP} scheme with a $150\times 150$ mesh is approximately 50\% less than that required for the \texttt{UM-SP} scheme with a $450\times 450$ mesh.
	\begin{figure}[hbt!]
		\centering
		\begin{subfigure}[b]{0.3\textwidth}
			\centering
			\includegraphics[width=1.0\linewidth]{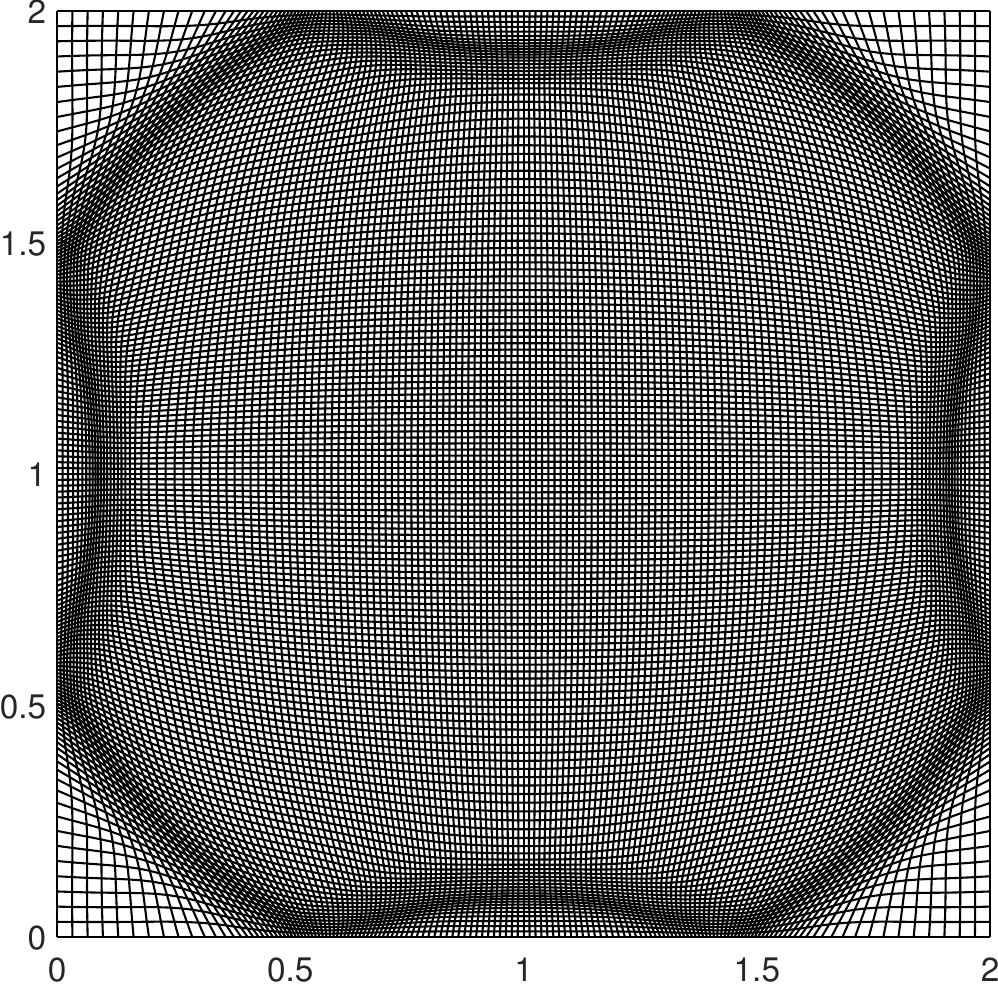}
			\caption{$150\times150$ adaptive mesh, \texttt{MM-SP}}
		\end{subfigure}
		\begin{subfigure}[b]{0.3\textwidth}
			\centering
			\includegraphics[width=1.0\linewidth]{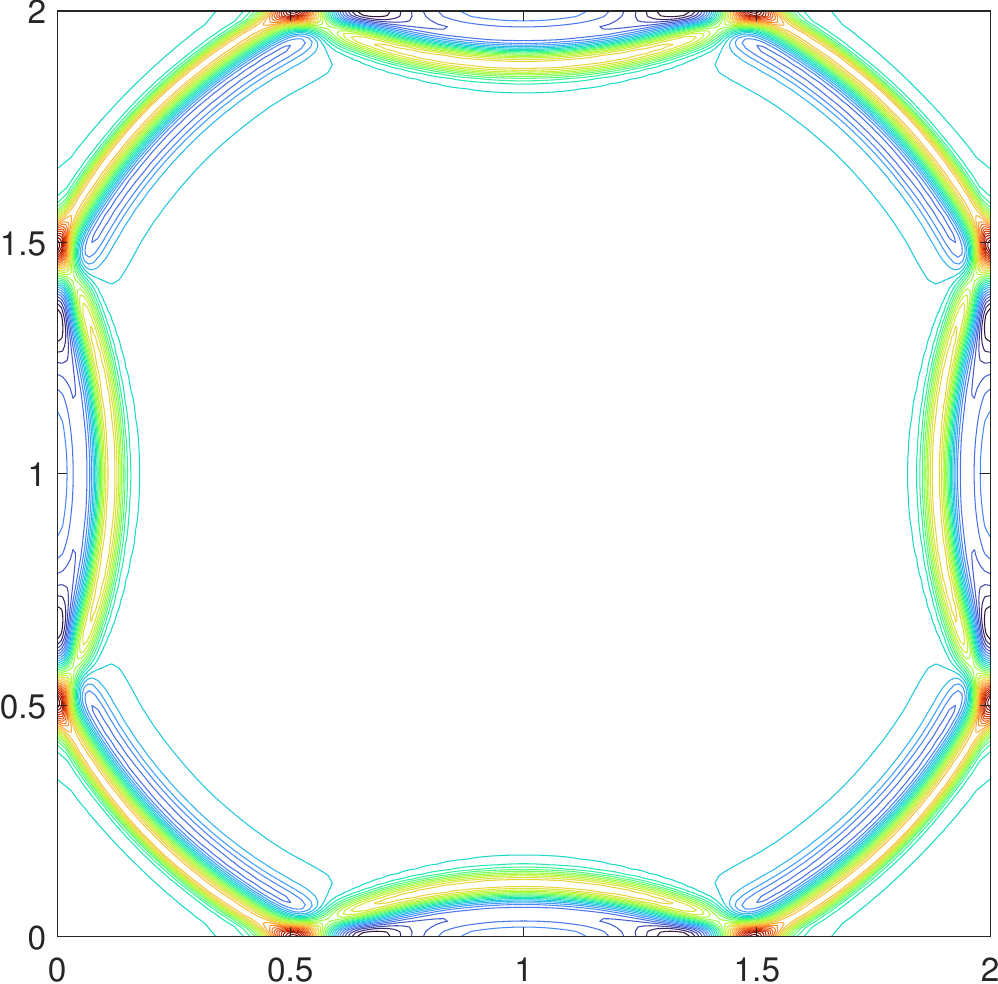}
			\caption{$h+b$, \texttt{MM-SP}}
		\end{subfigure}
		\quad
		\begin{subfigure}[b]{0.3\textwidth}
			\centering
			\includegraphics[width=1.0\linewidth]{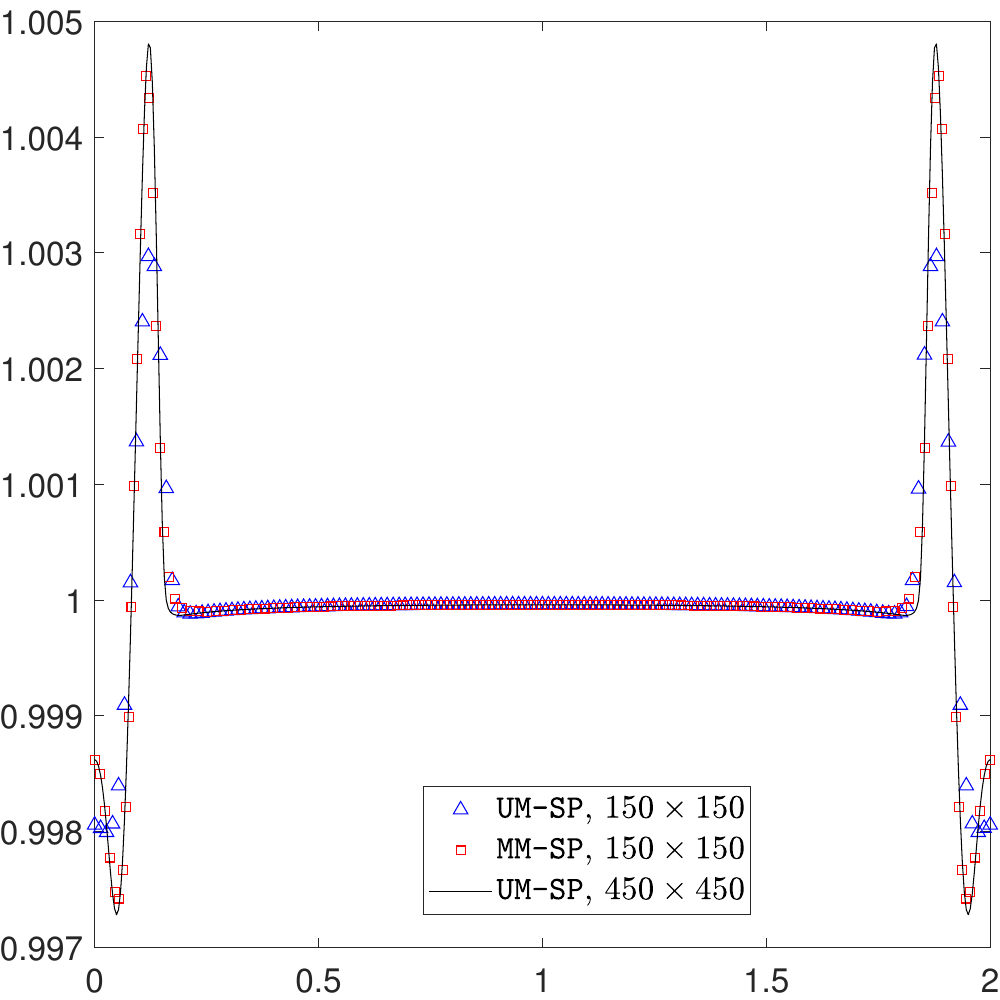}
			\caption{cut lines along $x_2=1$}
		\end{subfigure}
		\begin{subfigure}[b]{0.3\textwidth}
			\centering
			\includegraphics[width=1.0\linewidth]{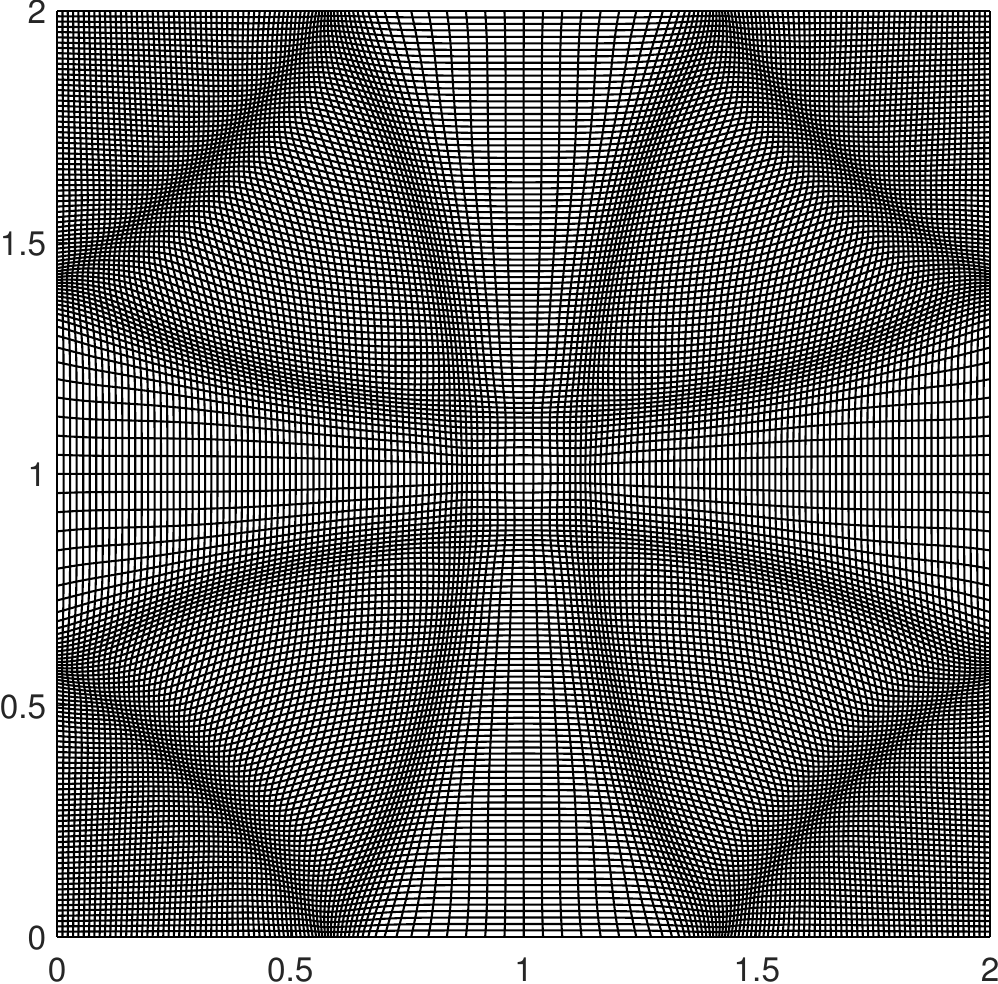}
			\caption{$150\times150$ adaptive mesh, \texttt{MM-SP}}
		\end{subfigure}
		\begin{subfigure}[b]{0.3\textwidth}
			\centering
			\includegraphics[width=1.0\linewidth]{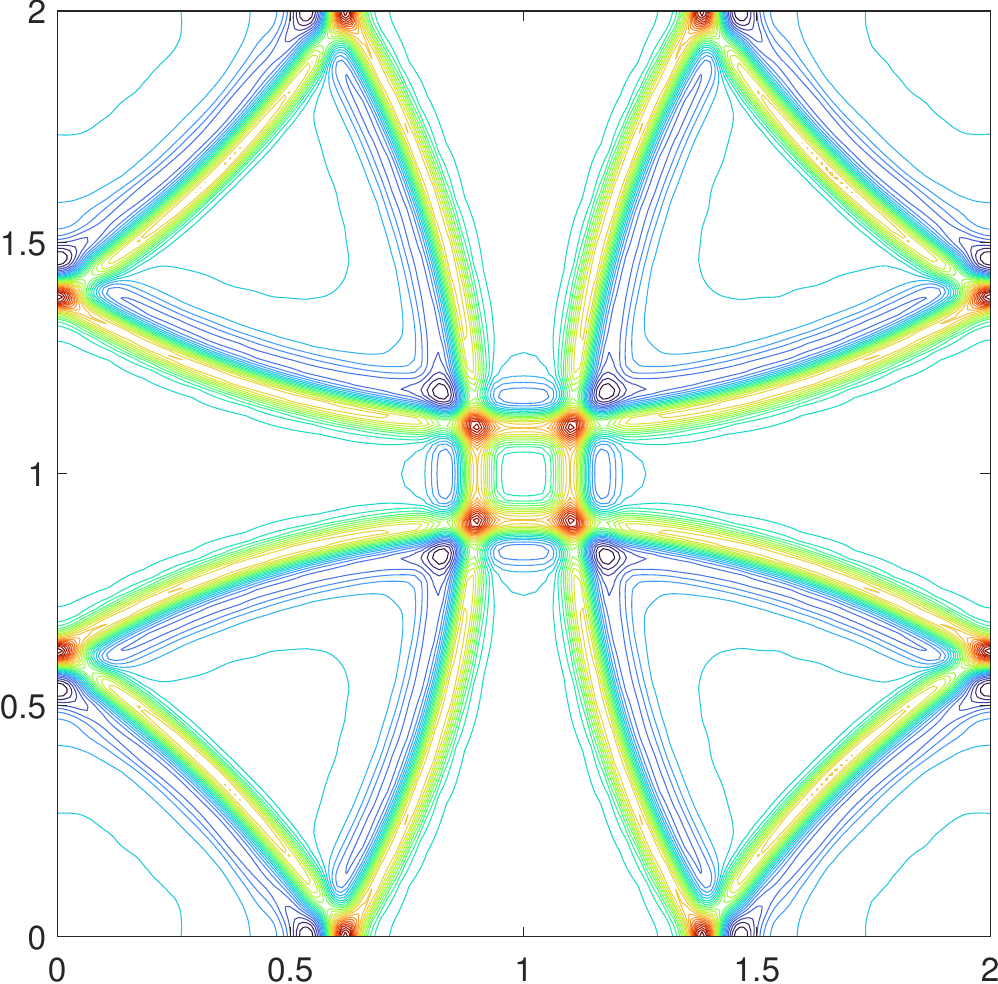}
			\caption{$h+b$, \texttt{MM-SP}}
		\end{subfigure}
		\quad
		\begin{subfigure}[b]{0.3\textwidth}
			\centering
			\includegraphics[width=1.0\linewidth]{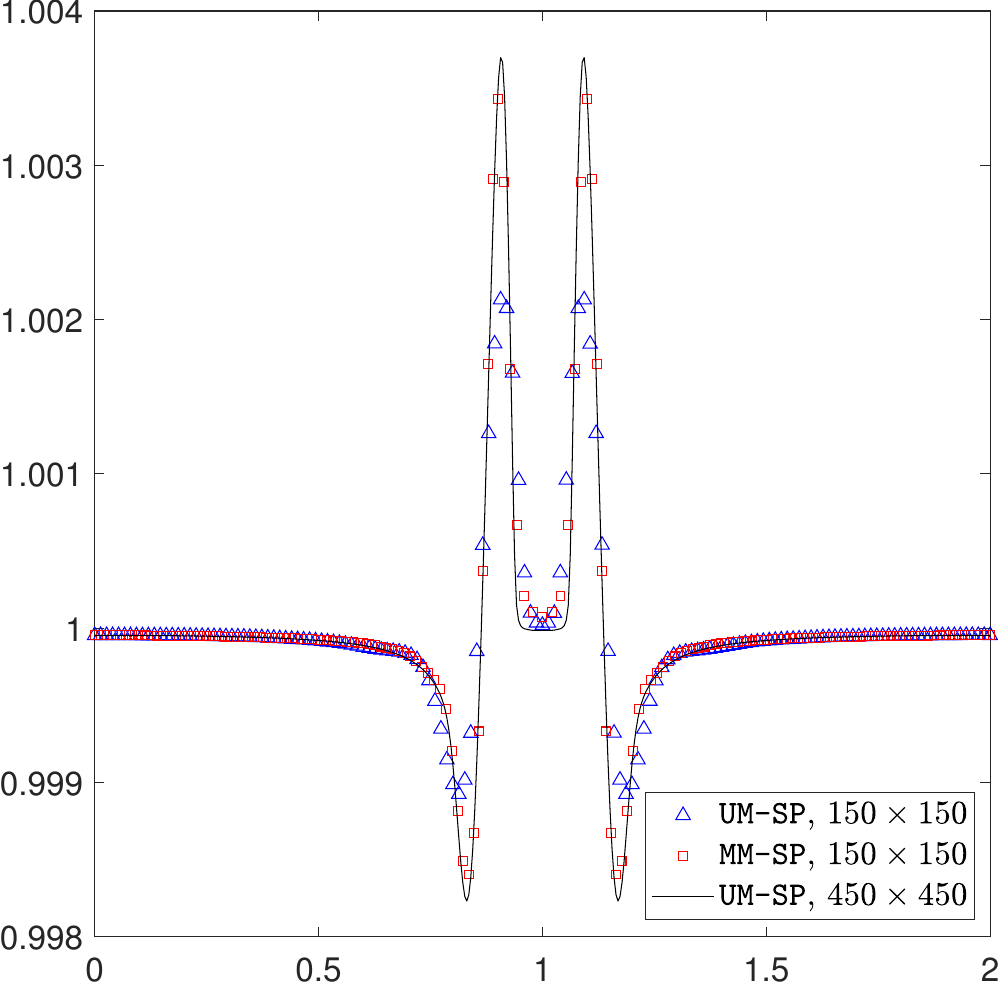}
			\caption{cut lines along $x_2=1$}
		\end{subfigure}
		
		\caption{Example \ref{ex:2D_Water_Drop}. The results obtained by using the \texttt{UM-SP} and \texttt{MM-SP} schemes. Top: $t=0.35$, bottom: $t=0.6$. }\label{fig:2D_Water_Drop}
	\end{figure}

	\begin{example}[Circular dam break on a non-flat river bed]\label{ex:2D_C_DamBreak}\rm
		The circular dam break problem \cite{Capilla2013New,Castro2009High} for the 2D SWEs is tested here to further illustrate the efficiency of our proposed schemes in the physical domain $[0,2]\times[0,2]$ with  outflow boundary conditions. 
		The initial conditions for water depth and velocities are
		\begin{equation*}
			\begin{aligned}
				& h = \begin{cases}
					1.1-b, & \text{if}~\sqrt{(x_1-1.25)^2+(x_2-1)^2} \leqslant 0.1, \\
					0.6-b, & \text {otherwise},
				\end{cases} \\
				& v_1=v_2=0.
			\end{aligned}
		\end{equation*}
		We examine two bottom topographies in our study:
		\begin{equation}\label{eq:2D_C_DamBreak_b}
			b(x_1, x_2) = \begin{cases}\dfrac{1}{8}(\cos (2 \pi(x_1-0.5))+1)(\cos (2 \pi x_2)+1), & \text{if}\quad \sqrt{(x_1-1.5)^2+(x_2-1)^2} \leqslant 0.5, \\ 0, & \text {otherwise},\end{cases}
		\end{equation}
		and 		\begin{equation}\label{eq:2D_C_DamBreak_PP_b}
			b(x_1, x_2) = \begin{cases}\dfrac{3}{20}(\cos (2 \pi(x_1-0.5))+1)(\cos (2 \pi x_2)+1), & \text{if}\quad \sqrt{(x_1-1.5)^2+(x_2-1)^2} \leqslant 0.5, \\ 0, & \text {otherwise}.\end{cases}
		\end{equation}
		The latter configuration is specifically employed to verify the PP properties of the proposed schemes, because it includes a dry region at $x_1 = 0.5$ and $x_2 = 0$. 
		The gravitational acceleration constant $g = 9.812$. The output times for the configurations \eqref{eq:2D_C_DamBreak_b} and \eqref{eq:2D_C_DamBreak_PP_b} are 0.15 and 0.2, repetitively. 
		The monitor function is the same as that in Example \ref{ex:2D_Pertubation}. 
	\end{example}
	
	\begin{figure}[hbt!]
		\centering
		\begin{subfigure}[b]{0.3\textwidth}
			\centering
			\includegraphics[width=1.0\linewidth]{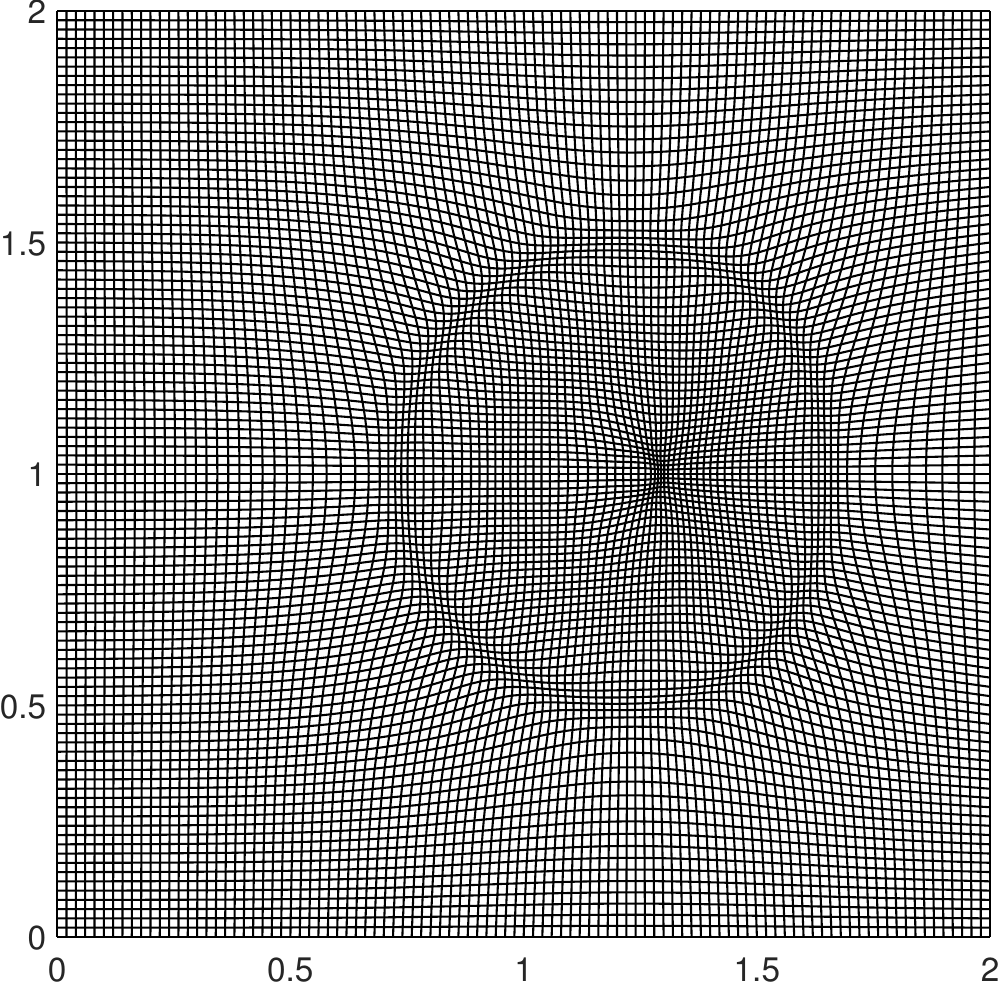}
			\caption{$100\times100$ adaptive mesh, \texttt{MM-SP}}
		\end{subfigure}
		\begin{subfigure}[b]{0.3\textwidth}
			\centering
			\includegraphics[width=1.0\linewidth]{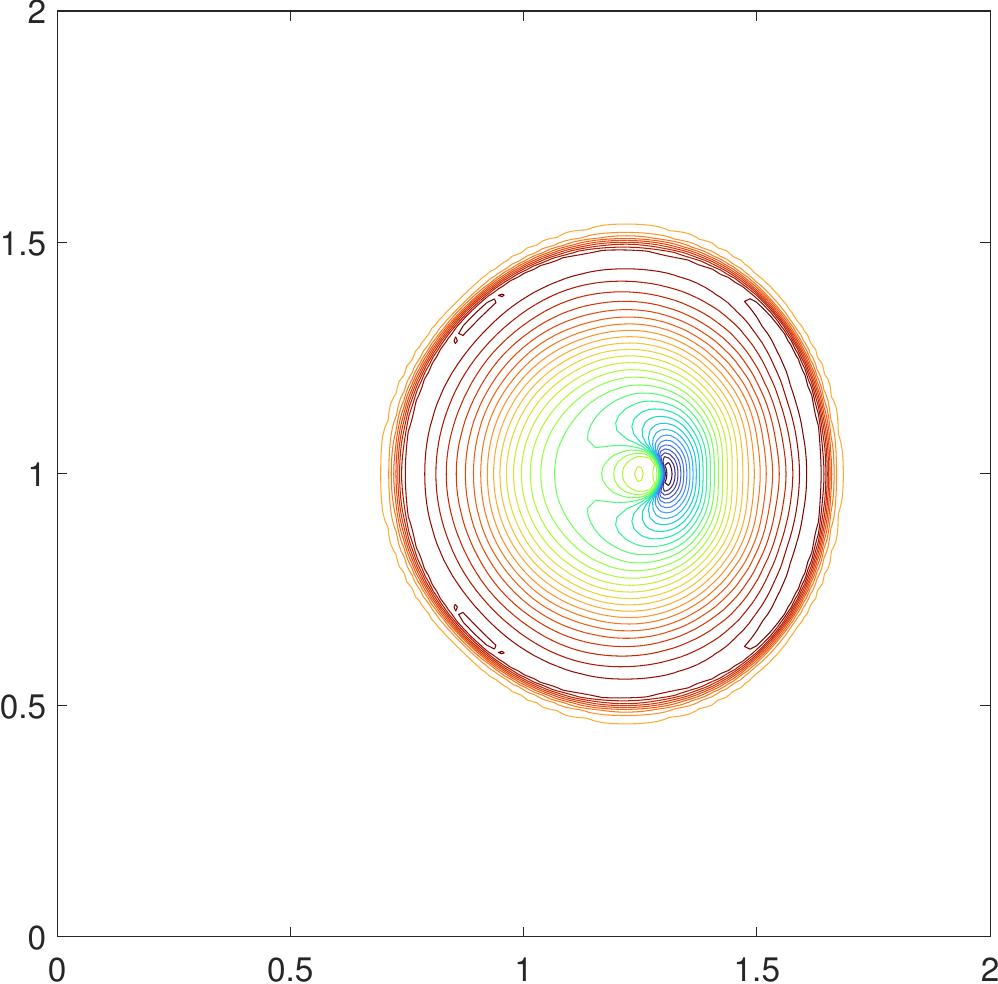}
			\caption{$h+b$, \texttt{MM-SP}}
		\end{subfigure}
		\quad
		\begin{subfigure}[b]{0.3\textwidth}
			\centering
			\includegraphics[width=1.0\linewidth]{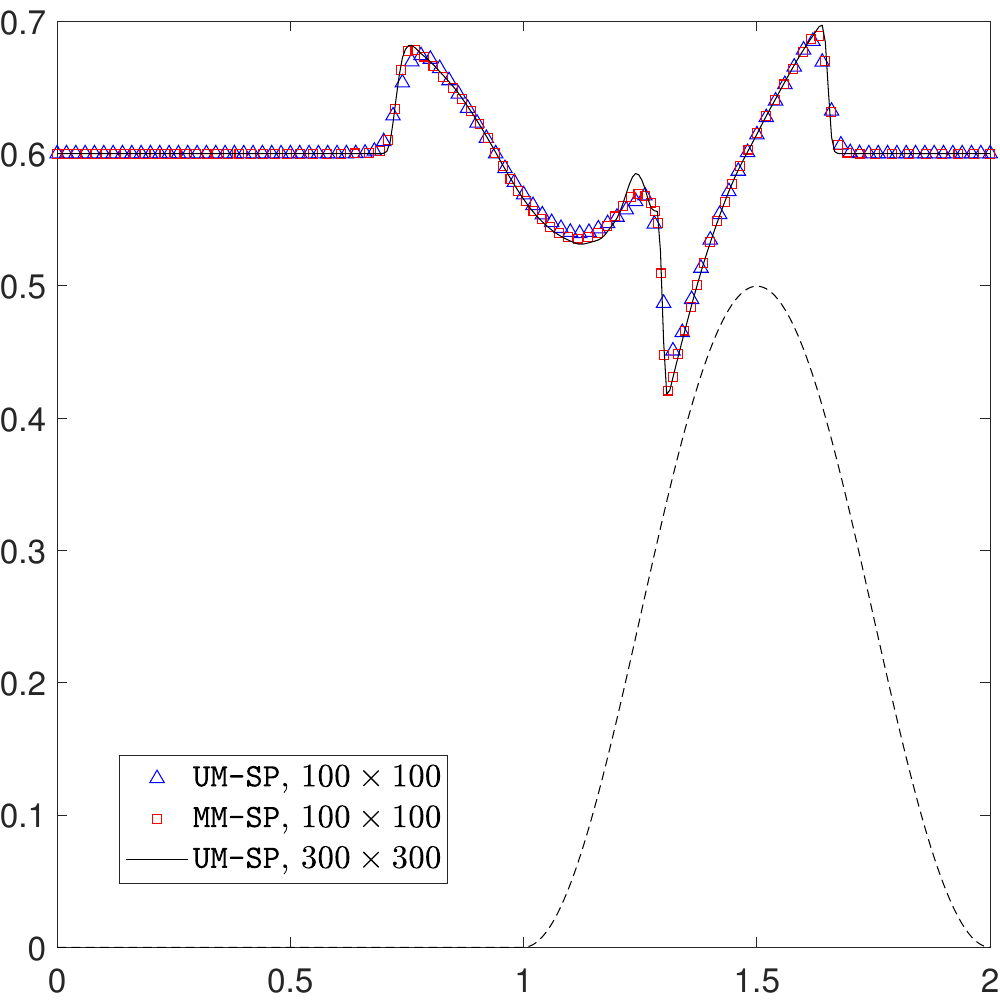}
			\caption{cut lines along $x_2=1$}
		\end{subfigure}
		\begin{subfigure}[b]{0.3\textwidth}
			\centering
			\includegraphics[width=1.0\linewidth]{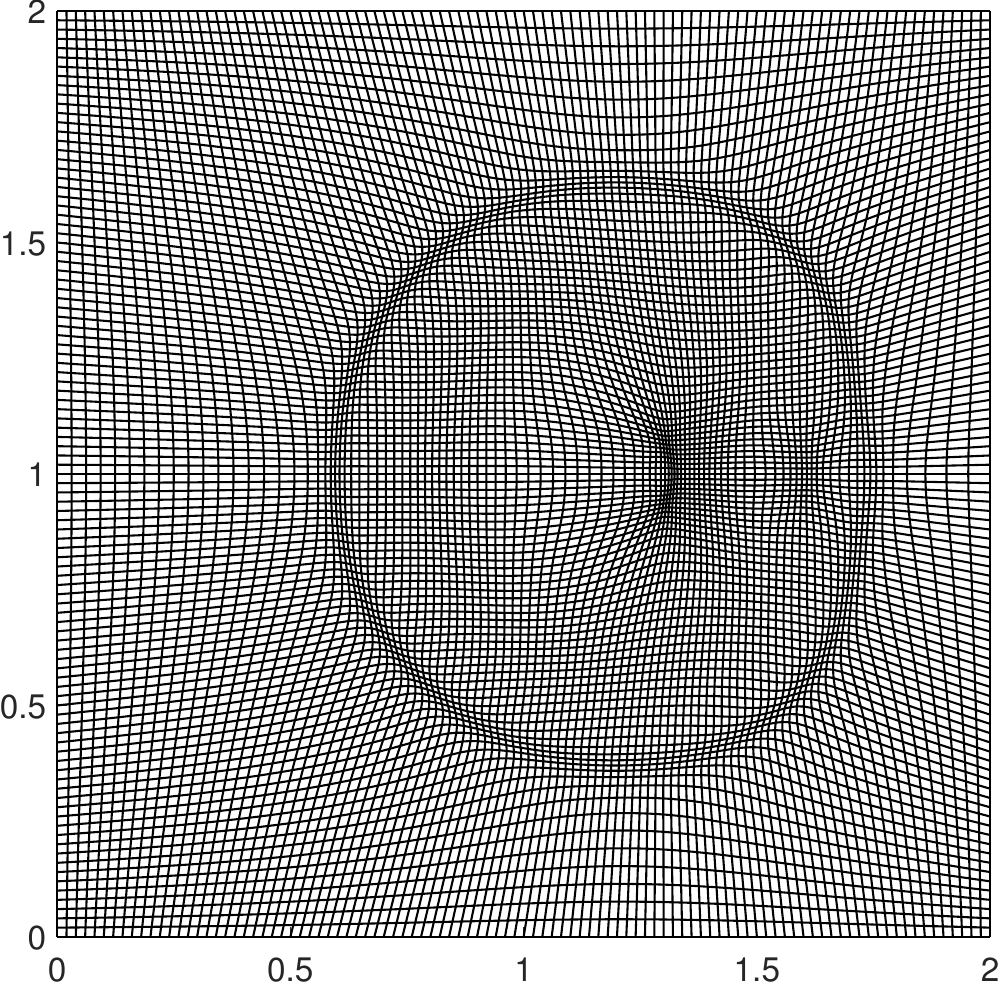}
			\caption{$100\times100$ adaptive mesh, \texttt{MM-SP}}
		\end{subfigure}
		\begin{subfigure}[b]{0.3\textwidth}
			\centering
			\includegraphics[width=1.0\linewidth]{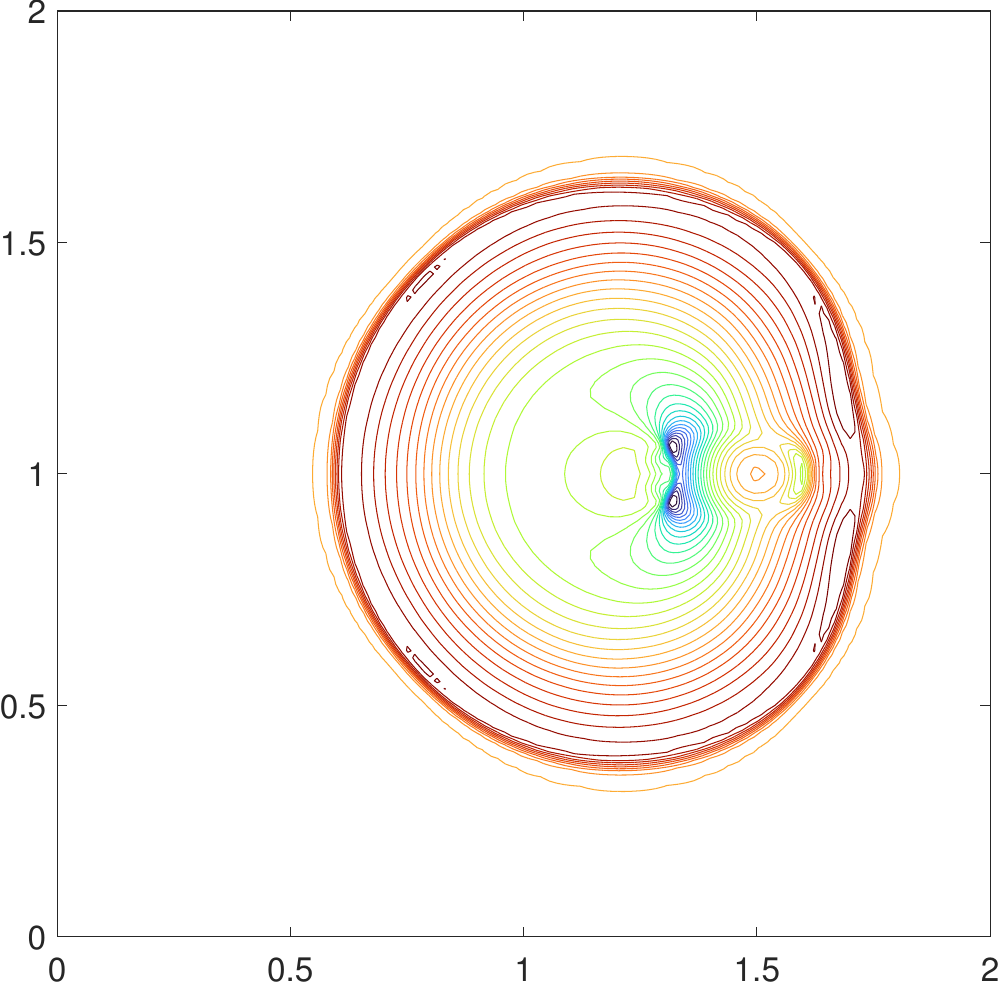}
			\caption{$h+b$, \texttt{MM-SP}}
		\end{subfigure}
		\quad
		\begin{subfigure}[b]{0.3\textwidth}
			\centering
			\includegraphics[width=1.0\linewidth]{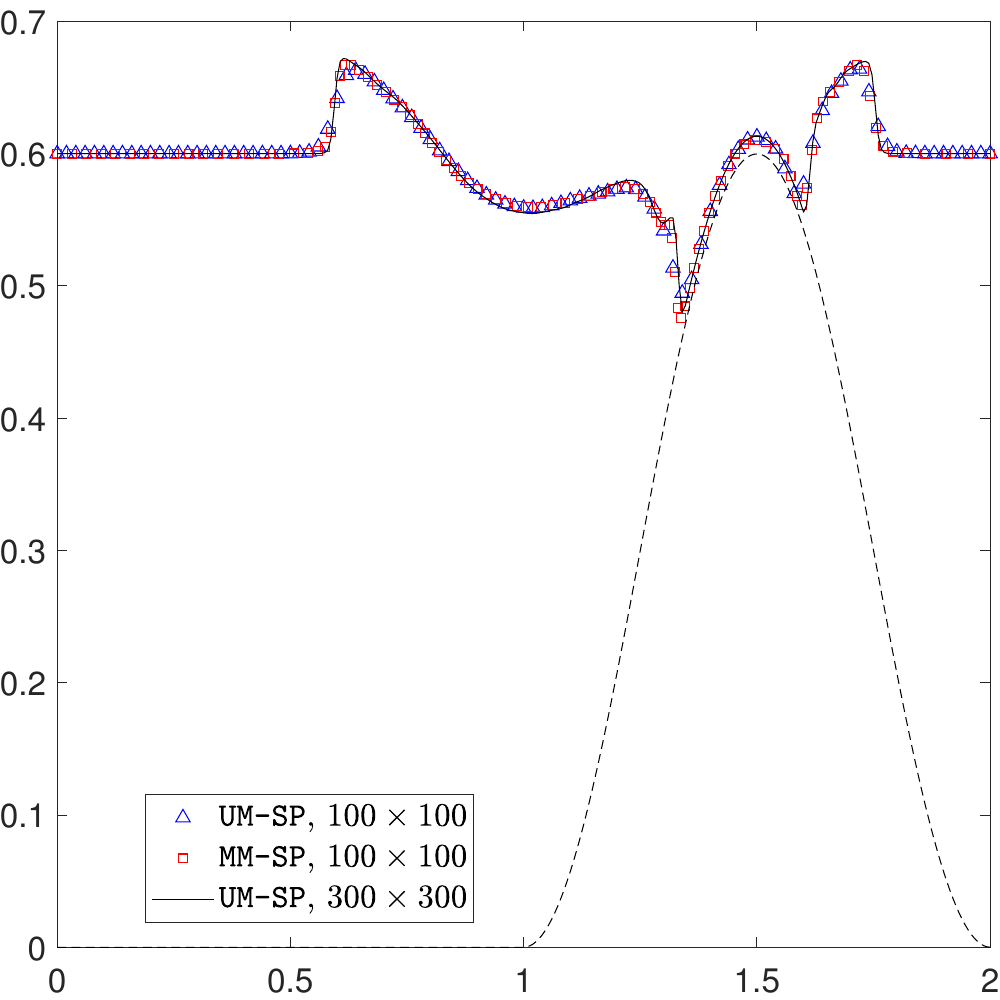}
			\caption{cut lines along $x_2=1$}
		\end{subfigure}
		\caption{Example \ref{ex:2D_C_DamBreak}. The numerical results obtained by using the \texttt{UM-SP} and \texttt{MM-SP} schemes. Top: for bottom topography \eqref{eq:2D_C_DamBreak_b}
			, $t=0.15$, bottom: for bottom topography \eqref{eq:2D_C_DamBreak_PP_b}, $t=0.2$.}\label{fig:2D_Circle_Dam}
	\end{figure}
	
	Figure \ref{fig:2D_Circle_Dam} illustrates the $100\times100$ adaptive mesh, 30 equally spaced contours of the water surface level $h+b$, and a comparison of the cut lines along $x_2 = 1$ obtained by our proposed schemes. The adaptive mesh clearly demonstrates that mesh points  concentrate around localized structures. In comparison, the \texttt{MM-SP} scheme produces sharper results near $x_1 \approx 1.26$ than the \texttt{UM-SP} scheme with the same number of mesh points. 
	Without the PP limiter, the numerical simulation results in negative water depths and quickly fails for the bottom topography in \eqref{eq:2D_C_DamBreak_PP_b}. The CPU time is detailed in Table \ref{tab:Time_Compare}, highlighting the superior performance of our \texttt{MM-SP} scheme. 
	\begin{table}
		\small 
		\caption{CPU times for Examples \ref{ex:2D_Pertubation}--\ref{ex:2D_C_DamBreak}. The measurements are taken on a laptop with an Intel® Core™ i7-8750H
			CPU @2.20GHz and 24 GB of memory, utilizing 8 OpenMP threads. The code is implemented in C++. Here, ``s" denotes seconds, and $N_1 \times N_2$ represents the number of mesh cells.}
		\centering
		\begin{tabular}{cccc} 
			\toprule
			\toprule   Examples   & \texttt{UM-PP}
			&\texttt{UM-PP}&\texttt{MM-PP} \\ 
			\midrule 
			Ex. \ref{ex:2D_Pertubation}~$~(t=0{\rightarrow}0.24)$  & $4.90\times 10^1$ s ($200\times 100$) 
			& $6.89 \times 10^2$ s ($500\times 250$)
			&  $3.56 \times 10^2$ s ($200\times 100$)\\ 
			Ex. \ref{ex:2D_Pertubation}~$~(t=0{\rightarrow}0.36) $ & $7.30\times 10^1$ s ($200\times 100$) 
			& $1.26 \times 10^3$ s ($500\times 250$) & $6.44 \times 10^2$ s ($200\times 100$)
			\\ 
			Ex. \ref{ex:2D_Water_Drop}$~(t=0{\rightarrow}0.35)$  & $3.50\times 10^1$ s ($150\times 150$) 
			& $8.89 \times 10^2$ s ($450\times 450$)
			&  $4.48 \times 10^2$ s ($150\times 150$)
			\\ 
			Ex. \ref{ex:2D_Water_Drop}$~(t=0{\rightarrow}0.60)$ & $5.90\times 10^1$ s ($150\times 150$) 
			& $1.52 \times 10^3$ s ($450\times 450$) & $6.65 \times 10^2$ s ($150\times 150$)
			\\ 
			Ex. \ref{ex:2D_C_DamBreak}$~(t=0{\rightarrow}0.15)$ & $0.60 \times 10^1$ s ($100\times 100$) 
			& $1.30 \times 10^2$ s ($300\times 300$)
			& $6.60 \times 10^1$ s ($100\times 100$)
			\\ Ex. \ref{ex:2D_C_DamBreak}$~(t=0{\rightarrow}0.20)$ & $6.20 \times 10^1$ s ($100\times 100$) 
			& $1.36 \times 10^3$ s ($300\times 300$)
			& $7.44 \times 10^2$ s ($100\times 100$)\\
			\bottomrule
			\bottomrule
		\end{tabular}
		\label{tab:Time_Compare}
	\end{table}

	\section{Conclusion}\label{Sec:Conclusion}
	This paper proposed high-order accurate, well-balanced (WB), and positivity-preserving (PP) finite volume schemes for shallow water equations (SWEs) on adaptive moving structured meshes.
	The WB property in curvilinear coordinates depends not only on the balance between flux gradients and source terms but also on the dynamics of mesh movement.
	To manage these complexities, we decomposed the WB property into two main components: flux source balance and mesh movement balance.
	We achieved flux source balance through the careful decomposition of source terms, the use of numerical fluxes designed via hydrostatic reconstruction, and the suitable discretization of geometric conservation laws (GCLs). These elements combined to ensure the provably WB property in complex curvilinear coordinates. 
	Additionally, we conducted rigorous analyses of the PP property under a sufficient condition enforced by a PP limiter for the proposed moving mesh schemes. Due to the involvement of mesh metrics and movement, these analyses were complicated. The PP property referred to the non-negativity of the water depth and the positivity of the determinant of the coordinate transformation. 
	Mesh adaptation in our method was executed through the iterative solving of Euler--Lagrange equations based on a mesh adaptation functional. The chosen monitor function adeptly concentrated mesh points near significant features, enhancing the efficiency and resolution for discontinuities.
	Extensive numerical examples were provided, demonstrating the high-order accuracy, high efficiency, and robust WB and PP properties of the schemes.
	
	\appendix
	\section{Proof of condition \eqref{eq:2D_condition} }\label{section:Appendix}
	This appendix provides the proof of the condition \eqref{eq:2D_condition}. We first prove 
	\begin{align*}
		\left(
		\widetilde{\bm{F}}_{n_{1}}
		\right)
		_{i+\frac{1}{2},j}^{L,\mu} - \sum\limits_{m=1}^{2}(\overline{{s}}_m)_{i,j}
		\left(
		\left(
		\widetilde{\bm{r}}_m
		\right)_{n_1}
		\right)
		_{i+\frac{1}{2},j}^{\mu}
		=
		\langle
		\left(
		\bm{n}_1
		\right)
		_{i+\frac{1}{2},j}^{\mu}
		,
		\bm{C}
		\rangle
	\end{align*}
	for the numerical fluxes $\left(\widetilde{\bm{F}}_{n_1}\right)_{i+\frac{1}{2},j}^{\mu}$ and $\left(
	\bm{
		\widetilde{\mathcal{R}}
	}_{m1}
	\right)
	_{i+\frac{1}{2},j}$ defined in \eqref{eq:2D_F_1_Flux} and \eqref{eq:2D_R_1_Flux}, respectively. When a steady state is reached, one has
	\begin{align*}
		\left(
		\widetilde{\bm{F}}_{n_{1}}
		\right)
		_{i+\frac{1}{2},j}^{L,\mu} - \sum\limits_{m=1}^{2}(\overline{{s}}_m)_{i,j}
		\left(
		\left(
		\widetilde{\bm{r}}_m
		\right)_{n_1}
		\right)
		_{i+\frac{1}{2},j}^{\mu}
		= 
		&\langle
		\left(
		\bm{n}_1
		\right)
		_{i+\frac{1}{2},j}^{\mu}
		,
		\widetilde{\bm{F}}
		_{i+\frac{1}{2},j}^{L,\mu}
		\rangle - 
		\sum\limits_{m=1}^{2}(\overline{s}_m)_{i,j}\langle
		\left(
		\bm{n}_1
		\right)
		_{i+\frac{1}{2},j}^{\mu}
		,\left(
		{\bm{r}}_m
		\right)_{i+\frac{1}{2},j}^{-,\mu}
		\rangle
		\\
		=&
		\langle
		\left(
		\bm{n}_1
		\right)
		_{i+\frac{1}{2},j}^{\mu}
		,
		\widetilde{\bm{F}}
		_{i+\frac{1}{2},j}^{L,\mu}
		-
		\sum\limits_{m=1}^{2}(\overline{s}_m)_{i,j}
		\left(
		{\bm{r}}_m
		\right)_{i+\frac{1}{2},j}^{-,\mu}
		\rangle.
	\end{align*}
	Note that 
	\begin{align*}
		\widetilde{\bm{F}}
		_{i+\frac{1}{2},j}^{L,\mu}
		-
		\sum\limits_{m=1}^{2}(\overline{s}_m)_{i,j}
		\left(
		{\bm{r}}_m
		\right)_{i+\frac{1}{2},j}^{-,\mu}& = (L\bm{F}_1,L\bm{F}_2)^{\top}
	\end{align*}
	with
	\begin{align*}
		L\bm{F}_1= 
		\left(
		\begin{array}{cc}
			0 \\
			\frac{g}{2}\left(h_{i+\frac{1}{2},j}^{-,\mu}\right)^2+g(h+b)_{i,j}b_{i+\frac{1}{2},j}^{-,\mu} - \frac{1}{2}g (b^2)_{i+\frac{1}{2},j}^{-,\mu}\\
			0
		\end{array}
		\right),
		\\
		L\bm{F}_2= 
		\left(
		\begin{array}{cc}
			0 \\
			0\\
			\frac{g}{2}
			\left(h_{i+\frac{1}{2},j}^{-,\mu}\right)^2
			+g(h+b)_{i,j} 
			b_{i+\frac{1}{2},j}^{-,\mu}-
			\frac{1}{2}g(b^2)_{i+\frac{1}{2},j}^{-,\mu}
		\end{array}
		\right).
	\end{align*}
	Using the condition $(h+b)_{i,j} = (h+b)_{i\pm\frac{1}{2},j}^{\mp,\mu} = C$ yields 
	\begin{align*}
		L\bm{F}_1= 
		\left(
		\begin{array}{cc}
			0 \\
			\frac{g}{2}
			C^2\\
			0
		\end{array}
		\right), \quad 
		L\bm{F}_2= 
		\left(
		\begin{array}{cc}
			0 \\
			0\\
			\frac{g}{2}
			C^2
		\end{array}
		\right).
	\end{align*}
	Therefore, we obtain 
	\begin{align*}
		\left(
		\widetilde{\bm{F}}_{n_{1}}
		\right)
		_{i+\frac{1}{2},j}^{L,\mu} - \sum\limits_{m=1}^{2}(\overline{{s}}_m)_{i,j}
		\left(
		\left(
		\widetilde{\bm{r}}_m
		\right)_{n_1}
		\right)
		_{i+\frac{1}{2},j}^{\mu}
		=
		\langle
		\left(
		\bm{n}_1
		\right)
		_{i+\frac{1}{2},j}^{\mu}
		,
		\bm{C}
		\rangle.
	\end{align*}
	
	The numerical fluxes $\left(\widetilde{\bm{F}}_{n_1}\right)_{i-\frac{1}{2},j}^{R,\mu}$ and $\left(
	\bm{
		\widetilde{\mathcal{R}}
	}_{m1}
	\right)
	_{i-\frac{1}{2},j}  $ are defined as 
	\begin{equation*}
		\begin{aligned}
			&
			\left(\widetilde{\bm{F}}_{n_1}\right)_{i-\frac{1}{2},j}^{R,\mu}
			= \langle
			\left(
			\bm{n}_1
			\right)
			_{i-\frac{1}{2},j}^{\mu}
			,
			\widetilde{\bm{F}}
			_{i-\frac{1}{2},j}^{R,\mu}
			\rangle
			-
			\frac{\beta}{2}
			(
			\bm{U}_{i-\frac{1}{2},j}^{+,\mu,*}
			-
			\bm{U}_{i-\frac{1}{2},j}^{-,\mu,*}
			),
			\\
			&
			\left(
			\bm{
				\widetilde{\mathcal{R}}
			}_{m1}
			\right)
			_{i-\frac{1}{2},j} 
			=   \sum_{\mu=1}^{Q}
			\omega_{\mu}
			\left[
			\left(
			L_{1}
			\right)
			_{i-\frac{1}{2},j}^{\mu}
			\left(
			\left(
			\widetilde{\bm{r}}_m
			\right)_{n_1}
			\right)
			_{i-\frac{1}{2},j}^{\mu}
			\right],\\
			&
			\left(
			\left(
			\widetilde{\bm{r}}_m
			\right)_{n_1}
			\right)
			_{i-\frac{1}{2},j}^{\mu} = 
			\langle
			\left(
			\bm{n}_1
			\right)
			_{i-\frac{1}{2},j}^{\mu}
			,\left(
			{\bm{r}}_m
			\right)_{i-\frac{1}{2},j}^{+,\mu}
			\rangle
		\end{aligned}
	\end{equation*}
	with 
	\begin{equation*}
		\left(\widetilde{\bm{F}}\right)
		_{i-\frac{1}{2},j}^{R,\mu}
		= \left(
		\left(
		\bm{\widetilde{F}}_1
		\right)
		_{i-\frac{1}{2},j}^{R,\mu}
		,
		\left(
		\bm{\widetilde{F}}_2
		\right)
		_{i-\frac{1}{2},j}^{R,\mu}
		\right)^{\top},
	\end{equation*}
	where
	\begin{equation*}
		\left(
		\widetilde{\bm{F}}_1
		\right)
		_{i-\frac{1}{2},j}^{R,\mu}    
		= \dfrac{1}{2}
		\left(
		{\bm{F}}_1
		(
		\bm{U}_{i-\frac{1}{2},j}^{-,\mu,*}
		)
		+
		{\bm{F}}_1
		(
		\bm{U}_{i-\frac{1}{2},j}^{+,\mu,*}
		)
		\right)
		+ \left(\begin{array}{cc}
			0  \\
			\frac{g}{2}\left(h_{i-\frac{1}{2},j}^{+,\mu}\right)^2-\frac{g}{2}
			\left(h_{i-\frac{1}{2},j}
			^{+,\mu,*}
			\right)^2 \\
			0
		\end{array}
		\right),
	\end{equation*}
	\begin{equation*}
		\left(
		\widetilde{\bm{F}}_2
		\right)
		_{i-\frac{1}{2},j}^{R,\mu}    
		= \dfrac{1}{2}
		\left(
		{\bm{F}}_2
		(
		\bm{U}_{i-\frac{1}{2},j}^{-,\mu,*}
		)
		+
		{\bm{F}}_2
		(
		\bm{U}_{i-\frac{1}{2},j}^{+,\mu,*}
		)
		\right)
		+ \left(\begin{array}{cc}
			0  \\
			0  \\
			\frac{g}{2}\left(h_{i-\frac{1}{2},j}^{+,\mu}\right)^2-\frac{g}{2}
			\left(h_{i-\frac{1}{2},j}
			^{+,\mu,*}
			\right)^2  
		\end{array}
		\right).
	\end{equation*}
	Similarly, 
	the numerical fluxes $\left(\widetilde{\bm{F}}_{n_2}\right)_{i,j+\frac{1}{2}}^{L,\mu}$, 
	$\left(\widetilde{\bm{F}}_{n_2}\right)_{i,j-\frac{1}{2}}^{R,\mu}$ and $	\left(
	\bm{
		\widetilde{\mathcal{R}}
	}_{m1}
	\right)
	_{i,j\pm\frac{1}{2}} $ are defined as 
	\begin{equation*}
		\begin{aligned}
			&
			\left(\widetilde{\bm{F}}_{n_2}\right)_{i,j+\frac{1}{2}}^{L,\mu}
			= \langle
			\left(
			\bm{n}_2
			\right)
			_{i,j+\frac{1}{2}}^{\mu}
			,
			\widetilde{\bm{F}}
			_{i,j+\frac{1}{2}}^{L,\mu}
			\rangle
			-
			\frac{\beta}{2}
			(
			\bm{U}_{i,j+\frac{1}{2}}^{+,\mu,*}
			-
			\bm{U}_{i,j+\frac{1}{2}}^{-,\mu,*}
			),
			\\
			&
			\left(\widetilde{\bm{F}}_{n_2}\right)_{i,j-\frac{1}{2}}^{R,\mu}
			= \langle
			\left(
			\bm{n}_2
			\right)
			_{i,j-\frac{1}{2}}^{\mu}
			,
			\widetilde{\bm{F}}
			_{i,j-\frac{1}{2}}^{R,\mu}
			\rangle
			-
			\frac{\beta}{2}
			(
			\bm{U}_{i,j-\frac{1}{2}}^{+,\mu,*}
			-
			\bm{U}_{i,j-\frac{1}{2}}^{-,\mu,*}
			),
			\\
			&
			\left(
			\bm{
				\widetilde{\mathcal{R}}
			}_{m2}
			\right)
			_{i,j\pm\frac{1}{2}} 
			=   \sum\limits_{\mu=1}^{Q}
			\omega_{\mu}
			\left[
			\left(
			L_{2}
			\right)
			_{i,j\pm\frac{1}{2}}^{\mu}
			\left(
			\left(
			\widetilde{\bm{r}}_m
			\right)_{n_2}
			\right)
			_{i,j\pm\frac{1}{2}}^{\mu}
			\right],\\
			&
			\left(
			\left(
			\widetilde{\bm{r}}_m
			\right)_{n_2}
			\right)
			_{i,j\pm\frac{1}{2}}^{\mu} = 
			\langle
			\left(
			\bm{n}_2
			\right)
			_{i,j\pm\frac{1}{2}}^{\mu}
			,\left(
			{\bm{r}}_m
			\right)_{i,j\pm\frac{1}{2}}^{+,\mu}
			\rangle
		\end{aligned}
	\end{equation*}
	with 
	\begin{align*}
		\left(\widetilde{\bm{F}}\right)
		_{i,j+\frac{1}{2}}^{L,\mu}
		= \left(
		\left(
		\bm{\widetilde{F}}_1
		\right)
		_{i,j+\frac{1}{2}}^{L,\mu}
		,
		\left(
		\bm{\widetilde{F}}_2
		\right)
		_{i,j+\frac{1}{2}}^{L,\mu}
		\right)^{\top}, \qquad 
		\left(\widetilde{\bm{F}}\right)
		_{i,j-\frac{1}{2}}^{R,\mu}
		= \left(
		\left(
		\bm{\widetilde{F}}_1
		\right)
		_{i,j-\frac{1}{2}}^{R,\mu}
		,
		\left(
		\bm{\widetilde{F}}_2
		\right)
		_{i,j-\frac{1}{2}}^{R,\mu}
		\right)^{\top},
	\end{align*}
	where
	\begin{align*}
		\left(
		\widetilde{\bm{F}}_1
		\right)
		_{i,j+\frac{1}{2}}^{L,\mu}    
		= \dfrac{1}{2}
		\left(
		{\bm{F}}_1
		(
		\bm{U}_{i,j+\frac{1}{2}}^{-,\mu,*}
		)
		+
		{\bm{F}}_1
		(
		\bm{U}_{i,j+\frac{1}{2}}^{+,\mu,*}
		)
		\right)
		+ \left(\begin{array}{cc}
			0  \\
			\frac{g}{2}\left(h_{i,j+\frac{1}{2}}^{-,\mu}\right)^2-\frac{g}{2}
			\left(h_{i,j+\frac{1}{2}}
			^{-,\mu,*}
			\right)^2 \\
			0
		\end{array}
		\right),
		\\
		\left(
		\widetilde{\bm{F}}_2
		\right)
		_{i,j+\frac{1}{2}}^{L,\mu}    
		= \dfrac{1}{2}
		\left(
		{\bm{F}}_2
		(
		\bm{U}_{i,j+\frac{1}{2}}^{-,\mu,*}
		)
		+
		{\bm{F}}_2
		(
		\bm{U}_{i,j+\frac{1}{2}}^{+,\mu,*}
		)
		\right)
		+ \left(\begin{array}{cc}
			0  \\
			0  \\
			\frac{g}{2}\left(h_{i,j+\frac{1}{2}}^{-,\mu}\right)^2-\frac{g}{2}
			\left(h_{i,j+\frac{1}{2}}
			^{-,\mu,*}
			\right)^2  
		\end{array}
		\right),
		\\
		\left(
		\widetilde{\bm{F}}_1
		\right)
		_{i,j-\frac{1}{2}}^{R,\mu}    
		= \dfrac{1}{2}
		\left(
		{\bm{F}}_1
		(
		\bm{U}_{i,j-\frac{1}{2}}^{-,\mu,*}
		)
		+
		{\bm{F}}_1
		(
		\bm{U}_{i,j-\frac{1}{2}}^{+,\mu,*}
		)
		\right)
		+ \left(\begin{array}{cc}
			0  \\
			\frac{g}{2}\left(h_{i-\frac{1}{2},j}^{+,\mu}\right)^2-\frac{g}{2}
			\left(h_{i,j-\frac{1}{2}}
			^{+,\mu,*}
			\right)^2 \\
			0
		\end{array}
		\right),\\
		\left(
		\widetilde{\bm{F}}_2
		\right)
		_{i,j-\frac{1}{2}}^{R,\mu}    
		= \dfrac{1}{2}
		\left(
		{\bm{F}}_1
		(
		\bm{U}_{i,j-\frac{1}{2}}^{-,\mu,*}
		)
		+
		{\bm{F}}_1
		(
		\bm{U}_{i,j-\frac{1}{2}}^{+,\mu,*}
		)
		\right)
		+ \left(\begin{array}{cc}
			0  \\
			0 \\
			\frac{g}{2}\left(h_{i-\frac{1}{2},j}^{+,\mu}\right)^2-\frac{g}{2}
			\left(h_{i,j-\frac{1}{2}}
			^{+,\mu,*}
			\right)^2
		\end{array}
		\right).
	\end{align*}
	Similar arguments show that  
	these numerical fluxes  also satisfy the condition \eqref{eq:2D_condition}.

	\section*{Acknowledgment}
	The works of Z.~Zhang and H.Z.~Tang were partially supported by the National Key R\&D Program of
	China (Project Numbers 2020YFA0712000 \& 2020YFE0204200), and the National Natural Science Foundation of China (Nos.~12171227, 12326314, \& 12288101). The work of K.~Wu was partially supported by Shenzhen Science and Technology Program
	(No. RCJC20221008092757098) and National Natural Science Foundation of China (No.~12171227).

	
	
	\bibliographystyle{siamplain}
	\bibliography{Ref}

\begin{thebibliography}{10}

\bibitem{AuDusse2004fast}
{\sc E.~Audusse, F.~Bouchut, M.-O. Bristeau, R.~Klein, and B.~Perthame}, {\em
  {A fast and stable well-balanced scheme with hydrostatic reconstruction for
  shallow water flows}}, SIAM J. Sci. Comput., 25 (2004), pp.~2050--2065.

\bibitem{Bermudez1994Upwind}
{\sc A.~Bermudez and M.~E. Vazquez}, {\em Upwind methods for hyperbolic
  conservation laws with source terms}, Comput. \& Fluids, 23 (1994),
  pp.~1049--1071.

\bibitem{Borges2008An}
{\sc R.~Borges, M.~Carmona, B.~Costa, and W.~S. Don}, {\em {An improved
  weighted essentially non-oscillatory scheme for hyperbolic conservation
  laws}}, J. Comput. Phys., 227 (2008), pp.~3191--3211.

\bibitem{Brackbill1993An}
{\sc J.~U. Brackbill}, {\em An adaptive grid with directional control}, J.
  Comput. Phys., 108 (1993), pp.~38--50.

\bibitem{Brackbill1982Adaptive}
{\sc J.~U. Brackbill and J.~S. Saltzman}, {\em Adaptive zoning for singular
  problems in two dimensions}, J. Comput. Phys., 46 (1982), pp.~342--368.

\bibitem{CAO1999221}
{\sc W.~Cao, W.~Huang, and R.~D. Russell}, {\em An r-adaptive finite element
  method based upon moving mesh {PDE}s}, J. Comput. Phys., 149 (1999),
  pp.~221--244.

\bibitem{Capilla2013New}
{\sc M.~T. Capilla and A.~Balaguer-Beser}, {\em A new well-balanced
  non-oscillatory central scheme for the shallow water equations on rectangular
  meshes}, J. Comput. Appl. Math., 252 (2013), pp.~62--74.

\bibitem{Castro2009High}
{\sc M.~J. Castro, E.~D. Fern\'{a}ndez-Nieto, A.~M. Ferreiro, J.~A.
  Garc\'{\i}a-Rodr\'{\i}guez, and C.~Par\'{e}s}, {\em High order extensions of
  {R}oe schemes for two-dimensional nonconservative hyperbolic systems}, J.
  Sci. Comput., 39 (2009), pp.~67--114.

\bibitem{CENICEROS2001609}
{\sc H.~D. Ceniceros and T.~Y. Hou}, {\em An efficient dynamically adaptive
  mesh for potentially singular solutions}, J. Comput. Phys., 172 (2001),
  pp.~609--639.

\bibitem{Davis1982}
{\sc S.~F. Davis and J.~E. Flaherty}, {\em An adaptive finite element method
  for initial-boundary value problems for partial differential equations}, SIAM
  J. Sci. Stat. Comput., 3 (1982), pp.~6--27.

\bibitem{DUAN2021109949}
{\sc J.~Duan and H.~Z. Tang}, {\em Entropy stable adaptive moving mesh schemes
  for 2{D} and 3{D} special relativistic hydrodynamics}, J. Comput. Phys., 426
  (2021), p.~109949.

\bibitem{Duan2022High}
{\sc J.~Duan and H.~Z. Tang}, {\em High-order accurate entropy stable adaptive
  moving mesh finite difference schemes for special relativistic
  (magneto)hydrodynamics}, J. Comput. Phys., 456 (2022), p.~111038.

\bibitem{holden2015front}
{\sc H.~Holden and N.~H. Risebro}, {\em {Front Tracking for Hyperbolic
  Conservation Laws}}, Springer, 2015.

\bibitem{Kurganov2018Finite}
{\sc A.~Kurganov}, {\em {Finite-volume schemes for shallow-water equations}},
  Acta Numer., 27 (2018), pp.~289--351.

\bibitem{Alexander2007Asecond}
{\sc A.~Kurganov and G.~Petrova}, {\em {A Second-Order Well-Balanced Positivity
  Preserving Central-Upwind Scheme for the Saint-Venant System}}, Commun. Math.
  Sci., 5 (2007), pp.~133--160.

\bibitem{Leveque1998Balancing}
{\sc R.~J. LeVeque}, {\em Balancing source terms and flux gradients in
  high-resolution {G}odunov methods: the quasi-steady wave-propagation
  algorithm}, J. Comput. Phys., 146 (1998), pp.~346--365.

\bibitem{Li2012Hybrid}
{\sc G.~Li, C.~Lu, and J.~Qiu}, {\em Hybrid well-balanced {WENO} schemes with
  different indicators for shallow water equations}, J. Sci. Comput., 51
  (2012), pp.~527--559.

\bibitem{Li2020High}
{\sc P.~Li, W.~S. Don, and Z.~Gao}, {\em High order well-balanced finite
  difference {WENO} interpolation-based schemes for shallow water equations},
  Comput. \& Fluids, 201 (2020), p.~104476.

\bibitem{Li2022High}
{\sc S.~Li, J.~Duan, and H.~Z. Tang}, {\em High-order accurate entropy stable
  adaptive moving mesh finite difference schemes for (multi-component)
  compressible {E}uler equations with the stiffened equation of state}, Comput.
  Methods Appl. Mech. Engrg., 399 (2022), p.~115311.

\bibitem{Miller1981}
{\sc K.~Miller}, {\em Moving finite elements. {II}}, SIAM J. Numer. Anal., 18
  (1981), pp.~1033--1057.

\bibitem{Noelle2006Well}
{\sc S.~Noelle, N.~Pankratz, G.~Puppo, and J.~R. Natvig}, {\em Well-balanced
  finite volume schemes of arbitrary order of accuracy for shallow water
  flows}, J. Comput. Phys., 213 (2006), pp.~474--499.

\bibitem{Noelle2007High}
{\sc S.~Noelle, Y.~Xing, and C.-W. Shu}, {\em High-order well-balanced finite
  volume {WENO} schemes for shallow water equation with moving water}, J.
  Comput. Phys., 226 (2007), pp.~29--58.

\bibitem{Ren2000An}
{\sc W.~Ren and X.~Wang}, {\em An iterative grid redistribution method for
  singular problems in multiple dimensions}, J. Comput. Phys., 159 (2000),
  pp.~246--273.

\bibitem{Stockie2001}
{\sc J.~M. Stockie, J.~A. Mackenzie, and R.~D. Russell}, {\em A moving mesh
  method for one-dimensional hyperbolic conservation laws}, SIAM J. Sci.
  Comput., 22 (2001), pp.~1791--1813.

\bibitem{Tang2004Solution}
{\sc H.~Z. Tang}, {\em Solution of the shallow-water equations using an
  adaptive moving mesh method}, Int. J. Numer. Meth. Fluids, 44 (2004),
  pp.~789--810.

\bibitem{Tang2003Adaptive}
{\sc H.~Z. Tang and T.~Tang}, {\em Adaptive mesh methods for one- and
  two-dimensional hyperbolic conservation laws}, SIAM J. Numer. Anal., 41
  (2003), pp.~487--515.

\bibitem{vater2019limiter}
{\sc S.~Vater, N.~Beisiegel, and J.~Behrens}, {\em {A limiter-based
  well-balanced discontinuous Galerkin method for shallow-water flows with
  wetting and drying: Triangular grids}}, Int. J. Numer. Meth. Fluids, 91
  (2019), pp.~395--418.

\bibitem{Vukovic2002ENO}
{\sc S.~Vukovic and L.~Sopta}, {\em E{NO} and {WENO} schemes with the exact
  conservation property for one-dimensional shallow water equations}, J.
  Comput. Phys., 179 (2002), pp.~593--621.

\bibitem{Wang2004A}
{\sc D.~Wang and X.~Wang}, {\em A three-dimensional adaptive method based on
  the iterative grid redistribution}, J. Comput. Phys., 199 (2004),
  pp.~423--436.

\bibitem{Winslow1967Numerical}
{\sc A.~M. Winslow}, {\em Numerical solution of the quasilinear {P}oisson
  equation in a nonuniform triangle mesh}, J. Comput. Phys., 1 (1967),
  pp.~149--172.

\bibitem{WuShu2023}
{\sc K.~Wu and C.-W. Shu}, {\em Geometric quasilinearization framework for
  analysis and design of bound-preserving schemes}, SIAM Rev., 65 (2023),
  pp.~1031--1073.

\bibitem{Xing2014Exactly}
{\sc Y.~Xing}, {\em Exactly well-balanced discontinuous {G}alerkin methods for
  the shallow water equations with moving water equilibrium}, J. Comput. Phys.,
  257 (2014), pp.~536--553.

\bibitem{Xing2017Numerical}
{\sc Y.~Xing}, {\em Numerical methods for the nonlinear shallow water
  equations}, in Handbook of numerical methods for hyperbolic problems, vol.~18
  of Handb. Numer. Anal., Elsevier/North-Holland, Amsterdam, 2017,
  pp.~361--384.

\bibitem{Xing2005High}
{\sc Y.~Xing and C.-W. Shu}, {\em {High order finite difference WENO schemes
  with the exact conservation property for the shallow water equations}}, J.
  Comput. Phys., 208 (2005), pp.~206--227.

\bibitem{Xing2006High}
{\sc Y.~Xing and C.-W. Shu}, {\em High order well-balanced finite volume {WENO}
  schemes and discontinuous {G}alerkin methods for a class of hyperbolic
  systems with source terms}, J. Comput. Phys., 214 (2006), pp.~567--598.

\bibitem{Xing2011High}
{\sc Y.~Xing and C.-W. Shu}, {\em {High-order finite volume WENO schemes for
  the shallow water equations with dry states}}, Adv. Water Resour., 34 (2011),
  pp.~1026--1038.

\bibitem{Xing2013Positivity}
{\sc Y.~Xing and X.~Zhang}, {\em Positivity-preserving well-balanced
  discontinuous {G}alerkin methods for the shallow water equations on
  unstructured triangular meshes}, J. Sci. Comput., 57 (2013), pp.~19--41.

\bibitem{Xing2010Positivity}
{\sc Y.~Xing, X.~Zhang, and C.-W. Shu}, {\em {Positivity-preserving high order
  well-balanced discontinuous Galerkin methods for the shallow water
  equations}}, Adv. Water Resour., 33 (2010), pp.~1476--1493.

\bibitem{Xu_freestream}
{\sc D.~Xu, X.~Deng, Y.~Chen, Y.~Dong, and G.~Wang}, {\em On the freestream
  preservation of finite volume method in curvilinear coordinates}, Comput. \&
  Fluids, 129 (2016), pp.~20--32.

\bibitem{Zhang2023Structure}
{\sc J.~Zhang, Y.~Xia, and Y.~Xu}, {\em Structure-preserving finite volume
  arbitrary {L}agrangian-{E}ulerian {WENO} schemes for the shallow water
  equations}, J. Comput. Phys., 473 (2023), p.~111758.

\bibitem{Zhang2021High}
{\sc M.~Zhang, W.~Huang, and J.~Qiu}, {\em {A high-order well-balanced
  positivity-preserving moving mesh DG method for the shallow water equations
  with non-flat bottom topography}}, J. Sci. Comput., 87 (2021), pp.~1--43.

\bibitem{Zhang2022AWell}
{\sc M.~Zhang, W.~Huang, and J.~Qiu}, {\em A well-balanced
  positivity-preserving quasi-{L}agrange moving mesh {DG} method for the
  shallow water equations}, Commun. Comput. Phys., 31 (2022), pp.~94--130.

\bibitem{Zhang2017On}
{\sc X.~Zhang}, {\em {On positivity-preserving high order discontinuous
  Galerkin schemes for compressible Navier–Stokes equations}}, J. Comput.
  Phys., 328 (2017), pp.~301--343.

\bibitem{Zhang2010On}
{\sc X.~Zhang and C.-W. Shu}, {\em On maximum-principle-satisfying high order
  schemes for scalar conservation laws}, J. Comput. Phys., 229 (2010),
  pp.~3091--3120.

\bibitem{Zhang2010OnM}
{\sc X.~Zhang and C.-W. Shu}, {\em On positivity-preserving high order
  discontinuous {G}alerkin schemes for compressible {E}uler equations on
  rectangular meshes}, J. Comput. Phys., 229 (2010), pp.~8918--8934.

\bibitem{Zhang2011Maximum}
{\sc X.~Zhang and C.-W. Shu}, {\em Maximum-principle-satisfying and
  positivity-preserving high-order schemes for conservation laws: survey and
  new developments}, Proc. R. Soc. A, 467 (2011), pp.~2752--2776.

\bibitem{Zhang2023High}
{\sc Z.~Zhang, J.~Duan, and H.~Z. Tang}, {\em High-order accurate well-balanced
  energy stable adaptive moving mesh finite difference schemes for the shallow
  water equations with non-flat bottom topography}, J. Comput. Phys., 492
  (2023), p.~112451.

\bibitem{Zhang2023HighML}
{\sc Z.~Zhang, H.~Z. Tang, and J.~Duan}, {\em High-order accurate well-balanced
  energy stable finite difference schemes for multi-layer shallow water
  equations on fixed and adaptive moving meshes}, arXiv: 2311.08124,  (2023).

\bibitem{Zhao2022Well}
{\sc Z.~Zhao and M.~Zhang}, {\em Well-balanced fifth-order finite difference
  {H}ermite {WENO} scheme for the shallow water equations}, J. Comput. Phys.,
  475 (2023), p.~111860.

\end{thebibliography}
	
\end{document}